\documentclass[11pt, a4paper]{amsart}
\usepackage[width=16.5cm, top=3cm, bottom=3cm]{geometry}                		

\usepackage[english]{babel}
\usepackage{amssymb,stmaryrd,enumerate,mathrsfs,bbm,xcolor,amsthm}

\catcode`\<=\active \def<{
\fontencoding{T1}\selectfont\symbol{60}\fontencoding{\encodingdefault}}
\catcode`\>=\active \def>{
\fontencoding{T1}\selectfont\symbol{62}\fontencoding{\encodingdefault}}
\newcommand{\Iota}{\mathrm{I}}
\newcommand{\assign}{:=}
\newcommand{\asterisk}{\mathord{*}}
\newcommand{\comma}{{,}}
\newcommand{\divides}{\mathrel{|}}
\newcommand{\mathD}{\mathrm{D}}
\newcommand{\mathd}{\mathrm{d}}
\newcommand{\nobracket}{}
\newcommand{\nocomma}{}
\newcommand{\tmcolor}[2]{{\color{#1}{#2}}}
\newcommand{\tmop}[1]{\ensuremath{\operatorname{#1}}}
\newcommand{\tmstrong}[1]{\textbf{#1}}
\newcommand{\tmtextbf}[1]{{\bfseries{#1}}}
\newcommand{\tmtextit}[1]{{\itshape{#1}}}
\newenvironment{enumeratealpha}{\begin{enumerate}[a{\textup{)}}] }{\end{enumerate}}
\newenvironment{enumerateroman}{\begin{enumerate}[i.] }{\end{enumerate}}
\newenvironment{itemizedot}{\begin{itemize} }{\end{itemize}}

\newtheorem{corollary}{Corollary}
\newtheorem{definition}{Definition}
\newtheorem{lemma}{Lemma}
\newtheorem{notation}{Notation}
\newtheorem{proposition}{Proposition}
\newtheorem{remark}{Remark}
\newtheorem{theorem}{Theorem}

\def\Xint#1{\mathchoice
{\XXint\displaystyle\textstyle{#1}}%
{\XXint\textstyle\scriptstyle{#1}}%
{\XXint\scriptstyle\scriptscriptstyle{#1}}%
{\XXint\scriptscriptstyle\scriptscriptstyle{#1}}%
\!\int}
\def\XXint#1#2#3{{\setbox0=\hbox{$#1{#2#3}{\int}$ }
\vcenter{\hbox{$#2#3$ }}\kern-.6\wd0}}

\def\dashint{\Xint-}

\newcommand{\bint}{\dashint}
\newcommand{\dint}{\backslash{\hspace{-.5em}\mathd}}

\newcommand{\VV}{\mathscr{C}}

\newcommand{\CF}{\mathscr{F}}
\newcommand{\CB}{\mathscr{B}}

\newcommand{\CS}{\mathscr{S}}

\newcommand{\CZ}{\mathscr{Z}}

\begin{document}

\title{A variational method for $\Phi^4_3$}

\author{N. Barashkov and M. Gubinelli }
\address{Hausdorff Center for Mathematics \&\\
Institute for Applied Mathematics\\
University of Bonn, Germany}

\date{April 8th, 2019}

\begin{abstract}
  We introduce an explicit description of the $\Phi^4_3$ measure on a bounded
  domain. Our starting point is the interpretation of its Laplace transform as
  the value function of a stochastic optimal control problem along the flow of
  a scale regularization parameter. Once small scale singularities have been
  renormalized by the standard counterterms, $\Gamma$-convergence allows to
  extend the variational characterization to the unregularized model.
\end{abstract}

\keywords{Constructive Euclidean quantum field theory, Bou{\'e}--Dupuis
formula, renormalization group, paracontrolled calculus,
$\Gamma$-convergence.}

{\maketitle}

\section{Introduction }

The $\Phi^4_d$ Gibbs measure on the $d$-dimensional torus $\Lambda = \Lambda_L
=\mathbb{T}_L^d = (\mathbb{R}/ (2 \pi L\mathbb{Z}))^d$ is the probability
measure $\nu$ obtained as the weak limit for $T \rightarrow \infty$ of the
family $(\nu_T)_{T > 0}$ given by
\begin{equation}
  \nu_T (\mathd \phi) = \frac{\exp [- V_T (\phi_T)]}{\CZ_T} \vartheta (\mathd
  \phi) \label{eq:gibbs},
\end{equation}
where
\[ V_T (\varphi) \assign \lambda \int_{\Lambda} (| \varphi (\xi) |^4 - a_T |
   \varphi (\xi) |^2 - b_T) \mathd \xi, \qquad \CZ_T \assign \int e^{- V_T
   (\phi_T)} \vartheta (\mathd \phi) \]
Here $\lambda \geqslant 0$ is a fixed constant, $\Delta$ is the Laplacian on
$\Lambda$, $\vartheta$ is the centered Gaussian measure with covariance $(1 -
\Delta)^{- 1}$, $\CZ_T$ is a normalization factor, $a_T, b_T$ given constants
and $\phi_T = \rho_T \ast \phi$ with $\rho_T$ some appropriate smooth and
compactly supported cutoff function such that $\rho_T \rightarrow \delta$ as
$T \rightarrow \infty$. The measures $\vartheta$ and $\nu_T$ are realized as
probability measures on $\CS' (\Lambda)$, the space of tempered distributions
on $\Lambda$. They are supported on the H{\"o}lder--Besov space $\VV^{(2 - d)
/ 2 - \kappa} (\Lambda)$ for all small $\kappa > 0$. The existence of the
limit $\nu$ is conditioned on the choice of a suitable sequence of
\tmtextit{renormalization constants} $(a_T, b_T)_{T > 0}$. The constant $b_T$
is not necessary, but is useful to decouple the behavior of the numerator from
that of the denominator in eq.~{\eqref{eq:gibbs}}.

\

The aim of this paper is to give a proof of convergence using a variational
formula for the partition function $\CZ_T$ and for the generating function of
the measure $\nu_T$. As a byproduct we obtain also a variational description
for the generating function of the limiting measure $\nu$ via
$\Gamma$-convergence of the variational problem. Let us remark that, to our
knowledge, it is the first time that such explicit description of the
unregulated $\Phi^4_3$ measure is available.

\

Our work can be seen as an alternative realization of
Wilson's~{\cite{wilson_renormalization_1983}} and
Polchinski's~{\cite{polchinski_renormalization_1984}} continuous
renormalization group (RG) method. This method has been made rigorous by
Brydges, Slade et
al.~{\cite{brydges_functional_1993,brydges_short_1995,brydges_lectures_2009}}
and as such witnesses a lot of progress and
successes~{\cite{brydges_renormalisation_2014_3,brydges_renormalisation_2014_2,bauerschmidt_renormalisation_2014,brydges_renormalisation_2014_1,brydges_renormalisation_2014,brydges_short_1995}}.
The key idea is the nonperturbative study of a certain infinite dimensional
Hamilton--Jacobi--Bellman equation~{\cite{brydges_mayer_1987}} describing the
effective, scale dependent, action of the theory. Here we avoid the analysis
involved by the direct study of the PDE by going to the equivalent stochastic
control formulation, well established and understood in finite
dimensions~{\cite{fleming_controlled_2005}}. The time parameter of the
evolution corresponds to an increasing amount of small scale fluctuations of
the Euclidean field and our main tool is a variational representation formula,
introduced by Bou{\'e} and Dupuis~{\cite{boue_variational_1998}}, for the
logarithm of the partition function interpreted as the value function of the
control problem. See also the related papers of
{\"U}st{\"u}nel~{\cite{ustunel_variational_2014}} and
Zhang~{\cite{zhang_variational_2009}} where extensions and further results on
the variational formula are obtained. The variational formula has been used by
Lehec~{\cite{lehec_representation_2013}} to prove some Gaussian functional
inequalities, following the work of Borell~{\cite{borell_diffusion_2000}}. In
this representation we can avoid the analysis of an infinite dimensional
second order operator and concentrate more on pathwise properties of the
Euclidean interacting fields. We are able to leverage techniques developed for
singular SPDEs, in particular the paracontrolled calculus developed
in~{\cite{gubinelli_paracontrolled_2015}}, to perform the renormalization of
various non-linear quantities and show uniform bounds in the $T \rightarrow
\infty$ limit.

\

Define the normalized free energy $\mathcal{W}_T$ for the cutoff $\Phi^4_3$
measure, as
\begin{equation}
  \mathcal{W}_T (f) \assign - \frac{1}{| \Lambda |} \log \int_{\CS' (\Lambda)}
  \exp [- | \Lambda | f (\phi) - V_T (\phi_T)] \vartheta (\mathd \phi)
  \label{eq:free-energy-T}
\end{equation}
where $f \in C \left( \CS' (\Lambda) ; \mathbb{R} \right)$ is a given
function. The main result of the paper is the following

\begin{theorem}
  \label{th:main}Let $d = 3$ and take a small $\kappa > 0$. There exist
  renormalization constants $a_T, b_T$ (which depend polynomially on
  $\lambda$) such that the limit
  \[ \mathcal{W} (f) \assign \lim_{T \rightarrow \infty} \mathcal{W}_T (f), \]
  exists for every $f \in C \left( \VV^{- 1 / 2 - \kappa} ; \mathbb{R}
  \right)$ with linear growth. Moreover the functional $\mathcal{W} (f)$ has
  the variational form
  \[ \mathcal{W} (f) = \inf_{u \in \mathbb{H}_a^{- 1 / 2 - \kappa}} \mathbb{E}
     \left[ f (W_{\infty} + Z_{\infty} (u)) + \Psi_{\infty} (u) + \lambda \|
     Z_{\infty} (u) \|_{L^4}^4 + \frac{1}{2} \| l (u) \|^2_{L^2 ([0, \infty)
     \times \Lambda)} \right] \]
  where
  \begin{itemizedot}
    \item $\mathbb{E}$ denotes expectations on the Wiener space of a
    cylindrical Brownian motion $(X_t)_{t \geqslant 0}$ on $L^2 (\Lambda)$
    with law $\mathbb{P}$;
    
    \item $\mathbb{W}$ a collection of polynomial functions of the Brownian
    motion $(X_t)_{t \geqslant 0}$ comprising a Gaussian process $(W_t)_{t
    \geqslant 0}$ such that $\tmop{Law}_{\mathbb{P}} (W_t) =
    \tmop{Law}_{\vartheta} (\phi_t)$;
    
    \item $\mathbb{H}_a^{- 1 / 2 - \kappa}$ is the space of predictable
    processes (wrt. the Brownian filtration) in $L^2 (\mathbb{R}_+ ; H^{- 1 /
    2 - \kappa})$;
    
    \item $(Z_t (u), l_t (u))_{t \geqslant 0}$ are explicit (non-random)
    functions of $u \in \mathbb{H}_a^{- 1 / 2 - \kappa}$ and $\mathbb{W}$;
    
    \item $\Psi_{\infty} (u)$ a nice polynomial (non-random) functional of
    $(\mathbb{W}, u)$, independent of $f$. 
  \end{itemizedot}
\end{theorem}

See Section~\ref{sec:three-d} and in particular Theorem~\ref{th:main-exact}
for precise definitions of the various objects and a more detailed statement
of this result. With respect to the notations in Lemma~\ref{pointwiseconv},
observe that
\[ f (W_{\infty} + Z_{\infty} (u)) + \Psi_{\infty} (u) = \Phi_{\infty}
   (\mathbb{W}, Z (u), K (u)), \]
where $K (u)$ is another functional of $(\mathbb{W}, u)$.

\

Theorem~\ref{th:main} implies directly the convergence of $(\nu_T)_T$ to a
limit measure $\nu$ on $\CS' (\Lambda)$. Taking $f$ in the linear dual of
$\VV^{- 1 / 2 - \kappa}$ it also gives the following formula for the Laplace
transform of $\nu$:
\begin{equation}
  \int_{\CS' (\Lambda)} \exp (- f (\phi)) \nu (\mathd \phi) = \exp (- |
  \Lambda | (\mathcal{W} (f / | \Lambda |) -\mathcal{W} (0))) .
  \label{eq:laplace}
\end{equation}

To our knowledge this is the first such explicit description (i.e. without
making reference of the limiting procedure). The difficulty is linked to the
conjectured singularity of the $\Phi^4_3$ measure with respect to the
reference Gaussian measure. Another possible approach to an explicit
description goes via integration by parts (IBP) formulas,
see~{\cite{albeverio_remark_2006}} for an early proof and a discussion of this
approach. More recently~{\cite{gubinelli_pde_2018}} gives a self--contained
proof of the IBP formula for any accumulation point of the $\Phi^4_3$ in the
full space. However is still not clear how to use these formulas directly to
obtain uniqueness of the measure and/or other properties (either on the torus
or on the more difficult situation of the full space). Therefore, while our
approach here is limited to the finite volume situation, it could be used to
prove additional results, like large deviations or weak universality very much
like for SPDEs, see
e.g.~{\cite{hairer_large_2016,hairer_large-scale_2018,furlan_weak_2018}}.

\

The parameter $L$, which determines the size of the spatial domain $\Lambda =
\Lambda_L$, will be kept fixed all along the paper and we will not attempt
here to obtain the infinite volume limit $L \rightarrow \infty$. For this
reason we will avoid to explicitly show the dependence of $\mathcal{W}_T$ with
$\Lambda$. However some care will be taken to obtain estimates uniform in the
volume $| \Lambda |$.

\

An easy consequence of the estimates needed to establish the main theorem is
the following corollary (well known in the literature, see
e.g.~{\cite{benfatto_ultraviolet_1980}}):

\begin{corollary}
  \label{corollary:energy-bounds-3d}There exists functions $E_+ (\lambda), E_-
  (\lambda)$ not depending on $| \Lambda |$, such that
  \[ \lim_{\lambda \rightarrow 0 +} \frac{E_{\pm} (\lambda)}{\lambda^3} = 0,
  \]
  and, for any $\lambda > 0$,
  \[ E_- (\lambda) \leqslant \mathcal{W}_T (0) \leqslant E_+ (\lambda) . \]
\end{corollary}

A similar statement for $d = 2$ will be sketched below in order to introduce
some of the ideas on which the $d = 3$ proof is based.

\

The construction of the $\Phi^4_{2, 3}$ measure in finite volume is basic
problem of constructive quantum field theory to which many works have been
devoted, especially in the $d = 2$ case. It is not our aim to provide here a
comprehensive review of this literature. As far as the $d = 3$ case is
concerned, let us just mention some of the results that, to different extent,
prove the existence of the limit as the ultraviolet (small scale)
regularization is removed. After the early work on Glimm and
Jaffe~{\cite{glimm_boson_1968,glimm_positivity_1973}}, in part performed in
the Hamiltonian formalism, all the subsequent research has been formulated in
the Euclidean setting: i.e. as the problem of construction and study of the
probability measures $\nu$ on a space of distributions.
Feldman~{\cite{feldman_lambda_1974}}, Park~{\cite{park_lambda_1975}}, Benfatto
et al.~{\cite{benfatto_ultraviolet_1980}}, Magn{\'e}n and
Seneor~{\cite{magnen_infinite_1976}} and finally Brydges et
al.~{\cite{brydges_new_1983}} obtained the main results we are aware of.
Recent advances in the analysis of singular SPDEs put forward by the invention
of regularity structures by M.~Hairer~{\cite{hairer_theory_2014}} and related
approaches~{\cite{gubinelli_paracontrolled_2015,catellier_paracontrolled_2013,otto_quasilinear_2016}}
or even RG--inspired ones~{\cite{kupiainen_renormalization_2016}}, have
allowed to pursue the stochastic quantization program to a point where now can
be used to prove directly the existence of the finite volume $\Phi^4_3$
measure in two different
ways~{\cite{mourrat_dynamic_2016,albeverio_invariant_2017}}. Uniqueness by
these methods requires additional efforts but seems at reach. Some results on
the existence of the infinite volume measure~{\cite{gubinelli_pde_2018}} and
dynamics~{\cite{gubinelli_global_2018}} have been obtained recently. For an
overview of the status of the constructive program wrt. the analysis of the
$\Phi^4_{2, 3}$ models the reader can consult the introduction
to~{\cite{albeverio_invariant_2017}} and~{\cite{gubinelli_pde_2018}}

\

This paper is organized as follows. In Section~\ref{sec:rg-flow} we set up
our main tool, the Bou{\'e}--Dupuis variational formula of
Theorem~\ref{th:variational}. Then, as a warmup exercise, we use the formula
to show bounds and existence of the $\Phi^4_2$ measure in
Section~\ref{sec:two-d}. We then pass to the more involved situation of three
dimensions in Section~\ref{sec:three-d} where we introduce the renormalized
variational problem. In Section~\ref{sec:bounds} we establish uniform bounds
for this new problem and in Section~\ref{sec:gamma-convergence} we prove
Theorem~\ref{th:main}. Section~\ref{section:analytic} and
Section~\ref{sec:stochastic} are concerned with some details of the analytic
and probabilistic estimates needed throughout the paper.
Appendix~\ref{sec:appendix-para} gather background material on functional
spaces, paraproducts and related functional analytic background material.

\medskip
\textbf{Acknowledgments.}
The authors would like to thank the Isaac Newton Institute
for Mathematical Sciences for support and hospitality during the program SRQ:
Scaling limits, Rough paths, Quantum field theory during which part of the
work on this paper was undertaken. This work was supported by the German DFG
via CRC 1060 and by EPSRC Grant Number EP/R014604/1.
\medskip

\medskip
\textbf{Conventions.}
  Let us fix some notations and objects. Let $\langle a \rangle = (1 + a^2)^{1
  / 2}$. Denote with $\CS (\Lambda)$ the space of Schwartz functions on
  $\Lambda$ and with $\CS' (\Lambda)$ the dual space of tempered
  distributions. The notation $\hat{f}$ or $\CF f$ stands for the space
  Fourier transform of $f$. \ In order to easily keep track of the volume
  dependence of various objects we normalize the Lebesgue measure on $\Lambda$
  to have unit mass. We denote the normalized integral and measure by
  \[ \bint f \assign \frac{1}{| \Lambda |} \int_{\Lambda} f, \quad \dint x =
     \frac{1}{| \Lambda |} \mathd x \]
  where $| \Lambda |$ is the volume of $\Lambda$. Norms in all the related
  functional spaces (Lebesgue, Sobolev and Besov spaces) are understood
  similarly normalized unless stated otherwise. The various constants
  appearing in the estimates will be understood uniform in $| \Lambda |$,
  unless otherwise stated. The constant $\kappa > 0$ represents a small
  positive number which can be different from line to line. The reader is
  referred to the Appendix for an overview of the functional spaces and the
  additional notations used in the paper.

\section{A stochastic control problem}\label{sec:rg-flow}

We begin by constructing a probability space $\mathbb{P}$ endowed with a
process $(Y_t)_{t \in [0, \infty]}$ belonging to $C ([0, \infty], \VV^{(2 - d)
/ 2 - \kappa} (\Lambda))$ and such that $\tmop{Law}_{\vartheta} (\phi_T) =
\tmop{Law}_{\mathbb{P}} (Y_T)$ for all $T \geqslant 0$ and
$\tmop{Law}_{\mathbb{P}} (Y_{\infty}) = \vartheta$.

Fix $\alpha < - d / 2$ and let $\Omega \assign C (\mathbb{R}_+ ; H^{-
\alpha})$, $(X_t)_{t \geqslant 0}$ the canonical process on $\Omega$ and $\CB$
the Borel $\sigma$--algebra of $\Omega$. On $\left( \Omega, \CB \right)$
consider the probability measure $\mathbb{P}$ which makes the canonical
process $X$ a cylindrical Brownian motion in $L^2 (\Lambda)$. In the following
$\mathbb{E}$ without any qualifiers will denote expectations wrt. $\mathbb{P}$
and $\mathbb{E}_{\mathbb{Q}}$ will denote expectations wrt. some other measure
$\mathbb{Q}$. On the measure space $\left( \Omega, \CB, \mathbb{P} \right)$
there exists a collection $(B_t^n)_{n \in (L^{- 1} \mathbb{Z})^d}$ of complex
(2-dimensional) Brownian motions, such that $\overline{B^n_t} = B^{- n}_t$,
$B^n_t, B^m_t$ independent for $m \neq \pm n$ and $X_t = | \Lambda |^{- 1 / 2}
\sum_{n \in (L^{- 1} \mathbb{Z})^d} e^{i \langle n, \cdot \rangle} B^n_t$.
Note that $X$ has a.s. trajectories in $C \left( \mathbb{R}_+, \VV^{- d / 2 -
\varepsilon} (\Lambda) \right)$ for any $\varepsilon > 0$ by standard
arguments.

Fix some $\rho \in C_c^{\infty} (\mathbb{R}_+, \mathbb{R}_+)$ such that $\rho
(0) = 1$. Let $\rho_t (x) \assign \rho (x / t)$ and
\[ \sigma_t (x) \assign (2 \dot{\rho}_t (x) \rho_t (x))^{1 / 2} = (- 2 (x / t)
   \rho (x / t) \rho' (x / t))^{1 / 2} / t^{1 / 2}, \]
where $\dot{\rho}_t$ is the partial derivative of $\rho_t$ with respect to
$t$. Consider the process $(Y_t)_{t \geqslant 0}$ defined by
\begin{equation}
  Y_t \assign \frac{1}{| \Lambda |^{1 / 2}} \sum_{n \in (L^{- 1}
  \mathbb{Z})^d} \int_0^t \frac{\sigma_s (n)}{\langle n \rangle} e^{i \langle
  n, \cdot \rangle} \tmop{dB}^n_s, \qquad t \geqslant 0. \label{eq:rep-Y}
\end{equation}
It is a centered Gaussian process with covariance
\begin{eqnarray*}
  \mathbb{E} [\langle Y_t, \varphi \rangle \langle Y_s, \psi \rangle] & = &
  \frac{1}{| \Lambda |} \sum_{n, m \in (L^{- 1} \mathbb{Z})^d} \mathbb{E}
  \left[ \int_0^t \frac{\sigma_u (n)}{\langle n \rangle} \tmop{dB}^n_u 
  \hat{\varphi} (n) \overline{\int_0^s \frac{\sigma_u (m)}{\langle m \rangle}
  \tmop{dB}^m_s  \hat{\psi} (m)} \right]\\
  & = & \frac{1}{| \Lambda |} \sum_{n \in (L^{- 1} \mathbb{Z})^d}
  \frac{\rho_{\min (s, t)}^2 (n)}{\langle n \rangle^2} \hat{\varphi} (n)
  \overline{\hat{\psi} (n)},
\end{eqnarray*}
for any $\varphi, \psi \in \CS (\Lambda)$ and $t, s \geqslant 0$, by Fubini
theorem and Ito isometry. By dominated convergence $\lim_{t \rightarrow
\infty} \mathbb{E} [\langle Y_t, \varphi \rangle \langle Y_t, \psi \rangle] =
| \Lambda |^{- 1} \sum_{n \in (L^{- 1} \mathbb{Z})^d} \langle n \rangle^{- 2}
\hat{\varphi} (n) \overline{\hat{\psi} (n)}$ for any $\varphi \nocomma, \psi
\in L^2 (\Lambda)$.

Note that up to any finite time $T$ the r.v. $Y_T$ has a bounded spectral
support and the stopped process $Y^T_t = Y_{t \wedge T}$ for any fixed $T >
0$, is in $C (\mathbb{R}_+, W^{k, 2} (\Lambda))$ for any $k \in \mathbb{N}$.
Furthermore $(Y^T_t)_t$ only depends on a finite subset of the Brownian
motions $(B^n)_n$.

We will usually write $g (\mathD)$ for the Fourier multiplier operator with
symbol $g$. With this convention we can compactly denote
\begin{equation}
  Y_t = \int_0^t J_s \mathd X_s, \qquad t \geqslant 0, \label{eq:def-Y}
\end{equation}
where $J_s \assign \langle D \rangle^{- 1} \sigma_s (D)$. We observe that
$Y_t$ has a distribution given by the pushforward $(\rho_t (\mathD))_{\ast}
\vartheta$ of $\vartheta$ through $\rho_t (D)$. We write the measure $\nu_T$
in~{\eqref{eq:gibbs}} in terms of expectations over $\mathbb{P}$ as
\begin{equation}
  \int g (\phi) \nu_T (\mathd \phi) = \frac{\mathbb{E} [g (Y_T) e^{- V_T
  (Y_T)}]}{\CZ_T}, \label{eq:equiv}
\end{equation}
for any bounded measurable $g : \CS' (\Lambda) \rightarrow \mathbb{R}$.

For fixed $T$ the polynomial appearing in the expression for $V_T (Y_T)$ is
bounded below (since $\lambda > 0$) and $\CZ_T$ is well defined and also
bounded away from zero (this follows easily from Jensen's inequality). However
as $T \rightarrow \infty$ we tend to loose both these properties due to the
fact that we will be obliged to take $a_T \rightarrow + \infty$ to renormalize
the non--linear terms. To obtain uniform upper and lower bounds we need a more
detailed analysis and we proceed as follows.

Denote by $\mathbb{H}_a$ the space of progressively measurable processes which
are $\mathbb{P}$--almost surely in $\mathcal{H} \assign L^2 (\mathbb{R}_+
\times \Lambda)$. We say that an element $v$ of $\mathbb{H}_a$ is a
\tmtextit{drift}. Below we will need also drifts belonging to
$\mathcal{H}^{\alpha} \assign L^2 (\mathbb{R}_+ ; H^{\alpha} (\Lambda))$ for
some $\alpha \in \mathbb{R}$, we denote the corresponding space with
$\mathbb{H}_a^{\alpha}$. Consider the measure $\mathbb{Q}_T$ on $\left(
\Omega, \CB \right)$ whose Radon--Nykodim derivative wrt. $\mathbb{P}$ is
given by
\[ \frac{\mathd \mathbb{Q}_T}{\mathd \mathbb{P}} = \frac{e^{- V_T
   (Y_T)}}{\CZ_T} . \]
Since $Y_T$ depends on finitely many Brownian motions $(B^n)_n$ then it is
well known~{\cite{revuz_continuous_2004,follmer_entropy_1985}} that any
$\mathbb{P}$--absolutely continuous probability can be expressed via Girsanov
transform. In particular, by the Brownian martingale representation theorem
there exists a drift $u^T \in \mathbb{H}_a$ such that
\[ \frac{\mathd \mathbb{Q}_T}{\mathd \mathbb{P}} = \exp \left( \int_0^{\infty}
   u_s^T \mathd X_s - \frac{| \Lambda |}{2} \int_0^{\infty} \| u_s^T
   \|^2_{L^2} \mathd s \right), \]
(recall that we normalized the $L^2 (\Lambda)$ norm) and the entropy of
$\mathbb{Q}_T$ wrt.~$\mathbb{P}$ is given by
\[ H (\mathbb{Q}_T |\mathbb{P}) =\mathbb{E}_{\mathbb{Q}_T} \left[ \log
   \frac{\mathd \mathbb{Q}_T}{\mathd \mathbb{P}} \right] = \frac{| \Lambda
   |}{2} \mathbb{E}_{\mathbb{Q}_T} \left[ \int_0^{\infty} \| u_s^T \|^2_{L^2}
   \mathd s \right] . \]
Here equality holds also if one of the two quantities is $+ \infty$. By
Girsanov theorem, the canonical process $X$ is a semimartingale under
$\mathbb{Q}_T$ with decomposition
\[ X_t = \tilde{X}_t + \int_0^t u^T_s \mathd s, \qquad t \geqslant 0, \]
where $(\tilde{X}_t)_t$ is a cylindrical $\mathbb{Q}_T$--Brownian motion in
$L^2 (\Lambda)$. Under $\mathbb{Q}_T$ the process $(Y_t)_t$ has the
semimartingale decomposition $Y_t = W_t + U_t$ with
\[ W_t \assign \int_0^t J_s \mathd X_s, \quad \text{and} \quad U_t = I_t
   (u^T), \]
where for any drift $v \in \mathbb{H}_a$ we define
\[ I_t (v) \assign \int_0^t J_s v_s \mathd s. \]
The integral in the density can be restricted to $[0, T]$ since $u_t^T = 0$ if
$t > T$. Now
\begin{equation}
  - \log \CZ_T = - \log \left[ e^{- V_T (Y_T)} \left( \frac{\mathd
  \mathbb{Q}_T}{\mathd \mathbb{P}} \right)^{- 1} \right] = V_T (Y_T) +
  \int_0^{\infty} u_s^T \mathd X_s - \frac{| \Lambda |}{2} \int_0^{\infty} \|
  u_s^T \|^2 \mathd s, \label{eq:equilib}
\end{equation}
and taking expectation of~{\eqref{eq:equilib}} wrt $\mathbb{Q}_T$ we get
\begin{equation}
  - \log \CZ_T =\mathbb{E}_{\mathbb{Q}_T} \left[ V_T (W_T + I_T (u^T)) +
  \frac{| \Lambda |}{2} \int_0^{\infty} \| u_s^T \|^2 \mathd s \right] .
  \label{eq:pressure}
\end{equation}
For any $v \in \mathbb{H}_a$ define the measure $\mathbb{Q}^v$ by
\[ \frac{\mathd \mathbb{Q}^v}{\mathd \mathbb{P}} = \exp \left( \int_0^{\infty}
   v_s \mathd X_s - \frac{| \Lambda |}{2} \int_0^{\infty} \| v_s \|^2 \mathd s
   \right) . \]
Denote with $\mathbb{H}_c \subseteq \mathbb{H}_a$ the set of drifts $v \in
\mathbb{H}_a$ for which $\mathbb{Q}^v (\Omega) = 1$, in particular $u^T \in
\mathbb{H}_c$. By Jensen's inequality and Girsanov transformation we have
\[ - \log \CZ_T = - \log \mathbb{E}_{\mathbb{P}} [e^{- V_T (Y_T)}] = - \log
   \mathbb{E}^v \left[ e^{- V_T (Y_T) - \int_0^{\infty} v_s \mathd X_s +
   \frac{| \Lambda |}{2} \int_0^{\infty} \| v_s \|^2 \mathd s} \right] \]
\[ \leqslant \mathbb{E}^v \left[ V_T (Y_T) + \int_0^{\infty} v_s \mathd X_s -
   \frac{| \Lambda |}{2} \int_0^{\infty} \| v_s \|^2 \mathd s \right], \]
for all $v \in \mathbb{H}_c$, where $\mathbb{E}^v \assign
\mathbb{E}_{\mathbb{Q}^v}$. We conclude that
\begin{equation}
  - \log \CZ_T \leqslant \mathbb{E}^v \left[ V_T (W^v_T + I_T (v)) + \frac{|
  \Lambda |}{2} \int_0^{\infty} \| v_s \|^2 \mathd s \right],
  \label{eq:pressure-bound}
\end{equation}
where $Y = W^v_T + I_T (v)$ and $\tmop{Law}_{\mathbb{Q}^v} (W^v) =
\tmop{Law}_{\mathbb{P}} (Y)$. The bound is saturated when $v = u^T$. We record
this result in the following lemma which is a precursor of our main tool to
obtain bounds on the partition function and related objects.

\begin{lemma}
  \label{lemma:variational-easy}The following variational formula for the free
  energy holds:
  \[ \mathcal{W}_T (f) = - \frac{1}{| \Lambda |} \log \mathbb{E} [e^{- V_T^f
     (Y_T)}] = \min_{v \in \mathbb{H}_c} \mathbb{E}^v \left[ \frac{1}{|
     \Lambda |} V_T^f (W^v_T + I_T (v)) + \frac{1}{2} \int_0^{\infty} \| v_s
     \|^2_{L^2} \mathd s \right] . \]
  where $V^f_T \assign | \Lambda | f + V_T$.
\end{lemma}

This formula is nice and easy to prove but somewhat inconvenient for certain
manipulations since the space $\mathbb{H}_c$ is indirectly defined and the
reference measure $\mathbb{E}^v$ depends on the drift $v$. A more
straightforward formula has been found by
Bou{\'e}--Dupuis~{\cite{boue_variational_1998}} which involves the fixed
canonical measure $\mathbb{P}$ and a general adapted drift $u \in
\mathbb{H}_a$. This formula will be our main tool in the following.

\begin{theorem}
  \label{th:variational}The Bou{\'e}--Dupuis (BD) variational formula for the
  free energy holds:
  \[ \mathcal{W}_T (f) = - \frac{1}{| \Lambda |} \log \mathbb{E} [e^{- V_T^f
     (Y_T)}] = \inf_{v \in \mathbb{H}_a} \mathbb{E} \left[ \frac{1}{| \Lambda
     |} V_T^f (Y_T + I_T (v)) + \frac{1}{2} \int_0^{\infty} \| v_s \|^2_{L^2}
     \mathd s \right] . \]
  where the expectation is taken wrt to the measure $\mathbb{P}$ on $\Omega$.
\end{theorem}

\begin{proof}
  The original proof can be found in
  Bou{\'e}--Dupuis~{\cite{boue_variational_1998}} for functionals bounded
  above. In our setting the formula can be proved using the result of
  {\"U}st{\"u}nel~{\cite{ustunel_variational_2014}} by observing that $V_T^f
  (Y_T)$ is a \tmtextit{tame} functional, according to his definitions.
  Namely, for some $p, q \geqslant 1$ such that $1 / p + 1 / q = 1$ we have
  \[ \mathbb{E} [| V_T^f (Y_T) |^p] +\mathbb{E} [e^{- q V_T^f (Y_T)}] < +
     \infty . \]
\end{proof}

\begin{remark}
  Some observations on these variational formulas.
  \begin{enumeratealpha}
    \item They originates directly from the variational formula for the free
    energy of a statistical mechanical systems: $V_T^f$ playing the role of
    the internal energy and the quadratic term playing the role of the
    entropy.
    
    \item The infimum could not be attained in Theorem~\ref{th:variational}
    while it is attained in Lemma~\ref{lemma:variational-easy}.
    
    \item The drift generated by absolutely continuous perturbations of the
    Wiener measure has been introduced and studied by
    F{\"o}llmer~{\cite{follmer_entropy_1985}}.
    
    \item They are a non--Markovian and infinite dimensional extension of the
    well known stochastic control problem representation of the
    Hamilton--Jacobi--Bellman equation in finite
    dimensions~{\cite{fleming_controlled_2005}}.
    
    \item The BD formula is easier to use than the formula in
    Lemma~\ref{lemma:variational-easy} since the probability do not depend on
    the drift $v$. Going from one formulation to the other requires proving
    that certain SDEs with functional drift admits strong solutions and that
    one is able to approximate unbounded functionals $V_T$ by bounded ones.
    See {\"U}st{\"u}nel~{\cite{ustunel_variational_2014}} and
    Lehec~{\cite{lehec_representation_2013}} for a streamlined proof of the BD
    formula and for applications of the formula to functional inequalities on
    Gaussian measures. For example, from this formula is not difficult to
    prove integrability of functionals which are Lipshitz in the
    Cameron--Martin directions.
  \end{enumeratealpha}
\end{remark}

The next lemma provides a deterministic regularity result for $I (v)$ which
will be useful below. It says that the drift $v$ generates shifts of the
Gaussian free field in directions which belong to $H^1$ uniformly in the scale
parameter up to $\infty$. The space $H^1$ is the Cameron--Martin space of the
free field~{\cite{janson_gaussian_1997}}.

\begin{lemma}
  \label{lemma:impr-reg-drift}Let $\alpha \in \mathbb{R}$. For any $v \in L^2
  ([0, \infty), H^{\alpha})$ we have
  \[ \sup_{0 \leqslant t \leqslant T} \| I_t (v) \|_{H^{\alpha + 1}}^2 +
     \sup_{0 \leqslant s < t \leqslant T} \frac{\| I_t (v) - I_s (v)
     \|_{H^{\alpha + 1}}^2}{1 \wedge (t - s)} \lesssim \int_0^T \| v_r
     \|^2_{H^{\alpha}} \mathd r. \]
\end{lemma}

\begin{proof}
  Using the fact that $\sigma_s (\mathD)$ is diagonal in Fourier space, and
  denoting with $(e_k)_{k \in \mathbb{Z}^d}$ the basis of trigonometric
  polynomials, we have
  \[ \begin{array}{lll}
       \left\| \int_r^t \sigma_s (\mathD) v_s \mathd s \right\|^2_{H^{\alpha}}
       & = & \frac{1}{| \Lambda |} \sum_k \langle k \rangle^{2 \alpha} \left|
       \int_r^t \langle \sigma_s (\mathD) e_k, v_s \rangle \mathd s
       \right|^2\\
       & \leqslant & \frac{1}{| \Lambda |} \sum_k \langle k \rangle^{2
       \alpha} \left( \int_r^t | \langle \sigma_s (\mathD) e_k, e_k \rangle
       |^2 \mathd s \right) \left( \int_r^t | \langle e_k, v_s \rangle |^2
       \mathd s \right)\\
       & \leqslant & \int_r^t \| v_s \|^2_{H^{\alpha}} \mathd s \sup_k
       \int_r^t \langle e_k, \sigma_s (\mathD)^2 e_k \rangle \mathd s\\
       & \leqslant & \int_r^t \| v_s \|^2_{H^{\alpha}} \mathd s \sup_k
       \langle e_k, \rho_t^2 (\mathD) e_k \rangle \leqslant \int_0^T \| v_s
       \|^2_{H^{\alpha}} \mathd s.
     \end{array} \]
  On the other hand $\sigma_s (D)$ is a smooth Fourier multiplier and using
  Proposition~\ref{multiplierestimate} we have the estimate $\| \sigma_s
  (\mathD) f \|_{H^{\alpha}} \lesssim \| f \|_{H^{\alpha}} / \langle s
  \rangle^{1 / 2}$ uniformly in $s \geqslant 0$, therefore, for all $0
  \leqslant r \leqslant t \leqslant T$, we have
  \[ \begin{array}{lll}
       \left\| \int_r^t \sigma_s (\mathD) v_s \mathd s \right\|^2_{H^{\alpha}}
       & \leqslant & \left( \int_r^t \| \sigma_s (\mathD) v_s \|_{H^{\alpha}}
       \mathd s \right)^2 \leqslant (t - r) \int_r^t \| \sigma_s (\mathD) v_s
       \|^2_{H^{\alpha}} \mathd s\\
       & \lesssim & (t - r) \left( \int_0^T \| v_s \|^2_{H^{\alpha}} \mathd s
       \right) .
     \end{array} \]
  We conclude that
  \[ \| I_t (v) - I_r (v) \|_{H^{\alpha + 1}}^2 \lesssim \left\| \int_r^t
     \sigma_s (\mathD) v_s \mathd s \right\|^2_{H^{\alpha}} \leqslant [1
     \wedge (t - r)] \int_0^T \| v_s \|^2_{H^{\alpha}} \mathd s. \]
  
\end{proof}

\begin{notation}
  In the estimates below the symbol $E (\lambda)$ will denote a generic
  positive deterministic quantity, not depending on $| \Lambda |$ and such
  that $E (\lambda) / \lambda^3 \rightarrow 0$ as $\lambda \rightarrow 0$.
  Moreover the symbol $Q_T$ will denote a generic random variable measurable
  wrt. $\sigma ((W_t)_{t \in [0, T]})$ and belonging to $L^1 (\mathbb{P})$
  uniformly in $T$ and $| \Lambda |$.
\end{notation}

\section{Two dimensions\label{sec:two-d}}

As a warm up we consider here the case $d = 2$ setting $f = 0$ for simplicity.
From Theorem~\ref{th:variational} we see that the relevant quantity to bound
is of the form
\begin{equation}
  F_T (u) \assign \mathbb{E} \left[ \frac{1}{| \Lambda |} V_T (W_T + I_T (u))
  + \frac{1}{2} \| u \|_{\mathcal{H}}^2 \right], \label{eq:var-energy}
\end{equation}
for $u \in \mathbb{H}_a$. Let $Z_t = I_t (u)$ and write $W_t = Y_t$ as a
mnemonic of the fact that under $\mathbb{P}$ the process $W$ is a martingale.
From now on we leave implicit the integration variable over the spatial domain
$\Lambda$. Choosing
\begin{equation}
  a_T = 6\mathbb{E} [W_T (0)^2], \qquad b_T = 3\mathbb{E} [W_T (0)^2]^2,
  \label{eq:renorm-const-2d}
\end{equation}
we have
\[ V_T (W_T + Z_T) = \lambda \bint \llbracket W_T^4 \rrbracket + 4 \lambda
   \bint \llbracket W_T^3 \rrbracket Z_T + 6 \lambda \bint \llbracket W_T^2
   \rrbracket Z_T^2 + 4 \lambda \bint W_T Z_T^3 + \lambda \bint Z_T^4, \]
where
\[ \begin{array}{lll}
     \llbracket W_T^4 \rrbracket & \assign & W_T^4 - 6\mathbb{E} [W_T^2] W_T^2
     + 3\mathbb{E} [W_T^2]^2,\\
     \llbracket W_T^3 \rrbracket & \assign & W_T^3 - 3\mathbb{E} [W_T^2]
     W_T,\\
     \llbracket W_T^2 \rrbracket & \assign & W_T^2 -\mathbb{E} [W_T^2],
   \end{array} \]
denote the Wick powers of the Gaussian r.v.
$W_T$~{\cite{janson_gaussian_1997}}. These polynomials, when seen as
stochastic processes in $T$, are $\mathbb{P}$--martingales wrt. the filtration
of $(W_t)_t$. In particular they have an expression as iterated stochastic
integrals wrt. the Brownian motions $(B^n_t)_{t, n}$ introduced in
eq.~{\eqref{eq:rep-Y}}. Using Theorem~\ref{th:variational} with $u = 0$ we
readily have an upper bound for the free energy:
\[ - \frac{1}{| \Lambda |} \log \CZ_T \leqslant \lambda \mathbb{E} \left[
   \bint \llbracket W_T^4 \rrbracket \right] = 0. \]
For a lower bound we need to estimate from below the average under
$\mathbb{P}$ of the variational expression
\[ \lambda \bint \llbracket W_T^4 \rrbracket + 4 \lambda \bint \llbracket
   W_T^3 \rrbracket Z_T + 6 \lambda \bint \llbracket W_T^2 \rrbracket Z_T^2 +
   4 \lambda \bint W_T Z_T^3 + \lambda \bint Z_T^4 + \frac{1}{2} \| u
   \|_{\mathcal{H}}^2 . \]
The strategy we adopt is to bound \tmtextit{pathwise}, and for a generic drift
$u$, the contributions
\[ \Phi_T (Z) \assign \underbrace{4 \lambda \bint \llbracket W_T^3 \rrbracket
   Z_T}_{\Iota} + \underbrace{6 \lambda \bint \llbracket W_T^2 \rrbracket
   Z_T^2}_{\Iota \Iota} + \underbrace{4 \lambda \bint W_T Z_T^3}_{\Iota \Iota
   \Iota}, \]
in term of quantities involving only the Wick powers of $W$ which we can
control in expectation and the last two \tmtextit{positive terms}
\[ \frac{1}{2} \| u \|_{\mathcal{H}}^2 + \lambda \bint Z_T^4 . \]
Any residual positive contribution depending on $u$ can be dropped in the
lower bound making the dependence on the drift disappear. To control term
$\Iota$ we see that by duality and Young's inequality, for any $\delta > 0$,
\begin{equation}
  \left| 4 \lambda \bint \llbracket W_T^3 \rrbracket Z_T \right| \leqslant 4
  \lambda \| \llbracket W_T^3 \rrbracket \|_{H^{- 1}} \| Z_T \|_{H^1}
  \leqslant C (\delta, d) \lambda^2 \| \llbracket W_T^3 \rrbracket \|_{H^{-
  1}}^2 + \delta \int_0^T \| u_s \|^2_{L^2} \mathd s. \label{eq:term-i}
\end{equation}
For the term $\Iota \Iota$ the following fractional Leibniz rule is of help:

\begin{proposition}
  \label{FractionalLeibniz}Let $1 < p < \infty$ and $p_1, p_2, p_1', p_2' > 1$
  such that $\frac{1}{p_1} + \frac{1}{p_2} = \frac{1}{p_1'} + \frac{1}{p_2'} =
  \frac{1}{p}$. Then for every $s, \alpha \geqslant 0$ there exists a constant
  C such that
  \[ \| \langle \mathD \rangle^s (f g) \|_{L^p} \leq \| \langle \mathD
     \rangle^{s + \alpha} f \|_{L^{p_2}}  \| \langle \mathD \rangle^{- \alpha}
     g \|_{L^{p_1}} + \| \langle \mathD \rangle^{s + \alpha} g \|_{L^{p_1'}} 
     \| \langle \mathD \rangle^{- \alpha} f \|_{L^{p_2'}} . \]
\end{proposition}

\begin{proof}
  See~{\cite{gulisashvili_exact_1996}}.
\end{proof}

Using Proposition~\ref{FractionalLeibniz} we get, for any $\delta > 0$,
\begin{equation}
  \begin{array}{lll}
    \left| 6 \lambda \bint \llbracket W_T^2 \rrbracket Z_T^2  \right| &
    \lesssim & \lambda \| \llbracket W_T^2 \rrbracket \|_{W^{- \varepsilon,
    5}} \| Z_T^2 \|_{W^{\varepsilon, \frac{5}{4}}}\\
    & \lesssim & \lambda \| \llbracket W_T^2 \rrbracket \|_{W^{- \varepsilon,
    5}} \| Z_T \|_{W^{\varepsilon, 2}} \| Z_T \|_{L^{\frac{10}{3}}}\\
    & \lesssim & \lambda \| \llbracket W_T^2 \rrbracket \|_{W^{- \varepsilon,
    5}} \| Z_T \|_{W^{1, 2}} \| Z_T \|_{L^4}\\
    & \leqslant & \frac{C^2 \lambda^3}{2 \delta} \| \llbracket W_T^2
    \rrbracket \|^4_{W^{- \varepsilon, 5}} + \frac{\delta}{4} \| Z_T
    \|^2_{W^{1, 2}} + \frac{\delta \lambda}{4} \| Z_T \|_{L^4}^4 .
  \end{array} \label{eq:term-ii}
\end{equation}

In order to bound the term $\Iota \Iota \Iota$ we observe the following:

\begin{lemma}
  \label{lemma:bound-cubic}Let $f \in W^{- 1 / 2 - \varepsilon, p}$ for every
  $1 \leqslant p < \infty$ and $\varepsilon > 0$. We claim that for any
  $\delta > 0$ there exists a constant $C = C (\delta, d)$, and an exponent $K
  < \infty$ such that for any $g \in W^{1, 2}$
  \[ \lambda \left| \bint f g^3 \right| \leqslant E (\lambda) \| f \|^K_{W^{-
     1 / 2 - \varepsilon, p} } + \delta (\| g \|^2_{W^{1 - \varepsilon, 2}} +
     \lambda \| g \|^4_{L^4 }) . \]
\end{lemma}

\begin{proof}
  By duality $\left| \bint f g^3 \right| \leqslant \| f \|_{W^{- 1 / 2 -
  \varepsilon, p}} \| g^3 \|_{W^{1 / 2 + \varepsilon, p'}}$. Applying again
  Proposition \ref{FractionalLeibniz} and Proposition
  \ref{gagliardo-nirenberg}, we get
  \[ \begin{array}{lll}
       \| g^3 \|_{W^{1 / 2 + \varepsilon, 14 / 13}} & \lesssim & \| \langle
       \mathD \rangle^{1 / 2 + \delta} g^3 \|_{L^{14 / 13}} \lesssim \|
       \langle \mathD \rangle^{5 / 8} g\|_{L^{14 / 6}} \| g \|_{L^4}^2\\
       & \lesssim & \|g\|^{5 / 7}_{H^{7 / 8}} \| g \|_{L^4}^{17 / 7} .
     \end{array} \]
  So
  \[ \begin{array}{lll}
       \lambda \left| \bint f g^3 \right| & \lesssim & \lambda \| f \|_{W^{- 1
       / 2 - \varepsilon, 14}} \|g\|^{5 / 7}_{H^{7 / 8}} \| g \|_{L^4}^{17 /
       7}\\
       & \lesssim & \lambda^{11} \| f \|^{28}_{W^{- 1 / 2 - \varepsilon, 14}}
       + \delta (\|g\|^2_{H^{7 / 8}} + \lambda \| g \|_{L^4}^4) .
     \end{array} \]
  
\end{proof}

\begin{remark}
  For the $d = 2$ case it would have been enough to estimate $\| g^3
  \|_{W^{\varepsilon, p}}$. The stronger estimate will be useful below for $d
  = 3$ since there we will only have $W_T \in W^{- 1 / 2 - \varepsilon, p}$
  for any large $p$.
\end{remark}

Then
\begin{equation}
  \left| 4 \lambda \bint W_T Z_T^3 \right| \leqslant C E (\lambda) \| W_T
  \|^K_{W^{- 1 / 2 - \varepsilon, p} } + \delta (\| Z_T \|^2_{W^{1 -
  \varepsilon, 2}} + \lambda \| Z_T \|^4_{L^4 }) . \label{eq:term-iii}
\end{equation}
Using eqs.~{\eqref{eq:term-i}},~{\eqref{eq:term-ii}} and~{\eqref{eq:term-iii}}
we obtain, for $\delta$ small enough,
\begin{equation}
  | \Phi_T (Z) | \leqslant Q_T + \delta \left[ \frac{1}{2} \| u
  \|_{\mathcal{H}}^2 + \lambda \bint Z_T^4 \right], \label{eq:bound-Psi}
\end{equation}
where
\[ Q_T = O (\lambda^2) [1 + \| \llbracket W_T^3 \rrbracket \|_{H^{- 1}}^2 + \|
   \llbracket W_T^2 \rrbracket \|^4_{W^{- \varepsilon, 5}} + \| W_T \|^K_{W^{-
   1 / 2 - \varepsilon, p} }] . \]
Therefore
\[ F_T (u) \geqslant -\mathbb{E} [Q_T] + (1 - \delta) \left[ \frac{1}{2} \| u
   \|_{\mathcal{H}}^2 + \lambda \bint Z_T^4 \right] \geqslant -\mathbb{E}
   [Q_T] . \]
This last average do not depends anymore on the drift and we are only left to
show that
\[ \sup_T \mathbb{E} [Q_T] < \infty . \]
However, it is well known that the Wick powers of the two dimensional Gaussian
free field are distributions belonging to $L^a (\Omega, W^{- \varepsilon, b})$
for any $a \geqslant 1$ and $b \geqslant 1$ and hypercontractivity plus an
easy argument gives the uniform boundedness of the above averages, see
e.g.~{\cite{mourrat_construction_2016}}. We have established:

\begin{theorem}
  \label{th:twod-bound}For any $\lambda > 0$ we have
  \[ \sup_T \frac{1}{| \Lambda |} \left| \log \CZ_T \right| \lesssim O
     (\lambda^2), \]
  where the constant in the r.h.s. is independent of $\Lambda$. 
\end{theorem}

\begin{remark}
  Observe that the argument above remains valid upon replacing $\lambda$ with
  $\lambda p$ with $p \geqslant 1$. This implies that $e^{- V_T (Y_T)}$ is in
  all the $L^p$ spaces wrt. the measure $\mathbb{P}$ uniformly in $T$ and for
  any $p \geqslant 1$.
\end{remark}

\section{Three dimensions\label{sec:three-d}}

In three dimensions the strategy we used in two dimensions fails. Indeed here
the Wick products are less regular: $\llbracket W_T^2 \rrbracket \in \VV^{- 1
- \kappa}$ uniformly in $T$ for any small $\kappa > 0$ and $\llbracket W_T^3
\rrbracket$ does not even converge to a well-defined random distribution. This
implies that there is no straightforward approach to control the terms
\begin{equation}
  \bint \llbracket W_T^3 \rrbracket Z_T, \qquad \text{and} \qquad \bint
  \llbracket W_T^2 \rrbracket Z_T^2, \label{eq:mixed-terms}
\end{equation}
like we did in Section~\ref{sec:two-d}. The only apriori estimate on the
regularity of $Z_T = I_T (u)$ is in $H^1$, coming from
Lemma~\ref{lemma:impr-reg-drift} and the quadratic term in the variational
functional $F_T (u)$. It is also well known that in three dimensions there are
further divergences beyond the Wick ordering which have to be subtracted in
order for the limiting measure to be non-trivial. For these reasons we
introduce in the energy $V_T$ further scale dependent renormalization
constants $\gamma_T, \delta_T$ \ beyond Wick ordering to have
\begin{equation}
  \frac{1}{| \Lambda |} V_T^f (Y_T) = f (Y_T) + \bint (\lambda \llbracket
  Y_T^4 \rrbracket - \lambda^2 \gamma_T \llbracket Y_T^2 \rrbracket -
  \delta_T) . \label{eq:additional-renorm}
\end{equation}
Repeating the computation from Section~\ref{sec:two-d} we arrive at
\begin{equation}
  \begin{array}{lll}
    F_T (u) & = & \mathbb{E} \left[ f (W_T + Z_T) + \lambda \bint
    \mathbb{W}_T^3 Z_T + \frac{\lambda}{2} \bint \mathbb{W}^2_T Z_T^2 + 4
    \lambda \bint W_T Z_T^3 \right]\\
    &  & -\mathbb{E} \left[ 2 \lambda^2 \gamma_T \bint W_T Z_T + \lambda^2
    \gamma_T \bint Z_T^2 + \lambda^2 \delta_T \right] +\mathbb{E} \left[
    \lambda \bint Z_T^4 + \frac{1}{2} \| u \|^2_{\mathcal{H}} \right] .
  \end{array} \label{eq:3d-FT}
\end{equation}
where we introduced the convenient notations
\[ \mathbb{W}_t^3 \assign 4 \llbracket W_t^3 \rrbracket, \qquad \mathbb{W}^2_t
   \assign 12 \llbracket W_t^2 \rrbracket, \qquad t \geqslant 0, \]
and we recall that $f$ is a fixed function belonging to $C \left( \VV^{- 1 / 2
- \kappa} ; \mathbb{R} \right)$ with linear growth.

This form of the functional is not very useful in the limit $T \rightarrow
\infty$ since some of the terms, taken individually, are not expected to
behave well. We will perform a change of variables in the variational
functional in order to obtain some explicit cancellations which will leave
well behaved quantities of $T$. The main drawback is that the functional will
have a less compact and canonical form.

Some care has to be taken in order for the resulting quantities to be still
controlled by the coercive terms. We will need some regularization which will
make compatible Fourier cutoffs with $L^4$ estimates. To introduce such a
regularization fix a smooth function $\theta : \mathbb{R}_+ \rightarrow
\mathbb{R}_+$ such that $\theta (\xi) = 1$ if $\xi \leqslant 1 / 4$ and
$\theta (\xi) = 0$ if $\xi \geqslant 1 / 2$. Then \
\begin{equation}
  \begin{array}{lll}
    \theta_t (\xi) \sigma_s (\xi) & = & \text{$0$ for $s \geqslant t$,}\\
    \theta_t (\xi) & = & \text{$1$ for $\xi \leqslant c t$ \text{for some} $c
    > 0$.}
  \end{array} \label{theta}
\end{equation}
By the Mihlin-H{\"o}rmander theorem we deduce that the operator $\theta_t =
\theta_t (\mathD)$ is bounded on $L^p$ for any $1 < p < \infty$, see
Proposition~\ref{multiplierestimate}. In the following, for any $f \in C
\left( [0, \infty], \CS' (\Lambda) \right)$ we define $f^{\flat}_t \assign
\theta_t f_t$ then \
\[ Z_t^{\flat} = \theta_t Z_t = \int_0^t \theta_t \langle \mathD \rangle^{-
   1} \sigma_s (\mathD) u_s \mathd s = \int_0^T \theta_t \langle \mathD
   \rangle^{- 1} \sigma_s (\mathD) u_s \mathd s = \theta_t Z_T . \]
In this way we have $\| Z_t^{\flat} \|_{L^p} \lesssim \| Z_T \|_{L^p}$ for all
$t \leqslant T$. The renormalized functional will depend on some specific
renormalized combinations of the martingales $(\llbracket \mathbb{W}^k_t
\rrbracket)_{t, k}$. Therefore it will be also convenient to introduce a
collective notation for all the stochastic objects appearing in the
functionals and specify the topologies in which they are expected to be well
behaved. Let
\[ \mathbb{W} \assign (\mathbb{W}^1_{}, \mathbb{W}^2_{}, \mathbb{W}^{\langle 3
   \rangle}_{}, \mathbb{W}^{[3] \circ 1}, \mathbb{W}^{2 \diamond [3]},
   \mathbb{W}^{\langle 2 \rangle \diamond \langle 2 \rangle}_{}), \]
with $\mathbb{W}^1 \assign W$,
\[ \mathbb{W}^{\langle 3 \rangle}_t \assign J_t \mathbb{W}^3, \quad
   \mathbb{W}_t^{[3]} \assign \int_0^t J_s \mathbb{W}_s^{\langle 3 \rangle}
   \mathd s, \quad \mathbb{W}_t^{[3] \circ 1} \assign \mathbb{W}^1_t \circ
   \mathbb{W}_t^{[3]}, \]
\[ \quad \mathbb{W}^{2 \diamond [3]}_t \assign \mathbb{W}^2_t \circ
   \mathbb{W}_t^{[3]} + 2 \gamma_t \mathbb{W}^1_t, \quad \mathbb{W}^{\langle 2
   \rangle \diamond \langle 2 \rangle}_t \assign (J_t \mathbb{W}^2_t) \circ
   (J_t \mathbb{W}^2_t) + 2 \dot{\gamma}_t . \]
We not not need to include $\mathbb{W}^{[3]}$ since it can be obtained as a
function of $\mathbb{W}^{\langle 3 \rangle}$ thanks to the bound
\[ \|\mathbb{W}^{[3]}_t -\mathbb{W}^{[3]}_s \|_{\VV^{1 / 2 - 2 \kappa}}
   \leqslant \int_s^t \|J_r \mathbb{W}^{\langle 3 \rangle}_r \|_{\VV^{1 / 2 -
   2 \kappa}} \mathd r \leqslant \left[ \int_0^T \|J_r \mathbb{W}^{\langle 3
   \rangle}_r \|_{\VV^{1 / 2 - 2 \kappa}}^2 \mathd r \right]^{1 / 2} | t - s
   |^{1 / 2} \]
\[ \leqslant \left[ \int_0^T \|\mathbb{W}^{\langle 3 \rangle}_r \|_{\VV^{- 1 /
   2 - \kappa}}^2 \frac{\mathd r}{\langle r \rangle^{1 + 2 \kappa}} \right]^{1
   / 2} | t - s |^{1 / 2} \lesssim \sup_{r \in [0, T]} \|\mathbb{W}^{\langle 3
   \rangle}_r \|_{\VV^{- 1 / 2 - \kappa}}^2 | t - s |^{1 / 2}, \]
valid for all $0 \leqslant s \leqslant t \leqslant T$ which shows that the
deterministic map $\mathbb{W}^{\langle 3 \rangle} \mapsto \mathbb{W}^{[3]}$ is
continuous from $C \left( [0, \infty], \VV^{- 1 / 2 - \kappa} \right)$ to
$C^{1 / 2} \left( [0, \infty], \VV^{1 / 2 - 2 \kappa} \right)$. The pathwise
regularity of all the other stochastic objects follows from the next Lemma,
provided the function $\gamma$ is chosen appropriately.

\begin{lemma}
  \label{lemma:stoch-reg}There exists a function $\gamma_t \in C^1
  (\mathbb{R}_+, \mathbb{R})$ such that
  \begin{equation}
    | \gamma_t | + \langle t \rangle | \dot{\gamma}_t | \lesssim \log \langle
    t \rangle, \qquad t \geqslant 0. \label{eq:bounds-gamma-main}
  \end{equation}
  and such that the vector $\mathbb{W}$ is almost surely in $\mathfrak{S}$
  where $\mathfrak{S}$ is the Banach space
  \[ \mathfrak{S}= C ([0, \infty], \mathfrak{W}) \cap \left\{
     \mathbb{W}^{\langle 3 \rangle}_{} \in L^2 \left( \mathbb{R}_+, \VV^{- 1 /
     2 - \kappa} \right), \mathbb{W}^{\langle 2 \rangle \diamond \langle 2
     \rangle}_{} \in L^1 \left( \mathbb{R}_+, \VV^{- \kappa} \right) \right\}
  \]
  with
  \[ \mathfrak{W}=\mathfrak{W}_{\kappa} \assign \VV^{- 1 / 2 - \kappa} \times
     \VV^{- 1 - \kappa} \times \VV^{- 1 - \kappa} \times \VV^{- \kappa} \times
     \VV^{- 1 / 2 - \kappa} \times \VV^{- \kappa}, \]
  and equipped with the norm
  \[ \| \mathbb{W} \|_{\mathfrak{S}} \assign \| \mathbb{W} \|_{C ([0,
     \infty], \mathfrak{W})} + \|\mathbb{W}^{\langle 3 \rangle}_{} \|_{L^2
     \left( \mathbb{R}_+, \VV^{- 1 / 2 - \kappa} \right)} +
     \|\mathbb{W}^{\langle 2 \rangle \diamond \langle 2 \rangle}_{} \|_{L^1
     \left( \mathbb{R}_+, \VV^{- \kappa} \right)} . \]
  The norm $\| \mathbb{W} \|_{\mathfrak{S}}$ belongs to all $L^p$ spaces.
  Moreover the averages of the Besov norms $B^{\alpha}_{p, r}$ of the
  components of $\mathbb{W}$ of regularity $\alpha$ are uniformly bounded in
  the volume $| \Lambda |$ if $p < \infty$.
\end{lemma}

\begin{proof}
  The proof is based on the observation that one can choose $\gamma$ in such a
  way that every component $\mathbb{W}^{(i)}$ of the vector $\mathbb{W}$ is
  such that $(\Delta_q \mathbb{W}_t^{(i)} (x))_{t \geqslant 0}$ for $q
  \geqslant - 1$ and $x \in \Lambda$ is a martingale wrt. the Brownian
  filtration. The reader can find details for this statement in the proof of
  Lemma~\ref{lemma:renormalized-terms} below, in particular
  eq.~{\eqref{eq:W2o3}} and~{\eqref{eq:W2o2}} give the stochastic integral
  representations of the most difficult terms which require renormalization by
  $\gamma$. The quadratic variation of these martingales can be controlled
  uniformly up to $t = \infty$. As a consequence, there exists a version of
  $(\Delta_q \mathbb{W}_t^{(i)} (x))_{t \geqslant 0}$ which is continuous in
  $t \in [0, \infty]$. Since each Littlewood--Paley block is a smooth function
  of $x$ is not difficult from this to deduce that $((t, x) \mapsto \Delta_q
  \mathbb{W}_t^{(i)} (x)) \in C ([0, \infty] \times \Lambda ; \mathbb{R})$. By
  Burkholder--David--Gundy (BDG) estimates and $L^p$ estimates on the the
  quadratic variations one can conclude that $\| \mathbb{W} \|_{\mathfrak{S}}$
  is in all $L^p$ spaces provided we can control the appropriate moments of
  the norm of $\mathbb{W}_T$ uniformly in $T$. This is achieved in
  Lemma~\ref{lemma:renormalized-terms} for the more difficult resonant
  products where we show also the claimed uniformly of the finite $r$ Besov
  norms $B^{\alpha}_{p, r}$.
\end{proof}

For convenience of the reader we summarize the probabilistic estimates in
Table~\ref{table:reg}.

\begin{table}[h]
  \begin{tabular}{lllllll}
    $\mathbb{W}^1_{}$ & $\mathbb{W}^2_{}$ & $\mathbb{W}^{\langle 3
    \rangle}_{}$ & $\mathbb{W}^{[3]}$ & $\mathbb{W}^{[3] \circ 1}$ &
    $\mathbb{W}^{2 \diamond [3]}$ & $\mathbb{W}^{\langle 2 \rangle \diamond
    \langle 2 \rangle}_{}$\\
    \hline
    $C \VV^{- 1 / 2 -}$ & $C \VV^{- 1 -}$ & $C \VV^{- 1 / 2 -} \cap L^2 \VV^{-
    1 / 2 -}$ & $C \VV^{1 / 2 -}$ & $C \VV^{0 -}$ & $C \VV^{- 1 / 2 -}$ & $C
    \VV^{0 -}{\cap}L^1 \VV^{0 -}$
  \end{tabular}
  \caption{\label{table:reg}Regularities of the various stochastic objects,
  the domain of the time variable is understood to be $[0, \infty]$. Estimates
  in these norms holds a.s. and in $L^p (\mathbb{P})$ for all $p \geqslant 1$
  (see Lemma~\ref{lemma:stoch-reg}).}
\end{table}

\begin{remark}
  The requirement that $\mathbb{W}^{\langle 3 \rangle} \in L^2 \VV^{- 1 / 2
  -}$ will be used in Section \ref{sec:gamma-convergence} to establish
  equicoercivity and to relax the variational problem to a suitable space of
  measures. 
\end{remark}

We are now ready to perform a change of variables which renormalizes the
variational functional.

\begin{lemma}
  \label{lemma:change-of-variables}Define $l = l^T (u) \in \mathbb{H}_a$, $Z =
  Z (u) \in C ([0, \infty], H^{1 / 2 - \kappa})$, $K = K (u) \in C ([0,
  \infty], H^{1 - \kappa})$ such that
  \begin{equation}
    \begin{array}{l}
      l_t^T (u) \assign u_t + \lambda \mathbb{1}_{t \leqslant T}
      \mathbb{W}^{\langle 3 \rangle}_t + \lambda \mathbb{1}_{t \leqslant T}
      J_t (\mathbb{W}_t^2 \succ Z_t^{\flat}),\\
      Z_t (u) \assign I_t (u), \quad K_t (u) \assign I_t (w), \quad w_t
      \assign - \lambda \mathbb{1}_{t \leqslant T} J_t (\mathbb{W}_t^2 \succ
      Z_t^{\flat}) + l_t,
    \end{array} \quad t \geqslant 0, \label{eq:full-ansatz}
  \end{equation}
  Then the functional $F_T (u)$ defined in eq.~{\eqref{eq:3d-FT}} has the form
  \begin{eqnarray*}
    F_T (u) & = & \mathbb{E} \left[ \Phi_T (\mathbb{W}, Z (u), K (u)) +
    \lambda \bint (Z_T (u))^4 + \frac{1}{2} \| l^T (u) \|_{\mathcal{H}}^2
    \right],
  \end{eqnarray*}
  where
  \[ \Phi_T (\mathbb{W}, Z, K) \assign f (W_T + Z_T) + \sum_{i = 1}^6
     \Upsilon^{(i)}, \]
  \begin{eqnarray*}
    \Upsilon^{(1)}_T & \assign & - \frac{\lambda}{2} \mathfrak{K}_2
    (\mathbb{W}^2_T, K_T, K_T) + \frac{\lambda}{2} \bint (\mathbb{W}^2_T \prec
    K_T) K_T - \lambda^2 \bint (\mathbb{W}^2_T \prec \mathbb{W}_T^{[3]}) K_T\\
    \Upsilon^{(2)}_T & \assign & \lambda \bint (\mathbb{W}^2_T \succ (Z_T -
    Z^{\flat}_T)) K_T\\
    \Upsilon^{(3)}_T & \assign & \lambda \int_0^T \bint (\mathbb{W}^2_t \succ
    \dot{Z}^{\flat}_t) K_t \mathd t\\
    \Upsilon^{(4)}_T & \assign & 4 \lambda \bint W_T K^3_T - 12 \lambda^2
    \bint W_T \mathbb{W}_T^{[3]} K^2_T + 12 \lambda^3 \bint W_T
    (\mathbb{W}_T^{[3]})^2 K_T\\
    \Upsilon^{(5)}_T & \assign & - \lambda^2 \bint \gamma_T Z^{\flat}_T (Z_T -
    Z^{\flat}_T) - \lambda^2 \bint \gamma_T (Z_T - Z^{\flat}_T)^2 - 2
    \lambda^2 \int^T_{ 0} \bint \gamma_t Z^{\flat}_t \dot{Z}_t^{\flat}
    \mathd t\\
    \Upsilon^{(6)}_T & \assign & - \lambda^2 \bint \mathbb{W}^{2 \diamond
    [3]}_T K_T - \frac{\lambda^2}{2} \int_0^T \bint \mathbb{W}_t^{\langle 2
    \rangle \diamond \langle 2 \rangle} (Z^{\flat}_t)^2 \mathd t -
    \frac{\lambda^2}{2} \int_0^T \mathfrak{K}_{3, t} (\mathbb{W}^2_t,
    \mathbb{W}^2_t, Z^{\flat}_t, Z^{\flat}_t) \mathd t
  \end{eqnarray*}
  where $\mathfrak{K}_2$ and $\mathfrak{K}_{3, t}$ are linear forms defined in
  the Appendix (and recalled in the proof below) and we have chosen
  \begin{equation}
    \begin{array}{lll}
      \delta_T & \assign & \frac{\lambda^2}{2} \mathbb{E} \int_0^T \bint
      (\mathbb{W}_t^{\langle 3 \rangle})^2 \mathd t + \frac{\lambda^3}{2}
      \mathbb{E} \bint \mathbb{W}^2_T (\mathbb{W}_T^{[3]})^2\\
      &  & + 2 \lambda^3 \gamma_T \mathbb{E} \bint W_T \mathbb{W}_T^{[3]} - 4
      \lambda^4 \mathbb{E} \bint W_T (\mathbb{W}_T^{[3]})^3 .
    \end{array} \label{eq:choice-delta}
  \end{equation}
\end{lemma}

\begin{proof}
  \tmtextbf{Step 1.} We are going to absorb the mixed
  terms~{\eqref{eq:mixed-terms}} via the quadratic cost function. To do so we
  develop them along the flow of the scale parameter via Ito formula. For the
  first we have
  \[ \lambda \bint \mathbb{W}_T^3 Z_T = \lambda \int_0^T \bint \mathbb{W}_t^3
     \dot{Z}_t \mathd t + \tmop{martingale}, \]
  and we can cancel the first term on the r.h.s. by making the change of
  variables
  \begin{equation}
    w_t \assign u_t + \lambda \mathbb{1}_{t \leqslant T} \mathbb{W}_t^{\langle
    3 \rangle}, \qquad t \geqslant 0, \label{eq:first-ansatz}
  \end{equation}
  into the cost functional to get
  \[ \lambda \bint \mathbb{W}_T^3 Z_T + \frac{1}{2} \int_0^{\infty} \| u_s
     \|^2_{L^2} \mathd s = - \frac{\lambda^2}{2} \int_0^T \bint
     (\mathbb{W}_t^{\langle 3 \rangle})^2 \mathd t + \frac{1}{2}
     \int_0^{\infty} \| w_s \|^2_{L^2} \mathd s + \tmop{martingale}, \]
  where we used that $J_t$ is self-adjoint. The divergent term $\bint
  \mathbb{W}_T^3 Z_T$ has been replaced with a divergent but purely stochastic
  term $\int_0^T \bint (\mathbb{W}_t^{\langle 3 \rangle})^2 \mathd t$ which
  does not affect anymore the variational problem and can be explicitly
  removed by adding its average to $\delta_T$. As a consequence, we are not
  able to control $(Z_t)_t$ in $H^1$ anymore and we should rely on the
  relation~{\eqref{eq:full-ansatz}} and on a control over the $H^1$ norm of
  $(K_t)_t$ coming from the residual quadratic term $\| w \|^2_{\mathcal{H}}$.
  
  \tmtextbf{Step 2.} From~{\eqref{eq:first-ansatz}} we have the relation
  \[ Z_T = - \lambda \mathbb{W}_T^{[3]} + K_T, \]
  which can be used to expand the second mixed divergent term
  in~{\eqref{eq:mixed-terms}} as
  \begin{equation}
    \frac{\lambda}{2} \bint \mathbb{W}^2_T Z_T^2 = \frac{\lambda^3}{2} \bint
    \mathbb{W}^2_T (\mathbb{W}_T^{[3]})^2 - \lambda^2 \bint \mathbb{W}^2_T
    \mathbb{W}_T^{[3]} K_T + \frac{\lambda}{2} \bint \mathbb{W}^2_T K_T^2 .
    \label{eq:r2}
  \end{equation}
  Again, the first term on the r.h.s. a purely stochastic object and will give
  a contribution independent of the drift $u$ absorbed in $\delta_T$. We are
  still not done since this operation has left two new divergent terms on the
  r.h.s. of eq.~{\eqref{eq:r2}}: the $H^1$ regularity of $K_T$ is not enough
  to control the products with $\mathbb{W}^2$ which has regularity $\VV^{- 1 -
  \kappa}$, a bit below $- 1$. In order to proceed further we will isolate the
  divergent parts of these products via a paraproduct decomposition (see
  Appendix~\ref{sec:appendix-para} for details) and expand
  \begin{eqnarray*}
    - \lambda^2 \bint \mathbb{W}^2_T \mathbb{W}_T^{[3]} K_T +
    \frac{\lambda}{2} \bint \mathbb{W}^2_T K_T^2 & = & \lambda \bint
    (\mathbb{W}^2_T \succ Z_T) K_T - \lambda^2 \bint (\mathbb{W}^2_T \circ
    \mathbb{W}_T^{[3]}) K_T\\
    &  & - \lambda^2 \bint (\mathbb{W}^2_T \prec \mathbb{W}_T^{[3]}) K_T +
    \frac{\lambda}{2} \bint (\mathbb{W}^2_T \prec K_T) K_T\\
    &  & + \frac{\lambda}{2} \left( \bint (\mathbb{W}^2_T \circ K_T) K_T -
    \bint (\mathbb{W}^2_T \succ K_T) K_T \right) .
  \end{eqnarray*}
  The first two terms will require renormalizations which we put in place in
  Step~3 below. All the other terms will be well behaved and we collect them
  in $\Upsilon^{(1)}_T$. In particular we observe that the last one can be
  rewritten as
  \[ \frac{\lambda}{2} \left( \bint (\mathbb{W}^2_T \circ K_T) K_T - \bint
     (\mathbb{W}^2_T \succ K_T) K_T \right) = - \frac{\lambda}{2}
     \mathfrak{K}_2 (\mathbb{W}^2_T, K_T, K_T) \]
  using the trilinear form $\mathfrak{K}_2$ defined in
  Proposition~\ref{adjointparaproduct}.
  
  \tmtextbf{Step 3.} As we anticipated, the resonant term $\mathbb{W}^2_T
  \circ \mathbb{W}_T^{[3]}$ needs renormalization. In the expression of $F_T$
  in~{\eqref{eq:3d-FT}} we have the counterterm $- 2 \lambda^2 \gamma_T \bint
  W_T Z_T$ available, which we put in use now by writing
  \[ - \lambda^2 \bint (\mathbb{W}^2_T \circ \mathbb{W}_T^{[3]}) K_T - 2
     \lambda^2 \gamma_T \bint W_T Z_T = - \lambda^2 \bint
     \underbrace{(\mathbb{W}^2_T \circ \mathbb{W}_T^{[3]} + 2 \gamma_T
     W_T)}_{\mathbb{W}^{2 \diamond [3]}_T} K_T + 2 \lambda^3 \gamma_T \bint
     W_T \mathbb{W}_T^{[3]} . \]
  The first contribution is collected in $\Upsilon^{(6)}_T$ and the
  expectation of the second will contribute to $\delta_T$.
  
  As far as the term $\lambda \bint (\mathbb{W}^2_T \succ Z_T) K_T$ is
  concerned, we want to absorb it into $\int \| w_s \|^2 \mathd s$ like we did
  with the linear term in Step~2. Before we can do this we must be sure that,
  after applying Ito's formula, it will be still possible to use $\bint Z^4_T$
  to control some of the growth of this term. Indeed the quadratic dependence
  in $K_T$ (via $Z_T$) cannot be fully taken care of by the quadratic cost
  $\int \| w_s \|^2 \mathd s$.
  
  We decompose
  \[ \lambda \bint (\mathbb{W}^2_T \succ Z_T) K_T = \lambda \bint
     (\mathbb{W}^2_T \succ Z^{\flat}_T) K_T + \lambda \bint (\mathbb{W}^2_T
     \succ (Z_T - Z^{\flat}_T)) K_T \]
  and using the fact that the functions $Z_T - Z_T^{\flat}$ and $K_T -
  K_T^{\flat}$ are spectrally supported outside of a ball or radius $c T$ we
  will be able to show that the second term is nice enough as $T \rightarrow
  \infty$ to not require further analysis and we collect it in
  $\Upsilon^{(2)}_T$. For the first we apply Ito's formula to decompose it
  along the flow of scales as
  \[ \lambda \bint (\mathbb{W}^2_T \succ Z^{\flat}_T) K_T = \lambda \int_0^T
     \bint (\mathbb{W}^2_t \succ Z^{\flat}_t) \dot{K_t} \mathd t + \lambda
     \int_0^T \bint (\mathbb{W}^2_t \succ \dot{Z}^{\flat}_t) K_t \mathd t +
     \text{martingale} . \]
  The second term will be fine and we collect it in $\Upsilon^{(3)}_T$.
  
  \tmtextbf{Step 4. }We are left with the singular term $\int_0^T \bint
  (\mathbb{W}^2_t \succ Z^{\flat}_t) \dot{K_t} \mathd t$. Using
  eq.~{\eqref{eq:full-ansatz}} and expanding $w$ in the residual quadratic
  cost function obtained in Step~1, we compute
  \[ \lambda \int_0^T \bint (\mathbb{W}^2_t \succ Z^{\flat}_t) \dot{K_t}
     \mathd t + \frac{1}{2} \int_0^{\infty} \| w_t \|^2_{L^2} \mathd t = -
     \frac{\lambda^2}{2} \int_0^T \bint (J_t (\mathbb{W}^2_t \succ
     Z^{\flat}_t))^2 \mathd t + \frac{1}{2} \int_0^{\infty} \| l_t \|^2_{L^2}
     \mathd t \]
  \begin{equation}
    = - \frac{\lambda^2}{2} \int_0^T \bint (J_t (\mathbb{W}^2_t \succ
    Z^{\flat}_t)) (J_t (\mathbb{W}^2_t \succ Z^{\flat}_t)) \mathd t +
    \frac{1}{2} \| l \|^2_{\mathcal{H}} \label{eq:res-dec}
  \end{equation}
  To renormalize the first term on the r.h.s. we observe that the remaining
  couterterm can be rewritten as
  \begin{equation}
    - \lambda^2 \gamma_T \bint Z^2_T = - \lambda^2 \gamma_T \bint
    (Z^{\flat}_T)^2 - \lambda^2 \gamma_T \bint Z^{\flat}_T (Z_T - Z^{\flat}_T)
    - \lambda^2 \gamma_T \bint (Z_T - Z^{\flat}_T)^2 .
    \label{eq:decomp-square}
  \end{equation}
  Differentiating in $T$ the first term in the r.h.s. of
  eq.~{\eqref{eq:decomp-square}} we get
  \begin{equation}
    - \lambda^2 \gamma_T \bint (Z^{\flat}_T)^2 = - \lambda^2 \int_0^T \bint
    \dot{\gamma}_t (Z^{\flat}_t)^2 \mathd t - 2 \lambda^2 \int^T_{ 0} \bint
    \gamma_t Z^{\flat}_t \dot{Z}_t^{\flat} \mathd t. \label{eq:res-Z}
  \end{equation}
  The last term in eq.~{\eqref{eq:decomp-square}} and the last two
  contributions in~{\eqref{eq:res-Z}} are collected in $\Upsilon^{(5)}_T$. The
  first contribution in eq.~{\eqref{eq:decomp-square}} has the right form to
  be used as a counterterm for the resonant product in~{\eqref{eq:res-dec}}.
  Using the commutator $\mathfrak{K}_{3, t}$ introduced in
  Proposition~\ref{prop:squarecomm} we have
  \[ - \frac{\lambda^2}{2} \int_0^T \bint [(J_t (\mathbb{W}^2_t \succ
     Z^{\flat}_t)) (J_t (\mathbb{W}^2_t \succ Z^{\flat}_t)) + 2 \dot{\gamma}_t
     (Z^{\flat}_t)^2] \mathd t \]
  \[ = - \frac{\lambda^2}{2} \int_0^T \bint \underbrace{[(J_t \mathbb{W}^2_t)
     \circ (J_t \mathbb{W}^2_t) + 2 \dot{\gamma}_t]}_{\mathbb{W}^{\langle 2
     \rangle \diamond \langle 2 \rangle}} (Z^{\flat}_t)^2 \mathd t -
     \frac{\lambda^2}{2} \int_0^T \mathfrak{K}_{3, t} (\mathbb{W}^2_t,
     \mathbb{W}^2_t, Z^{\flat}_t, Z^{\flat}_t) \mathd t \]
  and collect both terms in $\Upsilon^{(6)}_T$.
  
  \tmtextbf{Step 5.} We are now left with the cubic term which we rewrite as
  \[ 4 \lambda \bint W_T Z^3_T = - 4 \lambda^4 \bint W_T
     (\mathbb{W}_T^{[3]})^3 + 12 \lambda^3 \bint W_T (\mathbb{W}_T^{[3]})^2
     K_T - 12 \lambda^2 \bint W_T \mathbb{W}_T^{[3]} K^2_T + 4 \lambda \bint
     W_T K^3_T . \]
  The average of the first term is collected in $\delta_T$ while all the
  remaining terms in $\Upsilon^{(4)}_T$. At last we have established the
  claimed decomposition since the residual cost functional, from
  eq.~{\eqref{eq:res-dec}} has the form $\| l \|^2_{\mathcal{H}}$.
\end{proof}

\section{Bounds}\label{sec:bounds}

The aim of this section is to give upper and lower bounds on $\mathcal{W}_T
(f)$ uniformly on $T$ and $| \Lambda |$. In particular we will prove the
bounds of Corollary~\ref{corollary:energy-bounds-3d} taking the explicit
dependence on the coupling constant $\lambda$ into account.

\begin{lemma}
  \label{lemma:bounds}There exists a finite constant $C$, which does not
  depend on $| \Lambda |$, such that
  \[ \sup_T | \mathcal{W}_T (f) | \leqslant C. \]
\end{lemma}

\begin{proof}
  Observe that from Lemma~\ref{lemma:change-of-variables} and
  Section~\ref{section:analytic} we have that
  \[ \Phi_T (\mathbb{W}, Z, K) \leqslant Q_T + \varepsilon \left( \| Z_T
     \|^4_{L^4} + \int^{\infty}_0 \| l_t \|^2_{L^2} \mathd t \right), \]
  which immediately gives \
  \begin{equation}
    -\mathbb{E} [Q_T] \leqslant -\mathbb{E} [Q_T] + (1 - \varepsilon)
    \mathbb{E} \left( \| Z_T \|^4_{L^4} + \int^{\infty}_0 \| l_t \|^2_{L^2}
    \mathd t \right) \leqslant \mathcal{W}_T (f) . \label{eq:lower-bound}
  \end{equation}
  On the other hand for any suitable drift $\check{u} \in \mathbb{H}_a$ we get
  the bound
  \begin{equation}
    \mathcal{W}_T (f) \leqslant \mathbb{E} [Q_T] + (1 + \varepsilon)
    \mathbb{E} \left( \| I_T (\check{u}) \|^4_{L^4} + \int^{\infty}_0 \| l_t^T
    (\check{u}) \|^2_{L^2} \mathd t \right), \label{eq:upper-bound}
  \end{equation}
  where
  \begin{equation}
    l_t^T (\check{u}) = \check{u}_t + \lambda \mathbb{1}_{t \leqslant T} J_t
    (\mathbb{W}_t^3 +\mathbb{W}_t^2 \succ (I_t (\check{u}))^{\flat}) .
    \label{eq:elle-ub}
  \end{equation}
  Therefore it remains to produce an appropriate drift $\check{u}$ for which
  the r.h.s. in eq.~{\eqref{eq:upper-bound}} is finite (and so uniformly in $|
  \Lambda |$ and of $o (\lambda^3)$).
  
  \
  
  One possible strategy is to try and choose $\check{u}$ such that $l
  (\check{u}) = 0$, however this fails since estimates on this $u$ via
  Gronwall's inequality would rely on the H{\"o}lder norm of $\mathbb{W}_t^2$
  for which we do not have uniform control in the volume. In order to overcome
  this problem we decompose $\mathbb{W}^2$ and use weighted estimates
  similarly as done in~{\cite{gubinelli_global_2018}}.
  
  \
  
  Consider the decomposition
  \[ \mathbb{W}_s^2 = \mathcal{U}_{\geqslant} \mathbb{W}_s^2 +
     \mathcal{U}_{\leqslant} \mathbb{W}_s^2, \]
  where the random field $\mathcal{U}_{\geqslant} \mathbb{W}_s^2$ is
  constructed as follows. Let $\varphi$ be smooth function, positive and
  supported on $[- 2, 2]^3$ and such that $\sum_{m \in \Lambda \cap
  \mathbb{Z}^d} \varphi^2 (\bullet - m) = 1$. Denote by $\varphi_m = \varphi
  (\bullet - m)$. Let $\tilde{\chi}$ be a smooth function supported in $B (0,
  1)$. Denote by $\mathcal{X}_{> N} f$ the Fourier multiplier operator
  $\mathcal{F}^{- 1} \tilde{\chi} (k / N) \hat{f}$ and similarly
  $\mathcal{X}_{\leqslant N} f = \mathcal{F}^{- 1} (1 - \tilde{\chi} (k / N))
  \hat{f}$. Set $L_m (s) \assign (1 + \| \varphi_m \mathbb{W}_s^2
  \|)_{\mathcal{\VV}^{- 1 - \delta}}^{\frac{1}{2 \delta}}$ and let
  \[ \mathcal{U}_{>} \mathbb{W}_s^2 \assign \sum_{m \in \Lambda \cap
     \mathbb{Z}^d} \varphi_m \mathcal{X}_{> L_m (s)} (\varphi_m
     \mathbb{W}_s^2), \]
  and
  \[ \mathcal{U}_{\leqslant} \mathbb{W}_s^2 \assign \sum_{m \in \Lambda \cap
     \mathbb{Z}^d} \varphi_m \mathcal{X}_{\leqslant L_m (s)} (\varphi_m
     \mathbb{W}_s^2) . \]
  (with slight abuse of notation we drop the time dependence of the operators
  $\mathcal{U}_{\leqslant}$, $\mathcal{U}_{>}$).
  
  Observe that the laws of both $\mathcal{U}_{>} \mathbb{W}_s^2$ and
  $\mathcal{U}_{\leqslant} \mathbb{W}_s^2$ are translation invariant w.r.t to
  translations by $m \in \Lambda \cap \mathbb{Z}^d$. By {\cite{Triebel_1992}},
  Theorem 2.4.7 and Bernstein inequality
  \begin{eqnarray*}
    \| \mathcal{U}_{>} \mathbb{W}_s^2 \|_{\mathcal{\VV}^{- 1 - 3 \delta}} &
    \lesssim & \sup_m \| \mathcal{X}_{> L_m (s)} (\varphi_m \mathbb{W}_s^2)
    \|_{\mathcal{\VV}^{- 1 - 3 \delta}}\\
    & \lesssim & \sup_m \frac{1}{1 + \| \varphi_m \mathbb{W}_s^2 \|_{\VV^{- 1
    - \delta}}} \| \varphi_m \mathbb{W}_s^2 \|_{\mathcal{\VV}^{- 1 - \delta}}
    \lesssim 1
  \end{eqnarray*}
  Furthermore for a weight $\rho$ (see Appendix~\ref{sec:appendix-para} for
  precisions on the weighted spaces $L^p (\rho)$, $\VV^{\alpha} (\rho)$ and
  $B^{\alpha}_{p, q} (\rho)$ used below):
  \begin{equation}
    \begin{array}{lll}
      \| \mathcal{U}_{\leqslant} \mathbb{W}_s^2 \|_{\mathcal{\VV}^{- 1 +
      \delta} (\rho^2)} & \lesssim & \sup_m \| \varphi_m
      \mathcal{U}_{\leqslant} \mathbb{W}_s^2 \|_{\mathcal{\VV}^{- 1 + \delta}
      (\rho^2)}\\
      & \lesssim & \sup_m \left( 1 + \| \varphi_m \mathbb{W}_s^2
      \|_{\mathcal{\VV}^{- 1 - \delta}} \right) \| \varphi_m \mathbb{W}_s^2
      \|_{\VV^{- 1 - \delta} (\rho^2)}\\
      & \lesssim & \sup_m \rho (m)  \left( 1 + \| \varphi_m \mathbb{W}_s^2
      \|_{\mathcal{\VV}^{- 1 - \delta}} \right) \| \varphi_m \mathbb{W}_s^2
      \|_{\VV^{- 1 - \delta} (\rho)}\\
      & \lesssim & \sup_m  \left( 1 + \| \varphi_m \mathbb{W}_s^2
      \|_{\mathcal{\VV}^{- 1 - \delta} (\rho)} \right)_{} \| \varphi_m
      \mathbb{W}_s^2 \|_{\mathcal{\VV}^{- 1 - \delta} (\rho)}\\
      & \lesssim & 1 + \| \mathbb{W}_s^2 \|_{\mathcal{\VV}^{- 1 - \delta}
      (\rho)}^2,
    \end{array} \label{boundsmooth}
  \end{equation}
  where we used the possibility to compare weighted and unweighted norms once
  localized via $\varphi_m$. We now let $\check{u}$ be the solution to the
  (integral) equation
  \begin{equation}
    \check{u}_t = - \lambda \mathbb{1}_{t \leqslant T} [\mathbb{W}^{\langle 3
    \rangle}_t + J_t  \mathcal{U}_{>} \mathbb{W}^2_t \succ \theta_t (I_t
    (\check{u}))], \qquad t \geqslant 0, \label{eq:int-ub}
  \end{equation}
  which can be solved globally. For $3 \delta < 1 / 2$ and $p \geqslant 1$, we
  have, for $t \in [0, T]$,
  \[ \| I_t (\check{u}) \|_{B_{p, p}^{1 / 2 - 3 \delta} (\rho)} \lesssim
     \lambda \int^t_0 [\| J^2_s \mathbb{W}_s^3 \|_{B_{p, p}^{1 / 2 - \delta}
     (\rho)} + \lambda \| J^2_s \mathcal{U}_{>} \mathbb{W}_s^2 \succ \theta_s
     (I_s (\check{u})) \|_{B_{p, p}^{- 3 / 2 - \delta} (\rho)}] \mathd s \]
  \[ \lesssim \lambda \int^t_0 \frac{\mathd s}{\langle s \rangle^{1 / 2 +
     \delta}} \| J_s \mathbb{W}_s^3 \|_{B_{p, p}^{- 1 / 2 - \delta} (\rho)} +
     \lambda \int^t_0 \frac{\mathd s}{\langle s \rangle^{1 + \delta / 2}} \|
     \mathcal{U}_{>} \mathbb{W}_s^2 \|_{\mathcal{\VV}^{- 1 - \delta / 2}} \|
     I_s (\check{u}) \|_{B_{p, p}^{1 / 2 - 3 \delta} (\rho)} . \]

  Therefore Gronwall's lemma implies that, for $t \in [0, T]$:
  \begin{equation}
    \begin{array}{lll}
      \| I_t (\check{u}) \|_{B_{p, p}^{1 / 2 - \delta} (\rho)} & \lesssim &
      \left( \lambda \int^T_0 \frac{\mathd s}{\langle s \rangle^{1 / 2 +
      \delta}} \| J_s \mathbb{W}_s^3 \|_{B_{p, p}^{- 1 / 2 - \delta} (\rho)}
      \right) \exp \left( \lambda \int^T_0 \frac{\| \mathcal{U}_{>}
      \mathbb{W}_s^2 \|_{\mathcal{\VV}^{- 1 - \varepsilon}} \mathd s}{\langle
      s \rangle^{1 + \varepsilon}} \right)\\
      & \lesssim & \left( \lambda \int^T_0 \frac{\mathd s}{\langle s
      \rangle^{1 / 2 + \delta}} \| J_s \mathbb{W}_s^3 \|_{B_{p, p}^{- 1 / 2 -
      \delta} (\rho)} \right)\\
      & \lesssim & \lambda \| \mathbb{W}^{\langle 3 \rangle} \|_{L^2
      (\mathbb{R}_+, B_{p, p}^{- 1 / 2 - \delta} (\rho))} .
    \end{array} \label{eq:ub-I}
  \end{equation}
  Taking $\rho = \frac{\mathbb{1}_{\Lambda}}{| \Lambda |}$ and using Besov
  embedding we deduce from~{\eqref{eq:ub-I}}:
  \[ \sup_T \mathbb{E} \| I_T (\check{u}) \|^4_{L^4} \lesssim \lambda^4
     \mathbb{E} \left( \int^{\infty}_0 \frac{\mathd s}{\langle s \rangle^{1 /
     2 + \delta}} \| J_s \mathbb{W}_s^3 \|_{B_{4, 4}^{- 1 / 2 - \delta}}
     \right)^4 \lesssim \lambda^4 . \]
  Now computing $l^T (\check{u})$ from eq.~{\eqref{eq:elle-ub}}
  and~{\eqref{eq:int-ub}}, we obtain
  \[ l_t^T (\check{u}) = \lambda \mathbb{1}_{t \leqslant T} J_t 
     \mathcal{U}_{\leqslant} \mathbb{W}_t^2 \succ \theta_t (I_t (\check{u})),
     \qquad t \geqslant 0. \]
  It now remains to prove that $\mathbb{E} [\| l^T (\check{u})
  \|^2_{\mathcal{H}}] \lesssim O (\lambda^3)$ uniformly in $T > 0$. Note that,
  for $s \in [0, T]$,
  \begin{equation}
    \begin{array}{lll}
      \| J_s  \mathcal{U}_{\leqslant} \mathbb{W}_s^2 \succ \theta_s (I_s
      (\check{u})) \|_{L^2 (\rho^2)} & \lesssim & \frac{1}{\langle s
      \rangle^{1 / 2 + \delta / 2}} \| \mathcal{U}_{\leqslant} \mathbb{W}_s^2
      \succ \theta_s (I_s (\check{u})) \|_{B_{2, 2}^{- 1 + \delta / 2}
      (\rho^2)}\\
      & \lesssim & \frac{1}{\langle s \rangle^{1 / 2 + \delta / 2}} \|
      \mathcal{U}_{\leqslant} \mathbb{W}_s^2 \|_{\mathcal{\VV}^{- 1 + \delta /
      2} (\rho)} \| I_s (\check{u}) \|_{B_{2, 2}^{- 1 + \delta / 2} (\rho)} .
    \end{array} \label{eq:ub3}
  \end{equation}
  We know that the distribution of $\check{u}$ is invariant under translation
  by $m \in \Lambda \cap \mathbb{Z}^d$. Recalling that $\sum_{m \in \Lambda
  \cap \mathbb{Z}^d} \varphi^2 (\bullet - m) = 1$ and letting $\rho$ be a
  polynomial weight with sufficient decay and such that $\rho^3 \geqslant
  \varphi^2$, we have
  \begin{eqnarray*}
    \mathbb{E} [\| l^T (\check{u}) \|^2_{\mathcal{H}}] & = & \lambda^2
    \mathbb{E} [\| s \mapsto \mathbb{1}_{s \leqslant T} J_s
    \mathcal{U}_{\leqslant} \mathbb{W}_s^2 \succ \theta_s (I_s (\check{u}))
    \|^2_{\mathcal{H}}]\\
    & \leqslant & \lambda^2 \sum_{m \in \Lambda \cap \mathbb{Z}^d} \mathbb{E}
    [\| s \mapsto \varphi^2 (\bullet - m) J_s \mathcal{U}_{\leqslant}
    \mathbb{W}_s^2 \succ \theta_s (I_s (\check{u})) \|^2_{\mathcal{H}}]\\
    \text{(by trans. inv)} & \lesssim & \lambda^2 | \Lambda | \mathbb{E} [\| s
    \mapsto \mathbb{1}_{s \leqslant T} \varphi^2 J_s \mathcal{U}_{\leqslant}
    \mathbb{W}_s^2 \succ \theta_s (I_s (\check{u})) \|^2_{\mathcal{H}}]\\
    \text{(using $\rho^3 \geqslant \varphi^2$)} & \lesssim & \lambda^2
    \int^T_0 \mathd s\mathbb{E} [\| J_s \mathcal{U}_{\leqslant} \mathbb{W}_s^2
    \succ \theta_s (I_s (\check{u})) \|^2_{L^2 (\rho^3)}]\\
    \text{(by eq.~{\eqref{eq:ub3}})} & \lesssim & \lambda^2 \int^T_0
    \frac{\mathd s}{\langle s \rangle^{1 + \delta}} \mathbb{E} \left[ \|
    \mathcal{U}_{\leqslant} \mathbb{W}_s^2 \|_{\mathcal{\VV}^{- 1 + \delta /
    2} (\rho^2)}^2 \| I_s (\check{u}) \|_{B_{2, 2}^{- 1 + \delta / 2}
    (\rho)}^2 \right]\\
    & \lesssim & \lambda^2 \int^T_0 \frac{\mathd s}{\langle s \rangle^{1 +
    \delta}} \mathbb{E} \left[ \lambda^2 \| \mathcal{U}_{\leqslant}
    \mathbb{W}_s^2 \|^4_{\mathcal{\VV}^{- 1 + \delta / 2} (\rho^2)} +
    \lambda^{- 2} \| I_s (\check{u}) \|_{B_{2, 2}^{- 1 + \delta / 2} (\rho)}^4
    \right]\\
    \text{(by eqs.~{\eqref{eq:ub-I}}, {\eqref{boundsmooth}})} & \lesssim &
    \lambda^4 \int^{\infty}_0 \frac{\mathd s}{\langle s \rangle^{1 + \delta}}
    \left[ 1 +\mathbb{E} \| \mathbb{W}_s^2 \|^8_{\mathcal{\VV}^{- 1 - \delta /
    2} (\rho)} +\mathbb{E} \lambda \| \mathbb{W}^{\langle 3 \rangle} \|^4_{L^2
    (\mathbb{R}_+, B_{p, p}^{- 1 / 2 - \delta} (\rho))} \right]\\
    & \lesssim & \lambda^4 .
  \end{eqnarray*}
  This last quantities are bounded since standard arguments allow to bound,
  for $p$ sufficiently large,
  \[ \left[ \mathbb{E} \| \mathbb{W}_s^2 \|^8_{\mathcal{\VV}^{- 1 - \delta /
     2} (\rho)} \right]^{p / 8} \leqslant \mathbb{E} \| \mathbb{W}_s^2
     \|^p_{\mathcal{\VV}^{- 1 - \delta / 2} (\rho)} \leqslant \mathbb{E} \|
     \mathbb{W}_s^2 \|^p_{B_{p, p}^{- 1 - \delta} (\rho)} \]
  \[ = \sum_{i \geqslant - 1} 2^{i (- 1 - \delta / 2) p} \int_{\Lambda} \mathd
     x \rho (x) \mathbb{E} | \Delta_i \mathbb{W}_s^2 (x) |^p \lesssim \sum_{i
     \geqslant - 1} 2^{i (- 1 - \delta / 2) p} \mathbb{E} | \Delta_i
     \mathbb{W}_s^2 (0) |^p \lesssim 1 \]
  uniformly in $s \geqslant 0$. Similarly, we have
  \[ [\mathbb{E} \| \mathbb{W}_s^3 \|^4_{\nobracket B_{p, p} (\rho))}]^{p / 4}
     \leqslant \mathbb{E} \| \mathbb{W}_s^3 \|_{B_{p, p} (\rho)}^p =\mathbb{E}
     | \mathbb{W}_s^3 (0) |^p, \]
  Now
  \[ \sup_{s \leqslant T} \mathbb{E} | \mathbb{W}_s^3 (0) |^p \lesssim
     \sup_{s \leqslant T} (\mathbb{E} | \mathbb{W}_s^3 (0) |^2)^{p / 2}
     \lesssim T^{3 p / 2} \]
  and
  \[ \| \mathbb{W}^{\langle 3 \rangle} \|_{L^2 (\mathbb{R}_+, B_{p, p}^{- 1 /
     2 - \delta} (\rho))} \lesssim \int^{\infty}_0 \| J_s \mathbb{W}^3
     \|_{B_{p, p}^{- 1 / 2 - \delta} (\rho)} \]
  \begin{eqnarray*}
    \| \mathbb{W}^{\langle 3 \rangle} \|^2_{L^2 (\mathbb{R}_+, B_{p, p}^{- 1 /
    2 - \delta} (\rho))} & \lesssim & \int^{\infty}_0 \| J_s \mathbb{W}_s^3
    \|^2_{B_{p, p}^{- 1 / 2 - \delta} (\rho)} \mathd s\\
    & \lesssim & \int^{\infty}_0 \left\| \frac{\sigma_t (D)}{\langle D
    \rangle} \mathbb{W}_t^3 \right\|^2_{B_{p, p}^{- 1 / 2 - \delta} (\rho)}
    \mathd s\\
    & \lesssim & \int^{\infty}_0 \| \sigma_t (D) \mathbb{W}_t^3 \|^{2
    m}_{B_{p, p}^{- 3 / 2 - \kappa}} \mathd s\\
    & \lesssim & \int^{\infty}_0 \langle t \rangle^{- 1 - \kappa} (\langle t
    \rangle^{- 3} \| \mathbb{W}_t^3 \|^2_{L^p}) \mathd s
  \end{eqnarray*}
  This concludes the proof of the lemma.
  
  \ 
\end{proof}

\begin{remark}
  The decomposition of the noise is similar to the one given
  in~{\cite{gubinelli_global_2018}} but differs in the fact that we choose the
  frequency cutoff dependent on the size of the noise instead of the point, to
  preserve translation invariance. The price to pay is that the decomposition
  is nonlinear, however this does not present any inconvenience in our
  context.
\end{remark}

\section{Gamma convergence\label{sec:gamma-convergence}}

In this section we establish the $\Gamma$-convergence of the variational
functional obtained in Lemma~\ref{lemma:change-of-variables} as $T \rightarrow
\infty$. $\Gamma$-convergence is a notion of convergence introduced by
De~Giorgi which is well suited for the study of variational problems. The
book~{\cite{braides_gamma_convergence_2002}} is a nice introduction to
$\Gamma$-convergence in the context of the calculus of variations. For the
convenience of the reader we recall here the basic definitions and results.

\begin{definition}
  \label{Gammaconv}Let $\mathcal{T}$ be a topological space and let $F, F_n :
  \mathcal{T} \rightarrow (- \infty, \infty]$. We say that the sequence of
  functionals $(F_n)_n$ $\Gamma$-converges to $F$ iff
  \begin{enumerateroman}
    \item For every sequence $x_n \rightarrow x$ in $\mathcal{T}$
    \[ F (x) \leqslant \liminf_{n \rightarrow \infty} F_n (x_n) ; \]
    \item For every point $x$ there exists a sequence $x_n \rightarrow x$
    (called a recovery sequence) such that
    \[ F (x) \geqslant \limsup_{n \rightarrow \infty} F_n (x_n) . \]
  \end{enumerateroman}
\end{definition}

\begin{definition}
  \label{equicoercive}A sequence of functionals $F_n : \mathcal{T} \rightarrow
  (- \infty, \infty]$ is called equicoercive if there exists a compact set
  $\mathcal{K} \subseteq \mathcal{T}$ such that for all $n \in \mathbb{N}$
  \[ \inf_{x \in \mathcal{K}} F_n (x) = \inf_{x \in \mathcal{T}} F_n (x) . \]
\end{definition}

A fundamental consequence of $\Gamma$-convergence is the convergence of
minima.

\begin{theorem}
  \label{fundamentallemma}If $(F_n)_n$ $\Gamma$-converges to $F$ and $(F_n)_n$
  is equicoercive, then $F$ admits a minimum and
  \[ \min_{\mathcal{T}} F = \lim_{n \rightarrow \infty} \inf_{\mathcal{T}}
     F_n . \]
\end{theorem}

For a proof see~{\cite{DalMaso1993}}.

\

In this section we allow all constants to depend on the volume $| \Lambda |$:
this is not critical since, at this point, the aim is to obtain explicit
formulas at fixed $\Lambda$.

\

We denote
\[ \mathcal{H}_{}^{\alpha, p} \assign L^2 ([0, \infty) ; W^{\alpha, p}),
   \qquad \alpha \in \mathbb{R}, \]
and by $\mathcal{H}_w^{\alpha, p}$ the reflexive Banach space
$\mathcal{H}^{\alpha, p}$ endowed with the weak topology. We will write
$\mathcal{H}^{\alpha}$ for $\mathcal{H}^{\alpha, 2}$, and $\mathcal{H}$ for
$\mathcal{H}^0$ and let $\mathcal{L} \assign \mathcal{H}^{- 1 / 2 - \kappa,
3}$, this space will be useful as it gives sufficient control over $Z$:

\begin{lemma}
  \label{lemma-Zcompact}For $\kappa$ small enough $u \mapsto Z (u)$ is a
  compact map $\mathcal{L} \rightarrow C ( [0, \infty], L^4)$.
\end{lemma}

\begin{proof}
  By definition of $Z$ we have for any $0 < \varepsilon < 1 / 8 - \kappa / 2$,
  \begin{eqnarray*}
    \| Z_{t_2} (u) - Z_{t_1} (u) \|_{W^{4, \varepsilon}} & = & \left\|
    \int^{t_2}_{t_1} J_s u_s \mathd s \right\|_{W^{4, \varepsilon}} \leqslant
    \int^{t_2}_{t_1} \left\| \frac{\sigma_s (D)}{\langle D \rangle} u_s
    \right\|_{W^{4, \varepsilon}} \mathd s\\
    & \lesssim & \int^{t_2}_{t_1} \frac{1}{\langle s \rangle^{1 / 2 +
    \varepsilon}} \| \langle D \rangle^{- 1 + \varepsilon} u_s \|_{W^{4,
    \varepsilon}} \mathd s\\
    & \lesssim & \int^{t_2}_{t_1} \frac{1}{\langle s \rangle^{1 / 2 +
    \varepsilon}} \| \langle D \rangle^{- 1 + \varepsilon} u_s \|_{W^{1 / 4 +
    \varepsilon, 3}} \mathd s\\
    & \lesssim & \int^{t_2}_{t_1} \frac{1}{\langle s \rangle^{1 + 2
    \varepsilon}} \mathd s \int \| u_s \|^2_{W^{- 1 / 2 - \kappa, 3}} \mathd s
    \lesssim \| u \|_{\mathcal{L}}^2 \int^{t_2}_{t_1} \frac{1}{\langle s
    \rangle^{1 + 2 \varepsilon}} \mathd s,
  \end{eqnarray*}
  where we have used a Sobolev embedding in the second to last line. Since
  \[ \lim_{t_1 \rightarrow t_2} \int^{t_2}_{t_1} \frac{1}{\langle s
     \rangle^{1 + 2 \varepsilon}} \mathd s = 0, \int^{\infty}_0
     \frac{1}{\langle s \rangle^{1 + 2 \kappa}} \mathd s < \infty, \]
  for any $t_2 \in [0, \infty]$, we can conclude by the Rellich--Kondrachov
  embedding theorem and the Arzela--Ascoli theorem, that bounded sets in
  $\mathcal{L}$ are mapped to compact sets in $C ( [0, \infty], L^4)$,
  proving the claim.
\end{proof}

We will need the following lemma, which establishes pointwise convergence for
the functional $\Phi_T$ defined in Lemma~\ref{lemma:change-of-variables}. In
the sequel, by an abuse of notation, we will denote both a generic element of
$\mathfrak{S}$ and the canonical random variable on $\mathfrak{S}$ by
\[ \mathbb{X}= (\mathbb{X}^1, \mathbb{X}^2, \mathbb{X}^{\langle 3 \rangle},
   \mathbb{X}^{[3] \circ 1}, \mathbb{X}^{\langle 2 \rangle \diamond \langle 2
   \rangle}_{}, \mathbb{X}^{2 \diamond [3]}) \]
\begin{lemma}
  \label{pointwiseconv}Define $l^{\infty} (u) = l^{\infty} (\mathbb{X}, u) \in
  \mathbb{H}_a$ such that
  \begin{equation}
    l_t^{\infty} (u) \assign u_t + \lambda \mathbb{W}^{\langle 3 \rangle}_t +
    \lambda J_t (\mathbb{W}_t^2 \succ Z_t^{\flat}), \qquad t \geqslant 0.
    \label{eq:full-ansatz-infty}
  \end{equation}
  For any sequence $(u^T, \mathbb{X}^T)_T$ such that $u^T \rightarrow u$ in
  $\mathcal{L}_w$, $l^T = l^T (\mathbb{X}^T, u^T) \rightarrow l = l^{\infty}
  (\mathbb{X}, u)$ in $\mathcal{H}_w$ and
  \[ \begin{array}{cll}
       \mathbb{X}^T & = & (\mathbb{X}^{T, 1}, \mathbb{X}^{T, 2},
       \mathbb{X}^{T, \langle 3 \rangle}, \mathbb{X}^{T, [3] \circ 1},
       \mathbb{X}^{T, \langle 2 \rangle \diamond \langle 2 \rangle}_{},
       \mathbb{X}^{T, 2 \diamond [3]})\\
       \downarrow &  & \\
       \mathbb{X} & = & (\mathbb{X}^1, \mathbb{X}^2, \mathbb{X}^{\langle 3
       \rangle}, \mathbb{X}^{[3] \circ 1}, \mathbb{X}^{2 \diamond [3]},
       \mathbb{X}^{\langle 2 \rangle \diamond \langle 2 \rangle}_{})
     \end{array} \]
  in $\mathfrak{S}$ we have
  \[ \lim_{T \rightarrow \infty} \Phi_T (\mathbb{X}^T, Z (u^T), K (u^T)) =
     \Phi_{\infty} (\mathbb{X}, Z (u), K (u)), \]
  where $\Phi_{\infty}$ is defined by
  \[ \Phi_{\infty} (\mathbb{X}, Z (u), K (u)) \assign f
     (\mathbb{X}^1_{\infty} + Z_{\infty} (u)) + \sum_{i = 1}^6
     \Upsilon_{\infty}^{(i)} (\mathbb{X}, Z (u), K (u)), \]
  with $\Upsilon_{\infty}^{(i)} (\mathbb{X}, Z, K) = \Upsilon_{\infty}^{(i)}$
  given by
  \begin{eqnarray*}
    \Upsilon^{(1)}_{\infty} & \assign & \frac{\lambda}{2} \mathfrak{K}_2
    (\mathbb{X}_{\infty}^2, K_{\infty}, K_{\infty}) + \frac{\lambda}{2} \bint
    (\mathbb{X}^2_{\infty} \prec K_{\infty}) K_{\infty} - \lambda^2 \bint
    (\mathbb{X}^2_{\infty} \prec \mathbb{X}_{\infty}^{[3]}) K_{\infty},\\
    \Upsilon^{(2)}_{\infty} & = & 0,\\
    \Upsilon^{(3)}_{\infty} & \assign & \lambda \int_0^{\infty} \bint
    (\mathbb{X}^2_t \succ \dot{Z}^{\flat}_t) K_t \mathd t,\\
    \Upsilon^{(4)}_{\infty} & \assign & 4 \lambda \bint \mathbb{X}^1_{\infty}
    K^3_{\infty} - 12 \lambda^2 \bint (\mathbb{X}^1_{\infty}
    \mathbb{X}_{\infty}^{[3]}) K^2_{\infty} + 12 \lambda^3 \bint
    \mathbb{X}^1_{\infty} (\mathbb{X}_{\infty}^{[3]})^2 K_{\infty},\\
    \Upsilon^{(5)}_{\infty} & \assign & - 2 \lambda^2 \int^{\infty}_{ 0}
    \bint \gamma_t Z^{\flat}_t \dot{Z}_t^{\flat} \mathd t,\\
    \Upsilon^{(6)}_{\infty} & \assign & - \lambda^2 \bint
    \mathbb{X}_{\infty}^{2 \diamond [3]} K_{\infty} - \lambda^2
    \int_0^{\infty} \bint \mathbb{X}^{\langle 2 \rangle \diamond \langle 2
    \rangle}_t (Z^{\flat}_t)^2 \mathd t + \frac{\lambda^2}{2} \int^{\infty}_0
    \mathfrak{K}_{3, t} (\mathbb{X}_t^2, \mathbb{X}_t^2, Z^{\flat}_t,
    Z^{\flat}_t) \mathd t,
  \end{eqnarray*}
  where, with abuse of notation, we set
  \begin{equation}
    \begin{array}{lll}
      \mathbb{X}^1_{\infty} \mathbb{X}_{\infty}^{[3]} & \assign &
      \mathbb{X}^1_{\infty} \succ \mathbb{X}_{\infty}^{[3]}
      +\mathbb{X}^1_{\infty} \prec \mathbb{X}_{\infty}^{[3]}
      +\mathbb{X}_{\infty}^{[3] \circ 1},\\
      \mathbb{X}^1_{\infty} (\mathbb{X}_{\infty}^{[3]})^2 & \assign &
      \mathbb{X}^1_{\infty} (\mathbb{X}_{\infty}^{[3]} \circ
      \mathbb{X}_{\infty}^{[3]}) + 2\mathbb{X}_{\infty}^{[3] \circ 1}
      \mathbb{X}_{\infty}^{[3]} + 2 \mathfrak{K}_1 (\mathbb{X}_{\infty}^{[3]},
      \mathbb{X}_{\infty}^{[3]}, \mathbb{X}^1_{\infty})\\
      &  & + 2\mathbb{X}^1_{\infty} \succ (\mathbb{X}_{\infty}^{[3]} \succ
      \mathbb{X}_{\infty}^{[3]}) + 2\mathbb{X}^1_{\infty} \prec
      (\mathbb{X}_{\infty}^{[3]} \succ \mathbb{X}_{\infty}^{[3]}),
    \end{array} \label{eq:prod-infty}
  \end{equation}
  and where $\mathfrak{K}_1, \mathfrak{K}_2, \mathfrak{K}_3$ are the trilinear
  forms defined in Proposition~\ref{commutatorestimate},
  Proposition~\ref{adjointparaproduct} and Proposition~\ref{prop:squarecomm}
  respectively.
\end{lemma}

\begin{proof}
  Lemma~\ref{lemma-Zcompact} implies that for any $u^T \rightarrow u$ in
  $\mathcal{L}_w$ we have $Z (u^T) \rightarrow Z (u)$ in $C ([0, \infty],
  L^4)$ and by the convergence of $l^T \rightarrow l$ in $\mathcal{H}_w$ we
  have also $K (u^T) \rightarrow K (u)$ in $C ([0, \infty], H^{1 - \kappa})$.
  The products $\mathbb{X}^{T, 1}_T \mathbb{X}_T^{T, [3]}$ and $\mathbb{X}^{T,
  1}_T (\mathbb{X}_T^{T, [3]})^2$ can be decomposed using paraproducts and,
  after replacing the resonant products by the corresponding stochastic
  objects in $\mathbb{X}^T$, we obtain the finite $T$ analogs of the
  expressions in eq.~{\eqref{eq:prod-infty}}. After this preprocessing, it is
  easy to see by continuity that we have $\mathbb{X}^{T, 1}_T \mathbb{X}_T^{T,
  [3]} \rightarrow \mathbb{X}_{\infty}^1 \mathbb{X}_{\infty}^{[3]}$ and
  $\mathbb{X}^{T, 1}_T (\mathbb{X}_T^{T, [3]})^2 \rightarrow
  \mathbb{X}^1_{\infty} (\mathbb{X}_{\infty}^{[3]})^2$ in $\VV^{1 / 2 -
  \kappa}$. For $\Upsilon^{(1)}$ and $\Upsilon^{(4)}$ and the first term of
  $\Upsilon^{(6)}$ the statement follows from uniform bounds for
  $(\mathbb{X}^T, Z (u^T), K (u^T))$ on $\mathfrak{S} \times C ([0, \infty],
  H^{1 / 2 - \kappa}) \times C ([0, \infty], H^{1 - \kappa})$ and
  multilinearity. For $\Upsilon^{(2)}$ and the first two terms of
  $\Upsilon^{(5)}$ convergence to $0$ follows from the bounds established in
  Lemma~\ref{gamma2} and Lemma~\ref{gamma5}. For $\Upsilon^{(3)}$, the last
  term of $\Upsilon^{(5)}$ and the last two terms of $\Upsilon^{(6)}$ we can
  again use uniform bounds and multilinearity as well as dominated
  convergence, thanks to Proposition~\ref{prop:squarecomm}.
\end{proof}

Going back to our particular setting recall that from
Lemma~\ref{lemma:change-of-variables} we learned that
\[ \mathcal{W}_T (f) = \inf_{u \in \mathbb{H}_a} F_T (u), \]
with
\[ F_T (u) =\mathbb{E} \left[ \Phi_T (\mathbb{W}, Z (u), K (u)) + \lambda \|
   Z_T (u) \|_{L^4}^4 + \frac{1}{2} \| l^T (u) \|_{\mathcal{H}}^2 \right], \]
where $l^T (u), Z (u), K (u)$ are functions of $u$ according to
eq.~{\eqref{eq:full-ansatz}}. This form of the functional is appropriate to
analyze the limit $T \rightarrow \infty$ and obtain the main result of the
paper, stated precisely in the following theorem which is a simple restatement
of the basic consequence of Theorem~\ref{th:gamma-lim} below.

\begin{theorem}
  \label{th:main-exact}We have
  \[ \lim_{T \rightarrow \infty} \mathcal{W}_T (f) =\mathcal{W} (f) \assign
     \inf_{u \in \mathbb{H}_a} F_{\infty} (u), \]
  where
  \[ F_{\infty} (u) =\mathbb{E} \left[ \Phi_{\infty} (\mathbb{W}, Z (u), K
     (u)) + \lambda \| Z_{\infty} (u) \|_{L^4}^4 + \frac{1}{2} \| l^{\infty}
     (u) \|_{\mathcal{H}}^2 \right], \]
  and where $\Phi_{\infty}$ and $l^{\infty}$ are defined in
  Lemma~\ref{pointwiseconv}.
\end{theorem}

In order to use $\Gamma$-convergence, we need to properly modify the
variational setting in order to guarantee enough compactness and continuity
uniformly as $T \rightarrow \infty$.

\

The analytic estimates contained in Section~\ref{section:analytic} below
allow to infer that there exists a small $\delta \in (0, 1)$, and a finite
constant $Q_T > 0$ uniformly bounded in $T$ such that
\begin{equation}
  - Q_T + (1 - \delta) \mathbb{E} \left[ \lambda \| Z_T \|_{L^4}^4 +
  \tmcolor{red}{\tmcolor{black}{\frac{1}{2}}} \| l^T (u) \|_{\mathcal{H}}^2
  \right] \leqslant F_T (u), \label{eq:lower-bound-3d}
\end{equation}
and
\begin{equation}
  F_T (u) \leqslant Q_T + (1 + \delta) \mathbb{E} \left[ \lambda \| Z_T
  \|_{L^4}^4 + \tmcolor{red}{\tmcolor{black}{\frac{1}{2}}} \| l^T (u)
  \|_{\mathcal{H}}^2 \right] \label{eq:upper-bound-3d} .
\end{equation}
As long as $T$ is finite, the original potential $V_T$ is bounded below so in
particular we have
\begin{equation}
  - C_T +\mathbb{E} \left[ \tmcolor{red}{\tmcolor{black}{\frac{1}{2}}} \| u
  \|_{\mathcal{H}}^2 \right] \leqslant F_T (u) . \label{eq:bound-u-finite-T}
\end{equation}
From this we conclude that we can relax the optimization problem and ask that
$u \in \mathbb{L}_a$, once we have established Lemma
\ref{lemma:boundcoercivity} where $\mathbb{L}_a$ is the space of predictable
processes in $\mathcal{L}$:
\[ \mathcal{W}_T (f) = \inf_{u \in \mathbb{L}_a} F_T (u) . \]
The reason of this relaxation lies in the fact that the cost terms $\| l^T (u)
\|_{\mathcal{H}}, \tmop{and} \| Z_T \|_{L^4}$ control the $\mathcal{L}$ norm
of $u$ uniformly in $T$, modulo constants depending only on $\| \mathbb{W}
\|_{\mathfrak{S}}$ which are bounded in average uniformly in $T$.

Note that eq.~{\eqref{eq:lower-bound-3d}} implies that for any sequence
$(u^T)_T$ such that $F_T (u^T)$ remains bounded we must have that also
\begin{equation}
  \sup_T \mathbb{E} [\| l^T (u^T) \|^2_{\mathcal{H}}] < \infty .
  \label{eq:coerc}
\end{equation}
To prove $\Gamma$-convergence we need to find a space with a topology which,
on the one hand is strong enough to enable to prove the $\Gamma$-liminf
inequality, and on the other hand allows to obtain enough compactness from
$F_T$. Almost sure convergence on $\mathfrak{S} \times \mathcal{L}$ would
allow for the former but is to strong for the latter. For this reason we need
a setting based on convergence in law as precised in the following definition.

\begin{definition}
  Denote by $(\mathbb{X}, u)$ be the canonical variables on $\mathfrak{S}
  \times \mathcal{L}$ and consider the space of probability measures
  \[ \mathcal{X} \assign \left\{ \mu \in \mathcal{P} (\mathfrak{S} \times
     \mathcal{L}) | \nobracket  \text{$\mu = \tmop{Law}_{\mathbb{P}}
     (\mathbb{W}, u)$ for some $u \in \mathbb{L}_a$ with $\mathbb{E}_{\mu} [\|
     u \|^2_{\mathcal{L}}] < \infty$}  \right\} . \]
  Equip $\mathcal{X}$ with the following topology: $\mu_n \rightarrow \mu$ iff
  \begin{enumeratealpha}
    \item $\mu_n$ converges to $\mu$ weakly on $\mathfrak{S} \times
    \mathcal{L}_w$,
    
    \item $\sup_n \mathbb{E}_{\mu_n} [\| u \|^2_{\mathcal{L}}] < \infty$.
  \end{enumeratealpha}
  Denote by $\overline{\mathcal{X}}$ the closure of $\mathcal{X}$ in the space
  of probability measures $\mu$ on $\mathfrak{S} \times \mathcal{L}_w$ such
  that $\mathbb{E}_{\mu} [\| u \|^2_{\mathcal{L}}] < \infty$ .
\end{definition}

Condition (b) allows to exclude pathological points in
$\overline{\mathcal{X}}$ and makes possible Lemma~\ref{lemma-boundedapprox}
below. Then
\begin{equation}
  \mathcal{W}_T (f) = \inf_{\mu \in \mathcal{X}} \breve{F}_T (\mu),
  \label{eq:functionalrestricted}
\end{equation}
where
\[ \breve{F}_T (\mu) \assign \mathbb{E}_{\mu} \left[ \Phi_T (\mathbb{X}, Z
   (u), K (u)) + \lambda \| Z_T (u) \|_{L^4}^4 + \frac{1}{2} \| l^T (u)
   \|^2_{\mathcal{H}} \right] \]
and where $\mathbb{E}_{\mu}$ denotes the expectation on $\mathfrak{S} \times
\mathcal{L}$ wrt. the probability measure $\mu$. Our first aim will be to
prove that the family $(\breve{F}_T)_T$ is indeed equicoercive on
$\overline{\mathcal{X}}$.

\begin{lemma}
  \label{lemma:tightness}Let $(\mu^T)_T$ be a family of measures on
  $\mathfrak{S} \times \mathcal{L}$ such that $\sup_T \mathbb{E}_{\mu_T} [\| u
  \|^2_{\mathcal{L}}] < \infty$. Then $(\mu^T)_T$ is tight on $\mathfrak{S}
  \times \mathcal{L}_w$, in particular there exists a subsequence converging
  in $\overline{\mathcal{X}}$.
\end{lemma}

\begin{proof}
  Observe that $\tmop{Law} (\mathbb{W})$ on $\mathfrak{S}$ is tight since
  $\mathfrak{S}$ is a separable metric space, so for any $\varepsilon > 0$, we
  can find a compact set $\mathcal{K}^1_{\varepsilon} \subset \mathfrak{S}$
  such that $\mu ((\mathfrak{S} \setminus \mathcal{K}_{\varepsilon}) \times
  \mathcal{L}) < \varepsilon / 2$. Now let $\mathcal{K}_{\varepsilon}^2
  \assign \mathcal{K}^1_{\varepsilon} \times B (0, C) \subset \mathfrak{S}
  \times \mathcal{L}$, for some large $C$ to be chosen later. Then
  $\mathcal{K}_{\varepsilon}^2$ is a compact subset of $\mathfrak{S} \times
  \mathcal{L}_w$ and
  \begin{eqnarray*}
    \mathbb{P}_{\mu^T} [(\mathbb{X}, u) \notin \mathcal{K}_{\varepsilon}^2] &
    \leqslant & \varepsilon + \frac{1}{C} \mathbb{E}_{\mu^T} [\| u
    \|^2_{\mathcal{L}}] .
  \end{eqnarray*}
  Choosing $C > \sup_T 2\mathbb{E}_{\mu^T} [\| u \|^2_{\mathcal{L}}] /
  \varepsilon$ gives tightness.
\end{proof}

\begin{lemma}
  \label{lemma:boundcoercivity}There exists a constant $C$, depending only on
  $\kappa$ and $\lambda$, such that
  \[ \mathbb{E}_{\mu^T} [\| u \|^2_{\mathcal{L}}] \lesssim C + 2 \lambda
     \mathbb{E}_{\mu^T} [\| Z_T (u) \|^4_{L^4}] +\mathbb{E}_{\mu^T} [\| l^T
     (u) \|^2_{\mathcal{H}}] . \]
\end{lemma}

\begin{proof}
  We use $\| l^T (u) \|_{\mathcal{L}} \lesssim \| l^T (u) \|_{\mathcal{H}}$ in
  the bound
  \[ \begin{array}{lll}
       \mathbb{E}_{\mu^T} [\| u \|^2_{\mathcal{L}}] & \lesssim & \lambda
       \mathbb{E}_{\mu^T} [\| \mathbb{X}^{\langle 3 \rangle}
       \|_{\mathcal{L}}^2] + \lambda \mathbb{E}_{\mu^T} [\| s \mapsto J_s
       (\mathbb{X}_s^2 \succ \theta_s Z_T (u)) \|^2_{\mathcal{L}}]
       +\mathbb{E}_{\mu^T} [\| l^T (u) \|^2_{\mathcal{H}}]\\
       & \lesssim & \lambda \mathbb{E}_{\mu^T} [\| \mathbb{X}^{\langle 3
       \rangle} \|^2_{\mathcal{L}}] + \lambda \mathbb{E}_{\mu^T} \left[
       \int^{\infty}_0 \frac{\| \mathbb{X}_s^2 \|^2_{\mathcal{\VV}^{- 1 -
       \kappa}}}{\langle s \rangle^{1 + \kappa}} \| Z_T (u) \|^2_{L^4} \mathd
       s \right]\\
       &  & +\mathbb{E}_{\mu^T} [\| l^T (u) \|^2_{\mathcal{H}}]\\
       & \lesssim & \lambda \mathbb{E}_{\mu^T} [\| \mathbb{X}^{\langle 3
       \rangle} \|^2_{\mathcal{L}}] + \frac{\lambda}{2} \mathbb{E}_{\mu^T}
       \left[ \int^{\infty}_0 \frac{\| \mathbb{X}_s^2 \|^4_{\mathcal{\VV}^{- 1
       - \kappa}}}{\langle s \rangle^{1 + \kappa}} \mathd s \right] + 2
       \lambda \mathbb{E}_{\mu^T} [\| Z_T (u) \|^4_{L^4}]\\
       &  & +\mathbb{E}_{\mu^T} [\| l^T (u) \|^2_{\mathcal{H}}] .
     \end{array} \]
  
\end{proof}

\begin{corollary}
  The family $(\breve{F}_T)_T$ is equicoercive on $\overline{\mathcal{X}}$.
\end{corollary}

\begin{proof}
  Define for some $K > 0$ large enough
  \[ \mathcal{K} = \left\{ \mu : \lambda \mathbb{E}_{\mu} [\| Z_T (u)
     \|^4_{L^4}] + \frac{1}{2} \mathbb{E}_{\mu} [\| l^T (u)
     \|^2_{\mathcal{H}}] \leqslant K \right\} \]
  From eq.~{\eqref{eq:lower-bound-3d}} we have
  \[ \lambda \mathbb{E}_{\mu} [\| Z_T (u) \|^4_{L^4}] + \frac{1}{2}
     \mathbb{E}_{\mu} [\| l^T (u) \|^2_{\mathcal{H}}] \leqslant C +
     \breve{F}_T (\mu), \]
  so
  \[ \inf_{\mu \in \mathcal{K}^c} \breve{F}_T (\mu) \geqslant K - C \]
  on the other hand from ~{\eqref{eq:upper-bound-3d}} it follows that
  \[ \sup_T \inf_{\mu \in \overline{\mathcal{X}}} \breve{F}_T (\mu) < \infty
     . \]
  So for $K$ large enough
  \[ \inf_{\mu \in \overline{\mathcal{X}}} \breve{F}_T (\mu) = \inf_{\mu \in
     \mathcal{K}} \breve{F}_T (\mu) \]
  And by \ Lemma \ref{lemma:boundcoercivity} and Lemma \ref{lemma:tightness}
  $\mathcal{K}$ is a compact set. 
\end{proof}

To be able to use the equicoercivity we will need to show that we can extend
the infimum in~{\eqref{eq:functionalrestricted}} to $\bar{\mathcal{X}}$. For
this we will first need some properties of the space $\bar{\mathcal{X}}$. In
particular we will need to show that measures with sufficiently high moments
are dense in $\overline{\mathcal{X}}$ in a way which behaves well with respect
to $\breve{F}_T$. For this we introduce some approximations which will be
useful in the sequel.

\begin{definition}
  Let $u \in \mathcal{L}$, $N \in \mathbb{N}$, and
  $(\eta_{\varepsilon})_{\varepsilon > 0}$ be a smooth Dirac sequence on
  $\Lambda$ and $(\varphi_{\varepsilon})_{\varepsilon > 0}$ be another smooth
  Dirac sequence compactly supported on $\mathbb{R}_+ \times \Lambda$. Denote
  by $\ast_{\Lambda}$ the convolution only wrt the space variable, and by
  $\ast$ the space-time convolution. Define the following approximations of
  the identity:
  \[ \begin{array}{lll}
       \left( \mathrm{reg}_{x, \varepsilon} (u) \right) & \assign & u
       \ast_{\Lambda} \eta_{\varepsilon},\\
       \left( \mathrm{reg}_{t : x, \varepsilon} (u) \right) (t) & \assign &
       e^{- \varepsilon t} u \ast \varphi_{\varepsilon} (t) = e^{- \varepsilon
       t} \int^t_0 u (t - s) \asterisk_{\Lambda} \varphi_{\varepsilon} (s)
       \mathd s.
     \end{array} \]
  Denote by
  \[ \tilde{T}^N (u) \assign \inf \left\{ t \geqslant 0 \middle| \int^t_0 \|
     u (s) \|^2_{W^{- 1 / 2 - \kappa, 3}} \mathd s \geqslant N \right\}, \]
  and
  \[ \left( \mathrm{cut}_N (u) \right) (t) \assign u (t) \mathbbm{1}_{\{t
     \leqslant \tilde{T}^N (u) \}} . \]
  Observe the following properties of these maps:
  \begin{itemize}
    \item $\mathrm{reg}_{x, \varepsilon}$ is a continuous map $\mathcal{L}_w
    \rightarrow \mathcal{H}_w$ and $\mathcal{L} \rightarrow \mathcal{H}$;
    
    \item $\mathrm{reg}_{t : x, \varepsilon}$ is a continuous map
    $\mathcal{L}_w \rightarrow \mathcal{H}$;
    
    \item $\mathrm{cut}_N$ is continuous as a map $\mathcal{L} \rightarrow B
    (0, N) \subset \mathcal{L}$;
    
    \item if $u$ is a predictable process then $\mathrm{reg}_{x, \varepsilon}
    (u)$, $\mathrm{reg}_{t : x, \varepsilon} (u)$, $\mathrm{cut}_N (u)$ will
    also be predictable.
  \end{itemize}
  Furthermore we have the bounds
  \[ \left\| \mathrm{reg}_{x, \varepsilon} (u) \right\|_{\mathcal{L}},
     \left\| \mathrm{reg}_{t : x, \varepsilon} (u) \right\|_{\mathcal{L}},
     \left\| \mathrm{cut}_N (u) \right\|_{\mathcal{L}} \leqslant \| u
     \|_{\mathcal{L}} . \]
  uniformly in $\varepsilon, N$, and for every $u \in \mathcal{L}$,
  \[ \lim_{\varepsilon \rightarrow 0} \left\| \mathrm{reg}_{x, \varepsilon}
     (u) - u \right\|_{\mathcal{L}} = \lim_{\varepsilon \rightarrow 0} \left\|
     \mathrm{reg}_{t : x, \varepsilon} (u) - u \right\|_{\mathcal{L}} =
     \lim_{N \rightarrow \infty} \left\| \mathrm{cut}_N (u) - u
     \right\|_{\mathcal{L}} = 0. \]
  With abuse of notation, for $\mu \in \mathcal{P} (\mathfrak{S} \times
  \mathcal{L})$ and $f : \mathcal{L} \rightarrow \mathcal{L}$, we let
  \[ f_{\ast} \mu = (\tmop{Id}, f)_{\ast} \mu = \tmop{Law}_{\mu} (\mathbb{X},
     f (u)) . \]
\end{definition}

\begin{remark}
  Let us briefly comment on the uses of these approximation in the sequel.
  $\mathrm{reg}_{t : x, \varepsilon}$ will be used when one wants to obtain a
  sequence of weakly convergent measures on $\mathfrak{S} \times \mathcal{H}$
  or $\mathfrak{S} \times \mathcal{L}$ from a sequence of measures weakly
  convergent on $\mathfrak{S} \times \mathcal{L}_w$. $\mathrm{reg}_{x,
  \varepsilon}$ will be used when one wants to obtain a measure on
  $\mathfrak{S} \times \mathcal{H}$ from one on $\mathfrak{S} \times
  \mathcal{L}$, while preserving the estimates on the moments of $Z (u)$ since
  $Z (u \ast_{\Lambda} \eta_{\varepsilon}) = Z (u) \ast_{\Lambda}
  \eta_{\varepsilon}$.
\end{remark}

\begin{lemma}
  \label{strongconvergence}Let $\mu \in \overline{\mathcal{X}}$. Then there
  exist $(\mu_n)_n$ in $\mathcal{X}$ such that $\mu_n \rightarrow \mu$ on
  $\mathfrak{S} \times \mathcal{L}$ (now with the norm topology) and $\sup_n
  \mathbb{E}_{\mu_n} [\| u \|^2_{\mathcal{L}}] < \infty$.
\end{lemma}

\begin{proof}
  By definition of $\overline{\mathcal{X}}$ of there exists $\tilde{\mu}_n
  \rightarrow \mu$ weakly on \ $\mathfrak{S} \times \mathcal{L}_w$. Then
  $\left( \mathrm{reg}_{t : x, \varepsilon} \right)_{\ast} \tilde{\mu}_n
  \rightarrow \left( \mathrm{reg}_{t : x, \varepsilon} \right)_{\ast} \mu$ on
  $\mathfrak{S} \times \mathcal{L}$ as $n \rightarrow \infty$, and since
  $\left( \mathrm{reg}_{t : x, \varepsilon} \right)_{\ast} \mu \rightarrow
  \mu$ weakly on $\mathfrak{S} \times \mathcal{L}$ as $\varepsilon \rightarrow
  0$, we obtain the statement by taking a diagonal sequence.
\end{proof}

\begin{lemma}
  Let $\mu_n \rightarrow \mu$ on $\mathfrak{S} \times \mathcal{L}$, such that
  $\sup_n \mathbb{E}_{\mu_n} [\| u \|^2_{\mathcal{L}}]<{\infty}$. Then
  \begin{enumerate}
    \item for every Lipschitz function f on $\mathcal{L}$, $\mathbb{E}_{\mu_n}
    [f (\nobracket u) \nobracket] \rightarrow \mathbb{E}_{\mu} [f (\nobracket
    u) \nobracket]$;
    
    \item for every Lipschitz function f on $C ( [0, \infty], L^4)$ we have
    $\mathbb{E}_{\mu_n} [f (Z (u))] \rightarrow \mathbb{E}_{\mu} [f (Z (u))]$.
  \end{enumerate}
\end{lemma}

\begin{proof}
  Let $f$ be a Lipschitz function on $\mathcal{L}$ with Lipschitz constant
  $L$. Let $\eta \in C (\mathbb{R}, \mathbb{R})$ be supported on $B (0, 2)$
  with $\eta = 1$ on $B (0, 1)$, and $\eta_N (x) = \eta (x / N)$. Then $u
  \mapsto f (u) \eta_N (\| u \|_{\mathcal{L}})$ is bounded,
  \[ \lim_{n \rightarrow \infty} \mathbb{E}_{\mu_n} [f (u) \eta_N (\| u
     \|_{\mathcal{L}})] =\mathbb{E}_{\mu} [f (u) \eta_N (\| u
     \|_{\mathcal{L}})], \]
  and \ \
  \begin{eqnarray*}
    \mathbb{E}_{\mu_n} [f (u) \eta_N (\| u \|_{\mathcal{L}})]
    -\mathbb{E}_{\mu_n} [f (u)] & = & \mathbb{E}_{\mu_n} [(f (u) \eta_N (\| u
    \|_{\mathcal{L}}) - f (u)) \mathbbm{1}_{\{ \| u \|_{\mathcal{L}} \geqslant
    N \}} \nobracket\\
    & \leqslant & \mathbb{E}_{\mu_n} [2 L \| u \|_{\mathcal{L}} 
    \mathbbm{1}_{\{ \| u \|_{\mathcal{L}} \geqslant N \}}]\\
    & \leqslant & 2 L\mathbb{E}_{\mu_n} [\| u \|^2_{\mathcal{L}} ]^{1 / 2}
    \mu_n (\| u \|_{\mathcal{L}} \geqslant N)\\
    & \leqslant & \frac{2 L}{N} \mathbb{E}_{\mu_n} [\| u \|^2_{\mathcal{L}} ]
    .
  \end{eqnarray*}
  Using that $\sup_n \mathbb{E}_{\mu_n} [\| u \|^2_{\mathcal{L}} ] < \infty$
  we have
  \begin{eqnarray*}
    \lim_{n \rightarrow \infty} \mathbb{E}_{\mu_n} [f (u)] -\mathbb{E}_{\mu}
    [f (u)] & \leqslant & | \lim_{n \rightarrow \infty} \mathbb{E}_{\mu_n} [f
    (u) \eta_N (\| u \|^2_{\mathcal{L}})] -\mathbb{E}_{\mu} [f (u) \eta_N (\|
    u \|^2_{\mathcal{L}})] |\\
    &  & + \sup_n | \mathbb{E}_{\mu_n} [f (u) \eta_N (\| u
    \|^2_{\mathcal{L}})] -\mathbb{E}_{\mu_n} [f (u)] |\\
    &  & + \sup_n | \mathbb{E}_{\mu} [f (u) \eta_N (\| u \|^2_{\mathcal{L}})]
    -\mathbb{E}_{\mu} [f (u)] | \\
    & \leqslant & \frac{4 L}{N} \sup_n \mathbb{E}_{\mu_n} [\| u
    \|^2_{\mathcal{L}} ] \lesssim N^{- 1},
  \end{eqnarray*}
  and sending $N \rightarrow \infty$ gives the statement. The second statement
  follows from the first and Lemma~\ref{lemma-Zcompact}.
\end{proof}

From the definition of $\mathcal{X}$ we have that
\[ \mathcal{W}_T (f) = \inf_{\mu \in \mathcal{X}} \breve{F}_T (\mu), \]
where
\[ \breve{F}_T (\mu) \assign \mathbb{E}_{\mu} \left[ \Phi_T (\mathbb{X}, Z
   (u), K (u)) + \lambda \| Z_T (u) \|_{L^4}^4 + \frac{1}{2} \| l^T (u)
   \|^2_{\mathcal{H}} \right] . \]
The space $\mathcal{X}$ is not necessarily closed w.r.t weak convergence of
measures. To be able to use an argument involving equicoercivity we need to
show that we can pass to the closure $\overline{\mathcal{X}}$ of $\mathcal{X}$
in the infimum. To do this we first need to prove that we can approximate
measures in $\overline{\mathcal{X}}$ by measures with bounded support in the
second marginal which are still in $\overline{\mathcal{X}}$. This is the
content of the following

\begin{lemma}
  \label{lemma-boundedapprox}Let $\mu \in \overline{\mathcal{X}}$ such that
  $E_{\mu} [\| Z_T (u) \|^4_{L^4}] + E_{\mu} [\| u \|^2_{\mathcal{L}}] <
  \infty$. For any $L > 0$ there exists $\mu_L \in \overline{\mathcal{X}}$
  such that $\| u \|_{\mathcal{L}} \leqslant L$, $\mu_L$-almost surely, $\mu_L
  \rightarrow \mu$ weakly on \ $\mathfrak{S} \times \mathcal{L}$ as $L
  \rightarrow \infty$,
  \[ \mathbb{E}_{\mu_L} [\| Z_T (u) \|^4_{L^4}] \rightarrow \mathbb{E}_{\mu}
     [\| Z_T (u) \|^4_{L^4}], \quad \text{and} \quad \mathbb{E}_{\mu_L} [\| u
     \|^2_{\mathcal{L}}] \rightarrow \mathbb{E}_{\mu} [\| u
     \|^2_{\mathcal{L}}] . \]
  Furthermore for any $\mu_L$ there exists $(\mu_{L, n})_n \subset
  \mathcal{X}$ such that $\| u \|_{\mathcal{L}} \leqslant L$, $\mu_{L,
  n}$-almost surely and $\mu_{L, n} \rightarrow \mu_L$ weakly on $\mathfrak{S}
  \times \mathcal{L}_w$.
\end{lemma}

\begin{proof}
  \
  
  {\tmstrong{Step 1}} First let us show how to approximate $\mu$ with
  $\tilde{\mu}_L$ which are defined such that $\|Z_T (u) \|_{L^4} \leqslant
  L$, $\tilde{\mu}_L$ almost surely. As $\mu \in \overline{\mathcal{X}}$,
  there exists $(\mu_n)_n \subset \mathcal{X}$ such that $\mu_n \rightarrow
  \mu$ on $\mathfrak{S} \times \mathcal{L}$ and $\sup_n \mathbb{E}_{\mu_n} [\|
  u \|^2_{\mathcal{L}}] < \infty$. Since $\mu_n \in \mathcal{X}$ there exist
  $(u^n)_n$ adapted such that $\mu_n = \tmop{Law} (\mathbb{W}, u^n)$. Define
  $\tilde{Z}^n_s \assign \mathbb{E} \left[ \int^T_0 J_t u^n_t \mathd t|
  \mathcal{F}_s \right] = \int^T_0 J_t \mathbb{E} [u^n_t | \mathcal{F}_s]
  \mathd t$. Then $\tilde{Z}$ is a martingale with continuous paths in $L^4
  (\Lambda)$. Define the stopping time $T_{L, n} = \inf \{ t \in [0, T] |\|
  \tilde{Z}^n_t \|_{L^4} \geqslant L \}$ where the inf is equal to $T$ if the
  set is empty. Observe that $\tilde{Z}_{T_{L, n}} = \int^T_0 J_t \mathbb{E}
  [u^n_t | \mathcal{F}_{T_{L, n}}] \mathd t = Z_T (u^{L, n})$ with $u^{L, n}_t
  \assign \mathbb{E} [u^n_t | \mathcal{F}_{T_{L, n}}]$ adapted, by optional
  sampling, and almost surely $\| \tilde{Z}_{T_L} \|_{L^4} \leqslant L$. Now
  set $\tilde{\mu}_{L, n} \assign \tmop{Law}_{\mathbb{P}} (\mathbb{W}, u^{L,
  n})$.
  
  {\tmstrong{Step 1.1}} (Tightness) The next goal is to show that for fixed
  $L$, we can select a suitable convergent subsequence from $(\tilde{\mu}_{L,
  n})_n$. For this we first show that $(\tilde{\mu}_{L, n})_n$ is tight on
  $\mathfrak{S} \times \mathcal{L}_w$. From the definition of $\mathcal{X}$ we
  have that $\sup_n \mathbb{E}_{\mu_n} [\| u \|^2_{\mathcal{L}}] < \infty$,
  and by construction \
  \[ \sup_n \mathbb{E}_{\tilde{\mu}_{L, n}} [\| u \|^2_{\mathcal{L}}]
     \leqslant \sup_n \mathbb{E}_{\mathbb{P}} [\| \mathbb{E} [u^n_t |
     \mathcal{F}_{T_{L, n}}] \|^2_{\mathcal{L}}] \leqslant \sup_n
     \mathbb{E}_{\mathbb{P}} [\| u^n \|^2_{\mathcal{L}}] = \sup_n
     \mathbb{E}_{\mu_n} [\| u \|^2_{\mathcal{L}}] < \infty, \]
  which gives tightness according to
  \tmcolor{black}{Lemma}~\ref{lemma:tightness}. We can then select a
  subsequence which converges on $\mathcal{L}_w$.
  
  {\tmstrong{Step 1.2}} (Bounds) Let $\tilde{\mu}_L$ be the limit of the
  sequence constructed in Step 1.1. In this step we prove bounds on the
  relevant moments of $\tilde{\mu}_L$. Let $f^M_1, f^M_2$ be sequences of
  functions on $\mathbb{R}$ which are Lipschitz, convex and monotone for every
  $N$, while for every $x \in \mathbb{R}$
  \begin{eqnarray*}
    0 \leqslant f^M_1 (x) \leqslant x^2, &  & \lim_{M \rightarrow \infty}
    f_1^M (x) = x^2,\\
    0 \leqslant f^M_2 (x) \leqslant x^4, &  & \lim_{M \rightarrow \infty}
    f_2^M (x) = x^4 .
  \end{eqnarray*}
  Then $f_1^M (\| u \|_{\mathcal{L}})$ is a lower-semi continuous positive
  function on $\mathcal{L}_w$ so by the Portmanteau lemma we have
  \[ \mathbb{E}_{\tilde{\mu}_L} [f_1^M (\| u \|_{\mathcal{L}})] \leqslant
     \liminf_{n \rightarrow \infty} \mathbb{E}_{\tilde{\mu}_{L, n}} [f_1^N (\|
     u \|_{\mathcal{L}})], \]
  and since it is also Lipschitz continuous and convex we have
  \begin{eqnarray*}
    \liminf_{n \rightarrow \infty} \mathbb{E}_{\tilde{\mu}_{L, n}} [f_1^M (\|
    u \|_{\mathcal{L}})] & = & \liminf_{n \rightarrow \infty}
    \mathbb{E}_{\mathbb{P}} [f_1^M \nobracket \| \mathbb{E} [u_n |
    \mathcal{F}_{T_{L, n}} \nobracket] \|_{\mathcal{L}})]\\
    & \leqslant & \liminf_{n \rightarrow \infty} \mathbb{E}_{\mathbb{P}}
    [f_1^M (\| u_n \|_{\mathcal{L}})] =\mathbb{E}_{\mu} [f_1^M (\| u
    \|_{\mathcal{L}})] .
  \end{eqnarray*}
  Therefore
  \begin{eqnarray*}
    \mathbb{E}_{\tilde{\mu}_L} [\| u \|^2_{\mathcal{L}}] & = & \lim_{M
    \rightarrow \infty} \mathbb{E}_{\tilde{\mu}_L} [f_1^M (\| u
    \|_{\mathcal{L}})]\\
    & \leqslant & \lim_{M \rightarrow \infty} \mathbb{E}_{\mu} [f_1^M (\| u
    \|_{\mathcal{L}})] =\mathbb{E}_{\mu} [\| u \|^2_{\mathcal{L}}] .
  \end{eqnarray*}
  Proceeding similarly for $Z$, we see that $f^N_2 (\| Z_T \|_{L^4})$ is a
  continuous function on $L^4$ bounded below, and Lipschitz-continuous and
  convex on $L^4$ so we again can estimate
  \[ \mathbb{E}_{\tilde{\mu}_L} [f_2^N (\| Z_T \|_{L^4})] = \lim_{n
     \rightarrow \infty} \mathbb{E}_{\tilde{\mu}_{L, n}} [f_2^M (\| Z_T
     \|_{L^4})], \]
  \begin{eqnarray*}
    \mathbb{E}_{\tilde{\mu}_L} [f_2^N (\| Z_T \|_{L^4})] & = & \lim_{n
    \rightarrow \infty} \mathbb{E}_{\tilde{\mu}_{L, n}} [f_2^M (\| Z_T
    \|_{L^4})]\\
    & = & \lim_{n \rightarrow \infty} \mathbb{E}_{\mathbb{P}} [f_2^M (\|
    \mathbb{E} [Z_T (u_n) | \mathcal{F}_{T_{L, n}} |] \|_{L^4})]\\
    & \leqslant & \lim_{n \rightarrow \infty} \mathbb{E}_{\mathbb{P}} [f_2^M
    (\| \nobracket Z_T (u_n)] \|_{L^4})] =\mathbb{E}_{\mu} [f_2^M (\|
    \nobracket Z_T (u_n)] \|_{L^4})] .
  \end{eqnarray*}
  and, taking $N \rightarrow \infty$, obtain
  \[ \mathbb{E}_{\tilde{\mu}_L} [\| Z_T \|^4_{L^4}] \leqslant
     \mathbb{E}_{\mu} [\| Z_T \|^4_{L^4}] . \]

  {\tmstrong{Step 1.3}} (Weak convergence) Now we prove weak convergence of
  $\tilde{\mu}_L$ to $\mu$ on $\mathfrak{S} \times \mathcal{L}$. Let $f :
  \mathfrak{S} \times \mathcal{L} \rightarrow \mathbb{R}$ be bounded and
  continuous. By dominated convergence and continuity of $f$,
  $\lim_{\varepsilon} \mathbb{E}_{\tilde{\mu}_L} \left[ f \left( \mathbb{X},
  \mathrm{reg}_{t : x, \varepsilon} (u) \right) \right]
  =\mathbb{E}_{\tilde{\mu}_L} [f (\mathbb{X}, u)]$. Using furthermore that
  $(\mathbb{X}, u) \mapsto f \left( \mathbb{X}, \mathrm{reg}_{t : x,
  \varepsilon} (u) \right)$ is continuous on $\mathfrak{S} \times
  \mathcal{L}_w$ and Lemma \ref{lemma-Zcompact} in the 5th line below, we can
  estimate
  \begin{eqnarray*}
    &  & \lim_{L \rightarrow \infty} | \mathbb{E}_{\mu} [f (\mathbb{X}, u)]
    -\mathbb{E}_{\tilde{\mu}_L} [f (\mathbb{X}, u)] |\\
    & = & \lim_{L \rightarrow \infty} \lim_{\varepsilon \rightarrow 0} \left|
    \lim_{n \rightarrow \infty} \mathbb{E}_{\mu_n} \left[ f \left( \mathbb{X},
    \mathrm{reg}_{t : x, \varepsilon} (u^n) \right) \right]
    -\mathbb{E}_{\tilde{\mu}_{L, n}} \left[ f \left( \mathbb{X},
    \mathrm{reg}_{t : x, \varepsilon} (u^n) \right) \right] \right|\\
    & = & \lim_{L \rightarrow \infty} \lim_{\varepsilon \rightarrow 0} \left|
    \lim_{n \rightarrow \infty} \mathbb{E}_{\mathbb{P}} \left[ f \left(
    \mathbb{W}, \mathrm{reg}_{t : x, \varepsilon} (u^n) \right) - f \left(
    \mathbb{W}, \mathbb{E} \left[ \mathrm{reg}_{t : x, \varepsilon} (u^n) |
    \mathcal{F}_{T_L} \nobracket \right] \right) \right] \right|\\
    & = & \lim_{L \rightarrow \infty} \lim_{\varepsilon \rightarrow 0} \left|
    \lim_{n \rightarrow \infty} \mathbb{E}_{\mathbb{P}} \left[ f \left(
    \mathbb{W}, \mathrm{reg}_{t : x, \varepsilon} (u^n) \right) - f \left(
    \mathbb{W}, \mathbb{E} \left[ \mathrm{reg}_{t : x, \varepsilon} (u^n) |
    \mathcal{F}_{T_L} \nobracket \right] \right) \mathbbm{1}_{\{ T_L < \infty
    \}} \right] \right|\\
    & \leqslant & \lim_{L \rightarrow \infty} \lim_{\varepsilon \rightarrow
    0} \left| \lim_{n \rightarrow \infty} \mathbb{E}_{\mathbb{P}} \left[ f
    \left( \mathbb{W}, \mathrm{reg}_{t : x, \varepsilon} (u^n) \right) - f
    \left( \mathbb{W}, \mathbb{E} \left[ \mathrm{reg}_{t : x, \varepsilon}
    (u^n) | \mathcal{F}_{T_L} \nobracket \right] \right) \mathbbm{1}_{\{ \|
    u^n \|_{\mathcal{L}} > c L \}} \right] \right|\\
    & \leqslant & \frac{2}{c} \| f \|_{\infty} \lim_{L \rightarrow \infty}
    \sup_n \frac{\mathbb{E} [\| u^n \|^2_{\mathcal{L}}]}{L^2} = 0.
  \end{eqnarray*}

  {\tmstrong{Step 2}} In this step we improve the approximation to have
  bounded support. Let $\mu_n \rightarrow \mu$ be the subsequence selected in
  Step 1.1. Recall that $\mu_n = \tmop{Law} (\mathbb{W}, u^n)$ with adapted
  $u^n$. Define $\tilde{Z}_t^{n, N} \assign \mathbb{E} \left[ Z_T \left(
  \mathrm{cut}_N (u) \right) \divides \mathcal{F}_t \right]$, and similarly to
  Step 1, $T_{n, L, N} \assign \inf \{t \geqslant 0 |\| \tilde{Z}^{n, N}_t
  \|_{L^4} \geqslant L\}$, set $u^{n, N, L} \assign \mathbb{E} \left[
  \mathrm{cut}_N (u) \divides \mathcal{F}_{T_{n, L, N}} \right]$. Then $\|
  u^{n, N, L} \|_{\mathcal{L}} \leqslant N$ uniformly in $n$ and
  $\mathbb{P}$-almost surely, so $\mu_{n, L, N} = \tmop{Law} (\mathbb{W},
  u^{n, N, L})$ is tight on $\mathfrak{S} \times \mathcal{L}_w$ \ and we can
  select a weakly convergent subsequence. Denote the limit by $\mu_{L, N}$.
  Now we follow the strategy from Step 1.
  
  {\tmstrong{Step 2.1}} (Bounds) Let $f_1^M$ be defined like in Step 1.2. Then
  again we have
  \begin{eqnarray*}
    \liminf_{n \rightarrow \infty} \mathbb{E}_{\mu_{n, L, N}} [f_1^M (\| u
    \|_{\mathcal{L}})] & = & \liminf_{n \rightarrow \infty}
    \mathbb{E}_{\mathbb{P}} [f_1^M \nobracket \| \mathbb{E} [\bar{u}^{n, N} |
    \mathcal{F}_{T_{n, L, N}} \nobracket] \|_{\mathcal{L}})]\\
    & \leqslant & \tmcolor{red}{\tmcolor{black}{\lim_{n \rightarrow \infty}
    \mathbb{E}_{\mathbb{P}} [f_1^M (\| \bar{u}^{n, N} \|_{\mathcal{L}})]}}\\
    & = & \mathbb{E}_{\mu_n} \left[ f_1^M \left( \left\| \mathrm{cut}_N (u)
    \right\|_{\mathcal{L}} \right) \right] \leqslant \mathbb{E}_{\mu} [f_1^M
    (\| u \|_{\mathcal{L}})] .
  \end{eqnarray*}
  It follows that
  \begin{eqnarray*}
    \mathbb{E}_{\mu_{L, N}} [\| u \|^2_{\mathcal{L}}] & = & \lim_{M
    \rightarrow \infty} \mathbb{E}_{\tilde{\mu}_{L, N}} [f_1^M (\| u
    \|_{\mathcal{L}})]\\
    & \leqslant & \lim_{M \rightarrow \infty} \liminf_{n \rightarrow \infty}
    \mathbb{E}_{\mu_{n, L, N}} [f_1^M (\| u \|_{\mathcal{L}})]\\
    & \leqslant & \lim_{M \rightarrow \infty} \mathbb{E}_{\mu} [f_1^M (\| u
    \|_{\mathcal{L}})] =\mathbb{E}_{\mu} [\| u \|^2_{\mathcal{L}}] .
  \end{eqnarray*}

  {\tmstrong{Step 2.1}} (Weak convergence) Now we prove that $\mu_{L, N}
  \rightarrow \tilde{\mu}_L$ weakly on $\mathcal{L}$. Let $f : \mathfrak{S}
  \times \mathcal{L} \rightarrow \mathbb{R}$ be bounded and continuous. By
  dominated convergence and continuity of $f$, 
  $$\lim_{\varepsilon}
  \mathbb{E}_{\tilde{\mu}_L} \left[ f \left( \mathbb{X}, \mathrm{reg}_{t : x,
  \varepsilon} (u) \right) \right] =\mathbb{E}_{\tilde{\mu}_L} [f (\mathbb{X},
  u)],$$
   and furthermore since $f \left( \mathbb{X}, \mathrm{reg}_{t : x,
  \varepsilon} (u) \right)$ is continuous on $\mathfrak{S} \times
  \mathcal{L}_w$ we have
  \begin{eqnarray*}
    &  & \lim_{N \rightarrow \infty} | \mathbb{E}_{\tilde{\mu}_L} [f
    (\mathbb{X}, u)] -\mathbb{E}_{\mu_{L, N}} [f (\mathbb{X}, u)] |\\
    & = & \lim_{N \rightarrow \infty} \lim_{\varepsilon \rightarrow 0} \left|
    \lim_{n \rightarrow \infty} \mathbb{E}_{\mu_{n, L}} \left[ f \left(
    \mathbb{X}, \mathrm{reg}_{t : x, \varepsilon} (u) \right) \right]
    -\mathbb{E}_{\tilde{\mu}_{n, L, N}} \left[ f \left( \mathbb{X},
    \mathrm{reg}_{t : x, \varepsilon} (u) \right) \right] \right|\\
    & = & \lim_{N \rightarrow \infty} \lim_{\varepsilon \rightarrow 0} \left|
    \lim_{n \rightarrow \infty} \mathbb{E}_{\mathbb{P}} \left[ f \left(
    \mathbb{W}, \mathbb{E} \left[ \mathrm{reg}_{t : x, \varepsilon} (u^n) |
    \mathcal{F}_{T_L} \nobracket \right] \right) - f \left( \mathbb{W},
    \mathbb{E} \left[ \mathrm{reg}_{t : x, \varepsilon} (\bar{u}^{n, N}) |
    \mathcal{F}_{T_{n, L, N}} \nobracket \right] \right) \right] \right|\\
    & = & \lim_{N \rightarrow \infty} \lim_{\varepsilon \rightarrow 0} 
    \left| \lim_{n \rightarrow \infty} \mathbb{E}_{\mathbb{P}} \left[ \left( f
    \left( \mathbb{W}, \mathbb{E} \left[ \mathrm{reg}_{t : x, \varepsilon}
    (u^n) | \mathcal{F}_{T_L} \nobracket \right] \right) - f \left(
    \mathbb{W}, \mathbb{E} \left[ \mathrm{reg}_{t : x, \varepsilon}
    (\bar{u}^{n, N}) | \mathcal{F}_{T_{n, L, N}} \nobracket \right] \right)
    \right) \mathbbm{1}_{\{ \tilde{T}_{n, N} < \infty \}} \right] \right|\\
    & \leqslant & \lim_{N \rightarrow \infty} \sup_{\varepsilon} \left|
    \sup_n \mathbb{E}_{\mathbb{P}} \left[ \left( f \left( \mathbb{W},
    \mathbb{E} \left[ \mathrm{reg}_{t : x, \varepsilon} (u^n) |
    \mathcal{F}_{T_L} \nobracket \right] \right) - f \left( \mathbb{W},
    \mathbb{E} \left[ \mathrm{reg}_{t : x, \varepsilon} (\bar{u}^{n, N}) |
    \mathcal{F}_{T_{n, L, N}} \nobracket \right] \right) \right)
    \mathbbm{1}_{\{ \nobracket \| u^n \|_{\mathcal{L}} \} > N \}} \right]
    \right|\\
    & \leqslant & \quad (\sup_{\mathfrak{S} \times \mathcal{L}} \divides f
    \divides) \lim_{N \rightarrow \infty} \sup_n \frac{\mathbb{E} [\| u^n
    \|^2_{\mathcal{L}}]}{N^2}\\
    & = & 0
  \end{eqnarray*}
  {\tmstrong{Step 3.}} We now put everything together. Since all $\mu_{L, N}$
  are supported on the set $\{ u : \| Z_T (u) \|_{L^4} \leqslant L \}$, weak
  convergence and Lemma~\ref{lemma-Zcompact} imply
  \[ \lim_{N \rightarrow \infty} \mathbb{E}_{\mu_{N, L}} [\| Z_T (u)
     \|^4_{L^4}] =\mathbb{E}_{\tilde{\mu}_L} [\| Z_T (u) \|^4_{L^4}] . \]
  By the Portmanteau lemma,
  \begin{equation}
    \liminf_{N \rightarrow \infty} \mathbb{E}_{\mu_{N, L}} [\| u
    \|_{\mathcal{L}}^2] \geqslant \mathbb{E}_{\tilde{\mu}_L} [\| u
    \|_{\mathcal{L}}^2], \label{liminfNL}
  \end{equation}
  and
  \[ \liminf_{L \rightarrow \infty} \mathbb{E}_{\tilde{\mu}_L} [\| u
     \|_{\mathcal{L}}^2] \geqslant \mathbb{E}_{\mu} [\| u \|_{\mathcal{L}}^2]
  \]
  which together with Step~1.2 imply $\lim_{L \rightarrow \infty}
  \mathbb{E}_{\tilde{\mu}_L} [\| u \|_{\mathcal{L}}^2] =\mathbb{E}_{\mu} [\| u
  \|_{\mathcal{L}}^2]$, and by the same argument $\lim_{L \rightarrow \infty}
  \mathbb{E}_{\tilde{\mu}_L} [\| Z_T (u) \|^4_{L^4}] =\mathbb{E}_{\mu} [\| Z_T
  (u) \|^4_{L^4}]$. For any $\delta > 0$ we can choose a $\tilde{\mu}_L$ such
  that
  \[ | \mathbb{E}_{\tilde{\mu}_L} [\| Z_T (u) \|^4_{L^4}] -\mathbb{E}_{\mu}
     [\| Z_T (u) \|^4_{L^4}] | + | \mathbb{E}_{\tilde{\mu}_L} [\| u
     \|_{\mathcal{L}}^2] -\mathbb{E}_{\mu} [\| u \|_{\mathcal{L}}^2] |
     \leqslant \delta . \]
  Since then by~{\eqref{liminfNL}}
  \[ \mathbb{E}_{\mu} [\| u \|^2_{\mathcal{L}}] \geqslant \liminf_{N
     \rightarrow \infty} \mathbb{E}_{\mu_{N, L}} [\| u \|_{\mathcal{L}}^2]
     \geqslant \mathbb{E}_{\mu} [\| u \|_{\mathcal{L}}^2] - \delta \]
  we can choose $N$ large enough so that
  \[ | \mathbb{E}_{\mu_{N, L}} [\| Z_T (u) \|^4_{L^4}] -\mathbb{E}_{\mu} [\|
     Z_T (u) \|^4_{L^4}] | + | \mathbb{E}_{\mu_{N, L}} [\| u
     \|_{\mathcal{L}}^2] -\mathbb{E}_{\mu} [\| u \|_{\mathcal{L}}^2] |
     \leqslant \delta \]
  which implies the statement of the theorem. 
\end{proof}

\begin{lemma}
  If $T < \infty$ we have\label{closurefinite}
  \[ \mathcal{W}_T (f) = \inf_{\mu \in \bar{\mathcal{X}}} \breve{F}_T (\mu) .
  \]
\end{lemma}

\begin{proof}
  To prove the claim it is enough to show that for any $\mu \in
  \overline{\mathcal{X}}$, for any $\alpha > 0$, there exists a sequence
  $\mu_n \in \mathcal{X}$ such that $\limsup_{n \rightarrow \infty}
  \breve{F}_T (\mu_n) \leqslant \breve{F}_T (\mu) + \alpha$. W.l.o.g we can
  assume that $\breve{F}_T (\mu) < \infty$. Observe that, as long as $T <
  \infty$ we can also express
  \[ \breve{F}_T (\mu) =\mathbb{E}_{\mu} \left[ \frac{1}{| \Lambda |} V_T
     (\mathbb{X}^1_T + Z_T (u)) + \tmcolor{red}{\tmcolor{black}{\frac{1}{2}}}
     \| u \|^2_{\mathcal{H}} \right], \]
  and deduce that $\mathbb{E}_{\mu} \| u \|_{\mathcal{H}}^2 < \infty$ since
  $V_T$ is bounded below at fixed $T$. By Lemma~\ref{lemma-boundedapprox}
  there exists a sequence $(\mu_L)_L \subset \overline{\mathcal{X}}$ , such
  that $\mu_L$, $\| u \|_{\mathcal{L}} \leqslant L$ almost surely, $\mu_L
  \rightarrow \mu$ on $\mathfrak{S} \times \mathcal{L}$ and
  \[ \mathbb{E}_{\mu_L} [\| Z_T (u) \|^4_{L^4}] \rightarrow
     \mathbb{E}_{\mu_{}} [\| Z_T (u) \|^4_{L^4}], \qquad \mathbb{E}_{\mu_L}
     [\| u \|^2_{\mathcal{L}}] \rightarrow \mathbb{E}_{\mu} [\| u
     \|^2_{\mathcal{L}}] . \]
  First we have to improve the regularity of $\mu_L$ to get convergence on
  $\mathfrak{S} \times \mathcal{H}_w$ but without affecting our control on the
  moments of $Z_T$, so let $\mu^{\varepsilon}_L \assign \left(
  \mathrm{reg}_{x, \varepsilon} \right)_{\ast} \mu_L$ and $\mu^{\varepsilon}
  \assign \left( \mathrm{reg}_{x, \varepsilon} \right)_{\ast} \mu$. Then
  \[ \mathbb{E}_{\mu^{\varepsilon}_L} [\| Z_T (u) \|^4_{L^4}] \rightarrow
     E_{\mu^{\varepsilon}_{}} [\| Z_T (u) \|^4_{L^4}], \qquad
     \mathbb{E}_{\mu^{\varepsilon}_L} [\| u \|^2_{\mathcal{H}}] \rightarrow
     E_{\mu^{\varepsilon}} [\| u \|^2_{\mathcal{H}}], \]
  and $\mu^{\varepsilon}_L \rightarrow \mu^{\varepsilon}$ on $\mathfrak{S}
  \times \mathcal{H}$. By continuity of $\breve{F}_T$ and the
  bound~{\eqref{eq:upper-bound-3d}}, $\breve{F}_T (\mu^{\varepsilon}_L)
  \rightarrow \breve{F}_T (\mu^{\varepsilon})$ as $L \rightarrow \infty$ and
  $\breve{F}_T (\mu^{\varepsilon}) \rightarrow \breve{F}_T (\mu)$ as
  $\varepsilon \rightarrow 0$. In particular we can find $L$ and $\varepsilon$
  such that $| \breve{F}_T (\mu^{\varepsilon}_L) - \breve{F}_T (\mu) | <
  \alpha / 2$. By Lemma~\ref{lemma-boundedapprox} there exists a sequence
  $(\mu_{n, L})_{n, L}$ such that each measure $\mu_{n, L}$ is supported on
  $\mathfrak{S} \times B (0, L)$ and $\mu_{n, L} \rightarrow \mu_L$ weakly on
  $\mathfrak{S} \times \mathcal{H}_w$. Setting $\mu^{\varepsilon, \delta}_{n,
  L} \assign \left( \mathrm{reg}_{t ; x, \delta} \right)_{\ast} \left(
  \mathrm{reg}_{x, \varepsilon} \right)_{\ast} \mu_{n, L}$ and
  $\mu^{\varepsilon, \delta}_L \assign \left( \mathrm{reg}_{t ; x, \delta}
  \right)_{\ast} \left( \mathrm{reg}_{x, \varepsilon} \right)_{\ast} \mu_L$ we
  have $\mu^{\varepsilon, \delta}_{n, L} \rightarrow \mu^{\varepsilon,
  \delta}_L$ on $\mathfrak{S} \times \mathcal{H}$ with norm topology. Then,
  for some $\chi \in C (\mathbb{R}, \mathbb{R}), \chi = 1 \tmop{on} B (0, 1)$
  supported on $B (0, 2)$ and for any $N \in \mathbb{N}$, $V_T (\mathbb{X}^1_T
  + Z_T (u)) \chi (\| \mathbb{X} \|_{\mathfrak{S}} / N), \| u
  \|^2_{\mathcal{H}}$ are continuous bounded functions on the common support
  of $\mu^{\varepsilon, \delta}_{n, L}$ and
  \begin{eqnarray*}
    &  & \lim_{n \rightarrow \infty} \left| \mathbb{E}_{\mu^{\varepsilon,
    \delta}_{n, L}} \left[ \frac{1}{| \Lambda |} V_T (\mathbb{X}^1_T + Z_T
    (u)) + \tmcolor{red}{\tmcolor{black}{\frac{1}{2}}} \| u \|^2_{\mathcal{H}}
    \right] \right. -\\
    &  & \qquad \qquad \left. -\mathbb{E}_{\mu^{\varepsilon, \delta}_L}
    \left[ \frac{1}{| \Lambda |} V_T (\mathbb{X}^1_T + Z_T (u)) +
    \tmcolor{red}{\tmcolor{black}{\frac{1}{2}}} \| u \|^2_{\mathcal{H}}
    \right] \right|\\
    & \leqslant & \lim_{n \rightarrow \infty} \left|
    \mathbb{E}_{\mu^{\varepsilon, \delta}_{n, L}} \left[ \frac{1}{| \Lambda |}
    \chi (\| \mathbb{X} \|_{\mathfrak{S}} / N) V_T (\mathbb{X}^1_T + Z_T (u))
    + \tmcolor{red}{\tmcolor{black}{\frac{1}{2}}} \| u \|^2_{\mathcal{H}}
    \right] - \right.\\
    &  & \qquad \qquad \left. -\mathbb{E}_{\mu^{\varepsilon, \delta}_L}
    \left[ \frac{1}{| \Lambda |} \chi (\| \mathbb{X} \|_{\mathfrak{S}} / N)
    V_T (\mathbb{X}^1_T + Z_T (u)) +
    \tmcolor{red}{\tmcolor{black}{\frac{1}{2}}} \| u \|^2_{\mathcal{H}}
    \right] \right|\\
    &  & + \sup_n \left| \mathbb{E}_{\mu^{\varepsilon, \delta}_{n, L}} \left[
    \frac{1}{| \Lambda |} (1 - \chi) (\| \mathbb{X} \|_{\mathfrak{S}} / N) V_T
    (\mathbb{X}^1_T + Z_T (u)) \right] \right|\\
    &  & + \sup_n \left| \mathbb{E}_{\mu^{\varepsilon, \delta}_L} \left[
    \frac{1}{| \Lambda |} (1 - \chi) (\| \mathbb{X} \|_{\mathfrak{S}} / N) V_T
    (\mathbb{X}^1_T + Z_T (u)) \right] \right|\\
    & \leqslant & \sup_n | \mathbb{E}_{\mu^{\varepsilon, \delta}_{n, L}}
    [\mathbbm{1}_{\{ \| \mathbb{X} \|_{\mathfrak{S}} \geqslant N \}} V_T
    (\mathbb{X}^1_T + Z_T (u))] |\\
    &  & + | \mathbb{E}_{\mu^{\varepsilon, \delta}_L} [\mathbbm{1}_{\{ \|
    \mathbb{X} \|_{\mathfrak{S}} \geqslant N \}} V_T (\mathbb{X}^1_T + Z_T
    (u))] |\\
    \text{(by {\eqref{eq:upper-bound-3d}})} & \lesssim & 2 \sup_n |
    (\mu^{\varepsilon, \delta}_{n, L} (\| \mathbb{X} \|_{\mathfrak{S}}
    \geqslant N)) \mathbb{E}_{\mu^{\varepsilon, \delta}_{n, L}} [\| \mathbb{X}
    \|^p_{\mathfrak{S}} + L^8] |\\
    & \lesssim & 2 \frac{\mathbb{E}_{\mu^{\varepsilon, \delta}_{n, L}} [\|
    \mathbb{X} \|_{\mathfrak{S}}]}{N} \mathbb{E}_{\mu^{\varepsilon,
    \delta}_{n, L}} [\| \mathbb{X} \|^p_{\mathfrak{S}} + L^8]\\
    \text{(as $N \rightarrow \infty$)} & \rightarrow & 0.
  \end{eqnarray*}
  and by dominated convergence (since $\mu^{\varepsilon, \delta}_L$ is
  supported on $\mathfrak{S} \times B (0, L)$) we can find a $\delta$ such
  that $| \breve{F}_T (\mu^{\varepsilon, \delta}_L) - \breve{F}_T
  (\mu^{\varepsilon}_L) | < \alpha / 2$ which proves the statement. \ 
\end{proof}

The proof of Lemma~\ref{closurefinite} does not apply when $T = \infty$.
However also in this case the functional
\begin{equation}
  \breve{F}_{\infty} (\mu) \assign \mathbb{E}_{\mu} \left[ \Phi_{\infty}
  (\mathbb{X}, Z (u), K (u)) + \lambda \| Z_{\infty} (u) \|_{L^4}^4 +
  \frac{1}{2} \| l^{\infty} (u) \|^2_{\mathcal{H}} \right], \label{eq:f-infty}
\end{equation}
has a well defined meaning, so we can investigate the relation between the two
variational problems (on $\mathcal{X}$ and on $\bar{\mathcal{X}}$) when the
cutoff is removed. An additional difficulty derives from the fact that
approximating the drift $u$ we might destroy the regularity of $l^{\infty}
(u)$, since now $l^{\infty} (u)$ needs to be more regular than $u$, contrary
to the finite $T$ case. To resolve this problem we need to be able to smooth
out the remainder without destroying the bound on $Z_T (u)$. To do so
smoothing $l^{\infty} (u)$ directly, and constructing a corresponding new $u$
will not work, since $l^{\infty} (u)$ by itself does not give enough control
on $u$ and $Z (u)$. However we are still able to prove the following lemma by
regularizing an ``augmented'' version of $l^{\infty} (u)$.

\begin{lemma}
  \label{lemma:remreg}There exists a family of continuous functions
  $\mathrm{\mathrm{rem}}_{\varepsilon} : \mathcal{L} \mapsto \mathcal{L}$,
  which are also continuous $\mathcal{L}_w \rightarrow \mathcal{L}_w$, such
  that for any $T \in [0, \infty]$,
  \begin{eqnarray*}
    \left\| \mathrm{rem}_{\varepsilon} (u) \right\|_{\mathcal{L}} & \lesssim &
    \| \mathbb{X} \|_{\mathfrak{S}} + \| u \|_{\mathcal{L}},\\
    \left\| Z_T \left( \mathrm{rem}_{\varepsilon} (u) \right) \right\|_{L^4} &
    \lesssim & \| \mathbb{X} \|_{\mathfrak{S}} + \| Z_T (u) \|_{L^4},\\
    \left\| l^{\infty} \left( \mathrm{rem}_{\varepsilon} (\mathbb{X}, u)
    \right) \right\|^2_{\mathcal{H}} & \lesssim_{\varepsilon} & (1 + \|
    \mathbb{X} \|_{\mathfrak{S}})^4 + \| Z_{\infty} (u) \|^4_{L^4} + \| u
    \|^2_{\mathcal{L}},
  \end{eqnarray*}
  and $\left\| l^{\infty} \left( \mathrm{rem}_{\varepsilon} (\mathbb{X}, u)
  \right) \right\|_{\mathcal{H}}$ depends continuously on $(\mathbb{X}, u) \in
  \mathfrak{S} \times \mathcal{L}$. Furthermore
  \[ \mathrm{rem}_{\varepsilon} (\mathbb{X}, u) \rightarrow u \text{\text{in
     \ensuremath{\mathcal{L}}}}, \]
  and if $l^{\infty} (u) \in \mathcal{H}$
  \[ l^{\infty} (\tmop{rem}_{\varepsilon} (\mathbb{X}, u)) \rightarrow
     l^{\infty} (u)  \text{\text{in} $H$} \text{ as } \varepsilon \rightarrow
     0. \]
\end{lemma}

\begin{proof}
  \ Let $\mathbb{X}^2 = \mathcal{U}_{\leqslant} \mathbb{X}^2 +
  \mathcal{U}_{>} \mathbb{X}^2$ be the decomposition introduced in
  Section~\ref{sec:bounds}, and observe that for any $c > 0$ we can easily
  modify it to ensure that $\| \mathcal{U}_{>} \mathbb{X}^2
  \|_{\mathcal{\VV}^{- 1 - \kappa}} < c$, almost surely for any $\mu \in
  \overline{\mathcal{X}}$ and for any $1 \leqslant p < \infty$,
  $\mathbb{E}_{\mu} \left[ \| \mathcal{U}_{\leqslant} \mathbb{X}^2
  \|^p_{\mathcal{\VV}^{- 1 + \kappa}} \right] \leqslant C$ where $C$ depends
  on $| \Lambda |, \kappa, c, p$. Now set $\tilde{l}_t (u) = - \lambda J_s
  (\mathcal{U}_{\leqslant} \mathbb{X}_t^2 \succ Z_t^{\flat} (u)) + l^{\infty}
  (u)$. Then $u$ satisfies
  \[ u_s = - \lambda \mathbb{X}_s^{\langle 3 \rangle} - \lambda J_s
     (\mathcal{U}_{>} \mathbb{X}_s^2 \succ Z^{\flat}_s) + \tilde{l}_s (u) . \]
  From this equation we can see that, like in Section~\ref{sec:bounds},
  \[ \| u \|^2_{\mathcal{L}} \lesssim \lambda \| \mathbb{W}^{\langle 3
     \rangle} \|^2_{\mathcal{L}} + \lambda \int^{\infty}_0 \frac{1}{\langle s
     \rangle^{1 + \varepsilon}} \| \mathcal{U}_{>} \mathbb{X}_s^2
     \|^2_{\mathcal{\VV}^{- 1 - \kappa}} \mathd s \| u \|^2_{\mathcal{L}} +_{}
     \| \tilde{l}_s (u) \|^2_{\mathcal{L}}, \]
  and choosing, $c$ small enough we get
  \begin{equation}
    \| u \|_{\mathcal{L}} \lesssim \lambda \| \mathbb{X}^{\langle 3 \rangle}
    \|_{\mathcal{L}} + \| \tilde{l}_s (u) \|_{\mathcal{L}} . \label{bound-u-l}
  \end{equation}
  And similarly we observe that
  \[ Z_T (u) = - \lambda \mathbb{X}_T^{[3]} - \lambda \int^T_0 J^2_s
     (\mathcal{U}_{>} \mathbb{X}_s^2 \succ Z^{\flat}_s) \mathd s + Z_T
     (\tilde{l} (u)), \]
  so again with $c$ small enough and since $Z^{\flat}_s = \theta_s Z_T$ for $s
  \leqslant T$:
  \begin{equation}
    \| Z_T (u) \|_{L^4} \lesssim \lambda \| \mathbb{X}_T^{[3]} \|_{L^4} + \|
    Z_T (\tilde{l} (u)) \|_{L^4} . \label{bound-Z-u-l}
  \end{equation}
  Conversely, it is not hard to see that we have the inequalities
  \begin{eqnarray}
    \| Z_T (\tilde{l} (u)) \|_{L^4} & \lesssim & \lambda \| \mathbb{X}_T^{[3]}
    \|_{L^4} + \| Z_T (u) \|_{L^4}, \label{bound-Z-l-u} 
  \end{eqnarray}
  and
  \begin{equation}
    \| \tilde{l} (u) \|_{\mathcal{L}} \lesssim \lambda \| \mathbb{X}^{\langle
    3 \rangle} \|_{\mathcal{L}} + \| u \|_{\mathcal{L}} . \label{bound-l-u}
  \end{equation}
  Clearly the map $(\mathbb{X}, u) \mapsto (\mathbb{X}, \tilde{l} (u))$ is
  continuous as a map \ $\mathfrak{S} \times \mathcal{L} \rightarrow
  \mathcal{L}$ and using Lemma~\ref{lemma-Zcompact} also as a map
  $\mathfrak{S} \times \mathcal{L}_w \rightarrow \mathfrak{S} \times
  \mathcal{L}_w$ , and the inverse is clearly continuous $\mathfrak{S} \times
  \mathcal{L} \rightarrow \mathfrak{S} \times \mathcal{L}$. We now show that
  it is also continuous as a map $\mathfrak{S} \times \mathcal{L}_w
  \rightarrow \mathfrak{S} \times \mathcal{L}_w$. Assume that $\tilde{l} (u^n)
  \rightarrow l (u)$ weakly, since then $\| l (u^n) \|_{\mathcal{L}}$ bounded,
  this implies by~{\eqref{bound-u-l}} that also $\| u^n \|_{\mathcal{L}}$ is
  bounded, and so we can select a weakly convergent subsequence, converging to
  $u^{\star}$. Then $u^{\star}$ solves the equation
  \[ u^{\star}_s = - \lambda \mathbb{X}_s^{\langle 3 \rangle} - \lambda J_s
     (\mathcal{U}_{>} \mathbb{X}_s^2 \succ Z^{\flat}_s (u^{\star})) +
     \tilde{l}_s (u), \]
  (which can be seen for example by testing with some $h \in
  \mathcal{L}^{\ast}$) which implies that $u^{\star} = u$ (e.g. by Gronwall).
  Now define $\tmop{rem}_{\varepsilon} (u)$ to be the solution to the equation
  \[ \tmop{rem}_{\varepsilon} (u) = - \lambda \mathbb{X}_s^{\langle 3
     \rangle} - \lambda J_s (\mathcal{U}_{>} \mathbb{X}_s^2 \succ Z^{\flat}_s
     (\tmop{rem}_{\varepsilon} (u))) + \mathrm{reg}_{x, \varepsilon}
     (\tilde{l}_s (u)) . \]
  Then by the properties discussed above $u \mapsto \tmop{rem}_{\varepsilon}
  (u)$ is continuous in both the weak and the norm topology, we also have
  from~{\eqref{bound-u-l}} and~{\eqref{bound-l-u}} that
  \begin{eqnarray*}
    \| \tmop{rem}_{\varepsilon} (u) \|_{\mathcal{L}} & \lesssim & \lambda \|
    \mathbb{X}^{\langle 3 \rangle} \|_{\mathcal{L}} + \| u \|_{\mathcal{L}},
  \end{eqnarray*}
  from~{\eqref{bound-Z-u-l}} we have
  \[ \| Z_T  (\tmop{rem}_{\varepsilon} (u)) \|_{L^4} \lesssim \lambda \|
     \mathbb{X}_T^{[3]} \|_{L^4} + \| Z_T (u) \|_{L^4}, \label{bound-Z-reg} \]
  and by definition of $\tmop{rem}_{\varepsilon} (u)$
  \begin{eqnarray}
    \| \tilde{l} (\tmop{rem}_{\varepsilon} (u)) \|_{\mathcal{H}} & = & \left\|
    \mathrm{reg}_{x, \varepsilon} (\tilde{l} (u)) \right\|_{\mathcal{H}}
    \nonumber\\
    & \lesssim_{\varepsilon} & \lambda \| \mathbb{X}^{\langle 3 \rangle}
    \|_{\mathcal{L}} + \| u \|_{\mathcal{L}} \label{bound-l-u-reg} 
  \end{eqnarray}
  Now observe that
  \begin{eqnarray*}
    \| l^{\infty} (\tmop{rem}_{\varepsilon} (u)) \|^2_{\mathcal{H}} & \lesssim
    & \| s \mapsto \lambda J_s (\mathcal{U}_{\leqslant} \mathbb{X}_s^2 \succ
    Z_s^{\flat} (\tmop{rem}_{\varepsilon} (u))) \|_{\mathcal{H}} + \|
    \tilde{l} (\tmop{rem}_{\varepsilon} (u)) \|^2_{\mathcal{H}}\\
    & \lesssim_{\varepsilon} & \lambda \int \frac{1}{\langle s \rangle^{1 +
    \kappa}} \| \mathcal{U}_{\leqslant} \mathbb{X}_s^2 \|^2_{\mathcal{\VV}^{-
    1 + \kappa}} \| Z^{\flat}_s (\tmop{rem}_{\varepsilon} (u)) \|^2_{L^4}
    \mathd s + \lambda \| \mathbb{X}^{\langle 3 \rangle} \|^2_{\mathcal{L}} +
    \| u \|^2_{\mathcal{L}}\\
    & \lesssim & \lambda (1 + \| \mathbb{X} \|_{\mathfrak{S}})^4 + \|
    Z_{\infty} (\tmop{rem}_{\varepsilon} (u)) \|^4_{L^4} + \| u
    \|^2_{\mathcal{L}}\\
    & \lesssim & \lambda (1 + \| \mathbb{X} \|_{\mathfrak{S}})^4 + \|
    Z_{\infty} (u) \|^4_{L^4} + \| u \|^2_{\mathcal{L}} .
  \end{eqnarray*}
  Observing that also $\| \lambda J_s (\mathcal{U}_{\leqslant} \mathbb{X}_t^2
  \succ Z_t^{\flat} (\tmop{rem}_{\varepsilon} (u))) \|_{\mathcal{H}}$ depends
  continuously on $(\mathbb{X}, u)$ (both in the weak and strong topology on
  $\mathcal{L}$) gives the statement.
\end{proof}

\begin{lemma}
  For any $\mu \in \overline{\mathcal{X}}$ such that $\breve{F}_{\infty} (\mu)
  < \infty$ there exists a sequence of measures $\mu_L \in
  \overline{\mathcal{X}}$ such that \label{infty-approximation}
  \begin{enumerateroman}
    \item For any $p < \infty$,
    \begin{equation}
      \mathbb{E}_{\mu_L} [\| u \|^p_{\mathcal{L}}] +\mathbb{E}_{\mu_L} [\|
      l^{\infty} (u) \|^p_{\mathcal{H}}] < \infty, \label{inftymomentbound}
    \end{equation}
    \item $\mu_L \rightarrow \mu$ weakly on $\mathfrak{S} \times \mathcal{L}$
    and $\tmop{Law}_{\mu_L} (l^{\infty} (u)) \rightarrow \tmop{Law}_{\mu}
    (l^{\infty} (u))$ weakly on $\mathcal{H}$,
    
    \item
    \[ \lim_{L \rightarrow \infty} \breve{F}_{\infty} (\mu_L) =
       \breve{F}_{\infty} (\mu), \]
    \item For any $\mu_L$ there exists a sequence $\mu_{n, L} \in \mathcal{X}$
    such that
    \begin{equation}
      \sup_n (\mathbb{E}_{\mu_{n, L}} [\| u \|^p_{\mathcal{L}}]
      +\mathbb{E}_{\mu_{n, L}} [\| l^{\infty} (u) \|^p_{\mathcal{H}}]) <
      \infty, \label{inftymomentbound2}
    \end{equation}
    $\mu_{n, L} \rightarrow \mu_L$ weakly on $\mathfrak{S} \times
    \mathcal{L}_w$ and $\tmop{Law}_{\mu_{n, L}} (l^{\infty} (u)) \rightarrow
    \tmop{Law}_{\mu} (l^{\infty} (u))$ weakly on $\mathcal{H}_w$.
  \end{enumerateroman}
\end{lemma}

\begin{proof}
  By Lemma~\ref{lemma-boundedapprox} there exists a sequence $\mu_{\tilde{L}}
  \rightarrow \mu$ weakly on $\mathfrak{S} \times \mathcal{L}$ such that
  \[ \mathbb{E}_{\mu_{\tilde{L}}} [\| Z_T (u) \|^4_{L^4}] \rightarrow
     \mathbb{E}_{\mu_{}} [\| Z_T (u) \|^4_{L^4}], \qquad
     \mathbb{E}_{\mu_{\tilde{L}}} [\| u \|^2_{\mathcal{L}}] \rightarrow
     \mathbb{E}_{\mu} [\| u \|^2_{\mathcal{L}}], \]
  and $\mu_{\tilde{L}}$ is supported on $\mathfrak{S} \times B (0, \tilde{L})
  \subset \mathfrak{S} \times \mathcal{L}$. Now set
  $\mu^{\varepsilon}_{\tilde{L}} \assign (\tmop{rem}_{\varepsilon})_{\ast}
  \mu_{\tilde{L}}$. Then $\mu^{\varepsilon}_{\tilde{L}} \rightarrow
  \mu^{\varepsilon} \assign (\tmop{rem}_{\varepsilon})_{\ast} \mu$ on
  $\mathfrak{S} \times \mathcal{L}$ and by the bounds from Lemma
  \ref{lemma:remreg} also $\mathbb{E}_{\mu^{\varepsilon}_{\tilde{L}}} [\| Z_T
  (u) \|^4_{L^4}] \rightarrow \mathbb{E}_{\mu^{\varepsilon}_{}} [\| Z_T (u)
  \|^4_{L^4}]$ and $\mathbb{E}_{\mu^{\varepsilon}_{\tilde{L}}} [\| l^{\infty}
  (u) \|^2_{\mathcal{H}}] \rightarrow \mathbb{E}_{\mu^{\varepsilon}} [\|
  l^{\infty} (u) \|^2_{\mathcal{H}}]$. The bounds from
  Lemma~\ref{lemma:remreg} imply also $\mathbb{E}_{\mu^{\varepsilon}} [\| Z_T
  (u) \|^4_{L^4}] \rightarrow \mathbb{E}_{\mu_{}} [\| Z_T (u) \|^4_{L^4}]$,
  $\mathbb{E}_{\mu^{\varepsilon}} [\| l^{\infty} (u) \|^2_{\mathcal{H}}]
  \rightarrow \mathbb{E}_{\mu} [\| l^{\infty} (u) \|^2_{\mathcal{H}}]$, and
  furthermore
  \begin{eqnarray*}
    \mathbb{E}_{\mu^{\varepsilon}_{\tilde{L}}} [\| u \|^p_{\mathcal{L}}] &
    \lesssim & \mathbb{E}_{\mu_{\tilde{L}}} (\| \mathbb{X} \|^p_{\mathfrak{S}}
    + \| u \|^p_{\mathcal{L}})\\
    & \lesssim & \mathbb{E}_{\mu_{\tilde{L}}} (\| \mathbb{X}
    \|^p_{\mathfrak{S}}) + \tilde{L}^p,
  \end{eqnarray*}
  and similarly
  \begin{eqnarray*}
    \mathbb{E}_{\mu^{\varepsilon}_{\tilde{L}}} [\| l^{\infty} (u)
    \|^p_{\mathcal{L}}] & \lesssim_{\varepsilon} &
    \mathbb{E}_{\mu_{\tilde{L}}} (\| \mathbb{X} \|^p_{\mathfrak{S}} + \| u
    \|^p_{\mathcal{L}})\\
    & \lesssim & \mathbb{E}_{\mu_{\tilde{L}}} (\| \mathbb{X}
    \|^p_{\mathfrak{S}}) + \tilde{L}^p,
  \end{eqnarray*}
  and by continuity of $\breve{F}_{\infty}$ and~{\eqref{eq:upper-bound-3d}} we
  are also able to deduce that we can find $\varepsilon$ small enough and
  $\tilde{L}$ large enough depending on $\varepsilon$ such that $|
  \breve{F}_{\infty} (\mu^{\varepsilon}) - \breve{F}_{\infty} (\mu) | < 1 / 2
  L$ and $| \breve{F}_{\infty} (\mu_L^{\varepsilon}) - \breve{F}_{\infty}
  (\mu^{\varepsilon}) | < 1 / 2 L$. Choosing $\mu_L =
  \mu^{\varepsilon}_{\tilde{L}}$ we obtain the first three points of the
  Lemma. For the fourth point recall that from Lemma~\ref{lemma-boundedapprox}
  we have sequences $\mu_{n, \tilde{L}} \rightarrow \mu_{\tilde{L}}$ weakly on
  $\mathfrak{S} \times \mathcal{L}_w$, and $\mu_{n, \tilde{L}} \in
  \mathcal{X}$, which have support in $\mathfrak{S} \times B (0, \tilde{L})$
  and since $\tmop{rem}_{\varepsilon}$ is continuous on $\mathfrak{S} \times
  \mathcal{L}_w$ setting $\mu^{\varepsilon}_{n, \tilde{L}} \assign \left(
  \mathrm{reg}_{\varepsilon} \right)_{\ast} \mu_{n, \tilde{L}}$ we obtain the
  desired sequence. 
\end{proof}

\begin{lemma}
  If $T = \infty$ we have
  \[ \inf_{\mu \in \mathcal{X}} \breve{F}_{\infty} (\mu) = \inf_{\mu \in
     \bar{\mathcal{X}}} \breve{F}_{\infty} (\mu) . \]
\end{lemma}

\begin{proof}
  One can now proceed very similarly to the proof of
  Lemma~\ref{closurefinite}. Let $\mu \in \overline{\mathcal{X}}$ such that
  $\breve{F}_{\infty} (\mu) < \infty$. By Lemma~\ref{infty-approximation}, for
  any $L, \mu \in \bar{\mathcal{X}}$, there exists a $\mu_L$ such that $|
  \breve{F}_{\infty} (\mu) - \breve{F}_{\infty} (\mu_L) | < 1 / L$, and a
  sequence $(\mu_{n, L})_n$ such that $\mu_{n, L} \in \mathcal{X}$, $\mu_{n,
  L} \rightarrow \mu_L$ weakly on $\mathfrak{S} \times \mathcal{L}_w$, and
  such that~{\eqref{inftymomentbound2}} is satisfied. Define
  $\mu^{\varepsilon, \delta}_{n, L} \assign \tmop{Law} \left( \mathbb{X},
  \tmop{rem}_{\varepsilon} \left( \mathrm{reg}_{t : x, \varepsilon} (u)
  \right) \right)$, and observe that now $\mu^{\varepsilon, \delta}_{n, L}
  \rightarrow \mu^{\varepsilon, \delta}_L$ on $\mathfrak{S} \times
  \mathcal{L}$, $\tmop{Law}_{\mu^{\varepsilon, \delta}_{n, L}} (\mathbb{X},
  l^{\infty} (u)) \rightarrow \tmop{Law}_{\mu^{\varepsilon, \delta}_L}
  (\mathbb{X}, l^{\infty} (u))$ on $\mathfrak{S} \times \mathcal{H}$, and that
  we have $\sup_n (\mathbb{E}_{\mu^{\varepsilon, \delta}_{n, L}} [\| u
  \|^p_{\mathcal{L}}] +\mathbb{E}_{\mu^{\varepsilon, \delta}_{n, L}} [\|
  l^{\infty} (u) \|^p_{\mathcal{H}}]) < \infty$. Then for some $\chi \in C
  (\mathbb{R}, \mathbb{R}), \chi = 1 \tmop{on} B (0, 1)$ supported on $B (0,
  2)$, for any $N \in \mathbb{N}$, the function
  \begin{eqnarray*}
    &  & \chi \left( \frac{\| \mathbb{X} \|_{\mathfrak{S}} + \| u
    \|_{\mathcal{L}} + \| l^{\infty} (u) \|_{\mathcal{H}}}{N} \right) \left(
    \Phi_{\infty} (\mathbb{X}, Z (u), K (u)) + \lambda \| Z_{\infty} (u)
    \|_{L^4}^4 + \frac{1}{2} \| l^{\infty} (u) \|^2_{\mathcal{H}} \right)\\
    & = & \widetilde{\chi_{}}_N (\mathbb{X}, u) \left( \Phi_{\infty}
    (\mathbb{X}, Z (u), K (u)) + \lambda \| Z_{\infty} (u) \|_{L^4}^4 +
    \frac{1}{2} \| l^{\infty} (u) \|^2_{\mathcal{H}} \right)
  \end{eqnarray*}
  is bounded and continuous on $\mathfrak{S} \times \mathcal{L}$, and so by
  weak convergence
  \begin{eqnarray*}
    &  & \lim_{n \rightarrow \infty} | \breve{F}_{\infty} (\mu^{\varepsilon,
    \delta}_{n, L}) - \breve{F}_{\infty} (\mu^{\varepsilon, \delta}_L) |\\
    & \leqslant & \lim_{n \rightarrow \infty} \left|
    \mathbb{E}_{\mu^{\varepsilon, \delta}_{n, L}} \left[ \widetilde{\chi_{}}_N
    (\mathbb{X}, u) \left( \Phi_{\infty} (\mathbb{X}, Z (u), K (u)) + \lambda
    \| Z_{\infty} (u) \|_{L^4}^4 + \frac{1}{2} \| l^{\infty} (u)
    \|^2_{\mathcal{H}} \right) \right] - \right.\\
    &  & \qquad \left. -\mathbb{E}_{\mu^{\varepsilon, \delta}_L} \left[
    \widetilde{\chi_{}}_N (\mathbb{X}, u) \left( \Phi_{\infty} (\mathbb{X}, Z
    (u), K (u)) + \lambda \| Z_{\infty} (u) \|_{L^4}^4 + \frac{1}{2} \|
    l^{\infty} (u) \|^2_{\mathcal{H}} \right) \right] \right|\\
    &  & + \sup_n \mathbb{E}_{\mu^{\varepsilon, \delta}_{n, L}} \left[ \left|
    (1 - \widetilde{\chi_{}}_N (\mathbb{X}, u)) \left( \Phi_{\infty}
    (\mathbb{X}, Z (u), K (u)) + \lambda \| Z_{\infty} (u) \|_{L^4}^4 +
    \frac{1}{2} \| l^{\infty} (u) \|^2_{\mathcal{H}} \right) \right| \right]\\
    &  & +\mathbb{E}_{\mu^{\varepsilon, \delta}_L} \left[ \left| (1 -
    \widetilde{\chi_{}}_N (\mathbb{X}, u)) \left( \Phi_{\infty} (\mathbb{X}, Z
    (u), K (u)) + \lambda \| Z_{\infty} (u) \|_{L^4}^4 + \frac{1}{2} \|
    l^{\infty} (u) \|^2_{\mathcal{H}} \right) \right| \right]\\
    & \leqslant & 2 \sup_n \mathbb{E}_{\mu^{\varepsilon, \delta}_{n, L}}
    \left[ \mathbbm{1}_{\{ \| \mathbb{X} \|_{\mathfrak{S}} + \| u
    \|_{\mathcal{L}} + \| l^{\infty} (u) \|_{\mathcal{H}} > N \}} \left|
    \Phi_{\infty} (\mathbb{X}, Z (u), K (u)) + \lambda \| Z_{\infty} (u)
    \|_{L^4}^4 + \frac{1}{2} \| l^{\infty} (u) \|^2_{\mathcal{H}} \right|
    \right]\\
    & \lesssim & \sup_n \left( \mu^{\varepsilon, \delta}_{n, L} (\|
    \mathbb{X} \|_{\mathfrak{S}} + \| u \|_{\mathcal{L}} + \| l^{\infty} (u)
    \|_{\mathcal{H}} > N) \mathbb{E}_{\mu^{\varepsilon, \delta}_{n, L}} [\|
    \mathbb{X} \|^p_{\mathfrak{S}} + \| u \|^8_{\mathcal{L}} + \| l^{\infty}
    (u) \|^4_{\mathcal{H}}] \right)\\
    & \lesssim & \sup_n \left( \frac{1}{N} \mathbb{E}_{\mu^{\varepsilon,
    \delta}_{n, L}} [\| \mathbb{X} \|_{\mathfrak{S}} + \| u \|_{\mathcal{L}} +
    \| l^{\infty} (u) \|_{\mathcal{H}}] \mathbb{E}_{\mu^{\varepsilon,
    \delta}_{n, L}} [\| \mathbb{X} \|^p_{\mathfrak{S}} + \| u
    \|^8_{\mathcal{L}} + \| l^{\infty} (u) \|^4_{\mathcal{H}}] \right)\\
    & \rightarrow & 0 \text{as $N \rightarrow \infty$}
  \end{eqnarray*}
  As we can find $\varepsilon, \delta$ such that $| \breve{F}_{\infty}
  (\mu^{\varepsilon, \delta}_L) - \breve{F}_{\infty} (\mu_L) | < 1 / L$ we can
  conclude. 
\end{proof}

Finally we can state the key result of this section.

\begin{theorem}
  \label{th:gamma-lim}The family $(\breve{F}_T)_T$ $\Gamma$--converges to
  $\breve{F}_{\infty}$ on $\overline{\mathcal{X}}$. Therefore
  \[ \lim_T \mathcal{W}_T (f) = \lim_T \inf_{\mu \in \bar{\mathcal{X}}}
     \breve{F}_T (\mu) = \inf_{\mu \in \bar{\mathcal{X}}} \breve{F}_{\infty}
     (\mu) =\mathcal{W} (f) . \]
\end{theorem}

\begin{proof}
  In order to establish $\Gamma$-convergence consider a sequence $\mu^T
  \rightarrow \mu$ in $\overline{\mathcal{X}}$. We need to prove that
  $\liminf_{T \rightarrow \infty}  \breve{F}_T (\mu^T) \geqslant
  \breve{F}_{\infty} (\mu)$. It is enough to prove this statement for a
  subsequence, the full statement follows from the fact that every sequence
  has a subsequence satisfying the inequality. Take a subsequence (not
  relabeled) such that
  \begin{equation}
    \sup_T  \breve{F}_T (\mu^T) < \infty . \label{eq:unif-bound}
  \end{equation}
  If there is no such subsequence there is nothing to prove. Otherwise
  tightness for the subsequence follows like in the proof of equicoercivity.
  Then invoking the Skorokhod representation theorem of~{\cite{jakubowski}} we
  can extract a subsequence (again, not relabeled) and find random variables
  $(\tilde{\mathbb{X}}^T, \tilde{u}^T)_T$ and $(\tilde{\mathbb{X}},
  \tilde{u})$ on some probability space $(\tilde{\Omega}, \tilde{\mathbb{P}})$
  such that $\tmop{Law}_{\tilde{\mathbb{P}}} (\tilde{\mathbb{X}}^T,
  \tilde{u}^T) = \mu^T$, $\tmop{Law}_{\tilde{\mathbb{P}}} (\tilde{\mathbb{X}},
  \tilde{u}) = \mu$ and almost surely $\tilde{\mathbb{X}}^T \rightarrow
  \tilde{\mathbb{X}}$ in $\mathfrak{S}$, $\tilde{u}^T \rightarrow \tilde{u}$
  in $\mathcal{L}_w$. Note that $\tilde{l}^T \assign l^T
  (\tilde{\mathbb{X}}^T, \tilde{u}^T) \rightarrow l \assign l^{\infty}
  (\tilde{\mathbb{X}}, u)$ in $\mathcal{L}_w$ and
  using~{\eqref{eq:unif-bound}} we deduce that the almost sure convergence
  $l^T \rightarrow l$ in $\mathcal{H}_w$, maybe modulo taking another
  subsequence, again not relabeled. Note that, by our analytic estimates
  (which hold pointwise on the probability space) we have
  \[ \Phi_T (\tilde{\mathbb{X}}^T, Z (\tilde{u}^T), K (\tilde{u}^T)) + \lambda
     \| Z_T (\tilde{u}^T) \|_{L^4}^4 +
     \tmcolor{red}{\tmcolor{black}{\frac{1}{2}}} \| l^T (\tilde{u}^T)
     \|^2_{\mathcal{H}} + H (\tilde{\mathbb{X}}^T) \geqslant 0, \]
  for some $L^1 (\tilde{\mathbb{P}})$ random variable $H
  (\tilde{\mathbb{X}}^T)$ such that $\mathbb{E}_{\tilde{\mathbb{P}}} [H
  (\tilde{\mathbb{X}}^T)] =\mathbb{E} [H (\mathbb{W})]$. Fatou's lemma and
  Lemma~\ref{pointwiseconv} then give
  \[ \liminf_{T \rightarrow \infty}  \breve{F}_T (\mu^T) = \liminf_{T
     \rightarrow \infty} \mathbb{E}_{\tilde{\mathbb{P}}} \left[ \Phi_T
     (\tilde{\mathbb{X}}^T, Z (\tilde{u}^T), K (\tilde{u}^T)) + \lambda \| Z_T
     (\tilde{u}^T) \|_{L^4}^4 + \tmcolor{red}{\tmcolor{black}{\frac{1}{2}}} \|
     l^T \|^2_{\mathcal{H}} \right] \]
  \[ = \liminf_{T \rightarrow \infty} \mathbb{E}_{\tilde{\mathbb{P}}} \left[
     \Phi_T (\tilde{\mathbb{X}}^T, Z (\tilde{u}^T), K (\tilde{u}^T)) + \lambda
     \| Z_T (\tilde{u}^T) \|_{L^4}^4 +
     \tmcolor{red}{\tmcolor{black}{\frac{1}{2}}} \| l^T \|^2_{\mathcal{H}} + H
     (\tilde{\mathbb{X}}^T) \right] -\mathbb{E} [H (\mathbb{W})] \]
  \[ \geqslant \mathbb{E}_{\tilde{\mathbb{P}}} \liminf_{T \rightarrow \infty}
     \left[ \Phi_T (\tilde{\mathbb{X}}^T, Z (\tilde{u}^T), K (\tilde{u}^T)) +
     \lambda \| Z_T (\tilde{u}^T) \|_{L^4}^4 +
     \tmcolor{red}{\tmcolor{black}{\frac{1}{2}}} \| l^T \|^2_{\mathcal{H}} + H
     (\tilde{\mathbb{X}}^T) \right] -\mathbb{E} [H (\mathbb{W})] \]
  \[ \geqslant \mathbb{E}_{\tilde{\mathbb{P}}}  \left[ \Phi_{\infty}
     (\tilde{\mathbb{X}}, Z (\tilde{u}), K (\tilde{u})) + \| Z_{\infty}
     (\tilde{u}) \|_{L^4}^4 + \tmcolor{red}{\tmcolor{black}{\frac{1}{2}}} \|
     l^{\infty} (\tilde{u}) \|^2_{\mathcal{H}} \right] = \breve{F}_{\infty}
     (\mu), \]
  which is the $\Gamma$-liminf inequality. Now all that remains is
  constructing a recovery sequence, for this we can again assume w.l.o.g that
  $\breve{F}_{\infty} (\mu) < \infty$. From Lemma \ref{infty-approximation}
  there is $\mu_L$ such that $| \breve{F}_{\infty} (\mu) - \breve{F}_{\infty}
  (\mu_L) | < \frac{1}{L}$ and {\eqref{inftymomentbound}} is satisfied. Then
  choosing $\mu_L^T = \tmop{Law}_{\mu_L} (\mathbb{X}, \mathbbm{1}_{\{ t
  \leqslant T \}} u_t)$ we obtain that $l^T (\mathbbm{1}_{\{ t \leqslant T
  \}} u_t) = \mathbbm{1}_{\{ t \leqslant T \}} l^{\infty} (u)$, so $\| l^T
  (\mathbbm{1}_{\{ t \leqslant T \}} u_t) \|_{\mathcal{H}} \leqslant \|
  l^{\infty} (u) \|_{\mathcal{H}}$, and $\| Z_T (\mathbbm{1}_{\{ t \leqslant T
  \}} u_t) \|^4_{L^4} = \| Z_T (u) \|^4 \leqslant \| u \|^4_{\mathcal{L}}$,
  which is integrable by {\eqref{inftymomentbound}}. By dominated convergence
  and Lemma~\ref{pointwiseconv} we obtain $\lim_{T \rightarrow \infty}
  \breve{F}_T (\mu^T_L) = \breve{F}_{\infty} (\mu_L)$. Extracting a suitable
  diagonal sequence gives the recovery sequence.
\end{proof}

\section{Analytic estimates\label{section:analytic}}

In this section we collect a series of analytic estimate which together allow
to establish the pointwise bounds~{\eqref{eq:lower-bound-3d}}
and~{\eqref{eq:upper-bound-3d}} and the continuity required for
Lemma~\ref{pointwiseconv}. First of all note that
\begin{equation}
  \begin{array}{lll}
    \| K_t \|^2_{H^{1 - \kappa}} & \lesssim & \lambda^2 \int^t_0
    \frac{1}{\langle t \rangle^{1 + \delta}} \| \mathbb{W}^2_s \|^2_{B_{4,
    \infty}^s} \mathd s \| Z_T \|^2_{L^4} + \int^t_0 \| l_s \|_{L^2}^2 \mathd
    s\\
    & \lesssim & \lambda^3 \left( \int^t_0 \frac{1}{\langle t \rangle^{1 +
    \delta}} \| \mathbb{W}^2_s \|^2_{B_{4, \infty}^s} \mathd s \right)^2 +
    \lambda \| Z_T \|^4_{L^4} + \int^t_0 \| l_s \|_{L^2}^2 \mathd s,
  \end{array} \label{eq:control-K}
\end{equation}
which implies that quadratic functions of the norm $\| K_t \|_{H^{1 -
\kappa}}$ with small coefficient can always be controlled, uniformly in $[0,
\infty]$, by the coercive term
\[ \lambda \bint Z^4_T + \frac{1}{2} \int^{\infty}_0 \| l_s \|_{L^2}^2 \mathd
   s. \]
\begin{lemma}
  For any small $\varepsilon > 0$ there exists $\delta > 0$ such
  that\label{firstterm}
  \[ | \Upsilon^{(1)}_T | \leqslant C (\varepsilon, \delta) E (\lambda) Q_T +
     \varepsilon \| K_T \|^2_{H^{1 - \delta}} + \varepsilon \lambda \| Z_T
     \|^4_{L^4} . \]
\end{lemma}

\begin{proof}
  By Proposition~\ref{adjointparaproduct},
  \begin{equation}
    \begin{array}{l}
      \lambda \left| \bint (\mathbb{W}^2_T \succ K_T) K_T - \bint
      (\mathbb{W}^2_T \circ K_T) K_T \right|\\
      \quad \lesssim \lambda \| \mathbb{W}^2_T \|_{B_{7, \infty}^{- 9 / 8}} 
      \| K_T \|^2_{B_{7 / 3, 2}^{9 / 16}} \lesssim \lambda \| \mathbb{W}^2_T
      \|_{B_{7, \infty}^{- 9 / 8}}  \| K_T \|^2_{B_{7 / 3, 7 / 3}^{5 / 8}}\\
      \quad \lesssim \lambda \| \mathbb{W}^2_T \|_{B_{7, \infty}^{- 9 / 8}} 
      \| K_T \|^{10 / 7}_{H^{7 / 8}} \| K_T \|^{4 / 7}_{B_{4, 4}^0}\\
      \quad \lesssim \lambda^6 \| \mathbb{W}^2_T \|^7_{B_{7, \infty}^{- 9 /
      8}} + \| K_T \|^{10 / 7}_{H^{7 / 8}} + \lambda \| K_T \|^4_{L^4} .
    \end{array} \label{commutator1}
  \end{equation}
  By Proposition~\ref{paraproductestimate},
  \begin{eqnarray*}
    \left| \lambda \bint (\mathbb{W}^2_T \prec K_T) K_T \right| & \lesssim &
    \lambda \| \mathbb{W}^2_T \|_{B_{7, \infty}^{- 9 / 8}}  \| K_T \|^2_{B_{7
    / 3, 2}^{9 / 16}}
  \end{eqnarray*}
  which is estimated in the same way and finally,
  \begin{eqnarray*}
    \left| \lambda^2 \bint (\mathbb{W}^2_T \prec \mathbb{W}_T^{[3]}) K_T
    \right| & \lesssim & \lambda^2 \| \mathbb{W}^2_T \|_{B_{4, 4}^{- 1 -
    \delta / 2}}  \|\mathbb{W}_T^{[3]} \|_{B_{4, 4}^{^{1 / 2 - \delta / 2}}}
    \| K_T \|_{H^{1 / 2 + \delta}}\\
    & \leqslant & C (\delta) \lambda^4 \left( \| \mathbb{W}^2_T \|_{B_{4,
    4}^{- 1 - \delta / 2}} \|\mathbb{W}_T^{[3]} \|_{B_{4, 4}^{^{1 / 2 - \delta
    / 2}}} \right)^2 + \delta \| K_T \|^2_{H^{1 / 2 + \delta}} .
  \end{eqnarray*}
\end{proof}

\begin{lemma}
  For any small $\varepsilon > 0$ there exists $\delta > 0$ such
  that\label{gamma2}
  \begin{eqnarray*}
    | \Upsilon^{(2)}_T | & \leqslant & T^{- \delta} (C (\varepsilon, \delta) E
    (\lambda) Q_T + \varepsilon \| K \|_{H^{1 - \delta}} + \varepsilon \lambda
    \| Z_T \|_{L^4})
  \end{eqnarray*}
\end{lemma}

\begin{proof}
  Using the spectral support properties of the various terms we observe that
  \[ \| \mathbb{W}^2_T \|_{B_{p, q}^{- 1 + \delta}} \lesssim \| \mathbb{W}^2_T
     \|_{B_{p, q}^{- 1 + \delta}} T^{2 \delta}, \]
  and
  \[ T^{2 \delta} \| Z_T - Z^{\flat}_T \|_{L^2} \lesssim \| Z_T - Z^{\flat}_T
     \|_{H^{2 \delta}} \lesssim \| Z_T - Z^{\flat}_T \|^{\frac{2 \delta}{1 / 2
     - \delta}}_{H^{1 / 2 - \delta}} \| Z_T - Z^{\flat}_T \|^{\frac{1 / 2 - 3
     \delta}{1 / 2 - \delta}}_{L^2} \]
  \[ \lesssim \| Z_T \|^{\frac{2 \delta}{1 / 2 - \delta}}_{H^{1 / 2 - \delta}}
     \| Z_T \|^{\frac{1 / 2 - 3 \delta}{1 / 2 - \delta}}_{L^2}, \]
  where we used also interpolation and the $L^2$ bound $\| Z^{\flat}_T
  \|_{L^2} \lesssim \| Z_T \|_{L^2}$. We recall also that
  \begin{equation}
    Z_T = K_T + \lambda \mathbb{W}_T^{[3]} . \label{eq:bound-Z-K}
  \end{equation}
  Therefore we estimate as follows

  \begin{eqnarray*}
    \lambda \bint (\mathbb{W}^2_T \succ (Z_T - Z^{\flat}_T) ) K_T & = &
    \lambda \bint (\mathbb{W}^2_T \succ (K_T - K^{\flat}_T) ) K_T + \lambda^2
    \bint (\mathbb{W}^2_T \succ (\mathbb{W}_T^{[3]} -\mathbb{W}_T^{[3],
    \flat}) ) K_T
  \end{eqnarray*}
  For the second term we can estimate
  \begin{eqnarray*}
    \lambda^2 \bint (\mathbb{W}^2_T \succ (\mathbb{W}_T^{[3]}
    -\mathbb{W}_T^{[3], \flat}) ) K_T & \lesssim & \lambda^2 \| W^2_T
    \|_{B_{4, \infty}^{- 1 + \delta}} \| \mathbb{W}_T^{[3]}
    -\mathbb{W}_T^{[3], \flat} \|_{B_{4, 2}^0} \| K_T \|_{H^{1 - \delta}}\\
    & \lesssim & \lambda^2 T^{- \delta} \| W^2_T \|_{B_{4, \infty}^{- 1 -
    \delta}} \| \mathbb{W}_T^{[3]} \|_{B^{3 \delta}_{4, 2}} \| K_T \|_{H^{1 -
    \delta}}
  \end{eqnarray*}
  for the first term we get
  \begin{eqnarray*}
    \lambda \bint (\mathbb{W}^2_T \succ (K_T - K^{\flat}_T) ) K_T & \lesssim &
    \lambda \| W_T^2 \|_{B_{7, \infty}^{- 1 / 2 - \delta}} \| K_T \|_{B_{7 /
    3, 2}^0} \| K_T \|_{B_{7 / 3, 2}^{1 / 2 + \delta}}\\
    & \lesssim & \lambda \| W_T^2 \|_{B_{7, \infty}^{- 1 - \delta}} T^{1 / 2}
    T^{- 1 / 2 - \delta} \| K_T \|^2_{B_{7 / 3, 2}^{1 / 2 + \delta}}\\
    & \lesssim & \lambda T^{- \delta} \| W_T^2 \|_{B_{7, \infty}^{- 1 -
    \delta}} \| K_T \|^2_{B_{7 / 3, 2}^{1 / 2 + \delta}},
  \end{eqnarray*}
  which we can again estimate like in Lemma \ref{firstterm}.
\end{proof}

\begin{lemma}
  For any small $\varepsilon > 0$ there exists $\delta > 0$ such that
  \begin{eqnarray*}
    | \Upsilon^{(3)}_T | & \leqslant & C (\varepsilon, \delta) E (\lambda) Q_T
    + \varepsilon \sup_{0 \leqslant t \leqslant T} \| K_t \|_{H^{1 -
    \delta}}^2 + \varepsilon \lambda \| Z_T \|^4_{L^4} .
  \end{eqnarray*}
\end{lemma}

\begin{proof}
  First note that $\dot{\theta}_t (\mathD) = (\langle \mathD \rangle / t^2)
  \dot{\theta} (\langle \mathD \rangle / t)$. In particular $Z^{\flat}_t$ is
  spectrally supported in an annulus with inner radius $t / 4$ and outer
  radius $t / 2$. Then for any $\theta \in [0, 1]$
  \[ \| \dot{Z}^{\flat}_t \|_{B_{p, q}^{s + \theta}} = \left\| \dot{\theta}
     \left( \frac{\langle \mathD \rangle}{t} \right) \frac{\langle \mathD
     \rangle}{t^2} Z_T \right\|_{B_{p, q}^{s + \theta}} \lesssim \left\|
     \dot{\theta} \left( \frac{\langle \mathD \rangle}{t} \right)
     \frac{\langle \mathD \rangle^{1 + \theta}}{t^{2 + \theta}} Z_T
     \right\|_{B_{p, q}^{s + \theta}} \lesssim \frac{\| Z_T \|_{B_{p,
     q}^s}}{\langle t \rangle^{1 + \theta}} . \]
  By Proposition~\ref{paraproductestimate}, for any $\varepsilon > 0$ there
  exists $\delta > 0$ such that
  
  \[ \begin{array}{l}
       \left| \lambda \int_0^T \bint (\mathbb{W}^2_t \succ \dot{Z}^{\flat}_t)
       K_t \mathd t \right| \lesssim \lambda \int_0^T \| \mathbb{W}^2_t
       \|_{B^{- 1 + \delta}_{6, \infty}} \| \dot{Z}^{\flat}_t \|_{B^0_{3, 2}}
       \| K_t \|_{H^{1 - \delta}} \mathd t\\
       \quad \lesssim \lambda \int_0^T \| \mathbb{W}^2_t \|_{B^{- 1 +
       \delta}_{6, \infty}} \| Z_T \|_{B^{3 \delta}_{3, 2}} \| K_t \|_{H^{1 -
       \delta}} \frac{\mathd t}{\langle t \rangle^{1 + 3 \delta}}\\
       \quad \lesssim \lambda \| Z_T \|_{B^{4 \delta}_{3, 3}} \sup_{0
       \leqslant t \leqslant T} \| K_t \|_{H^{1 - \delta}} \int_0^T \|
       \mathbb{W}^2_t \|_{B^{- 1 + \delta}_{6, \infty}} \frac{\mathd
       t}{\langle t \rangle^{1 + \delta}}\\
       \quad \lesssim \lambda \| Z_T \|^{1 / 2}_{H^{1 / 2 - \delta}} \| Z_T
       \|^{1 / 2}_{B_{4, 4}} \sup_{0 \leqslant t \leqslant T} \| K_t \|_{H^{1
       - \delta}} \int_0^T \| \mathbb{W}^2_t \|_{B^{- 1 + \delta}_{6, \infty}}
       \frac{\mathd t}{\langle t \rangle^{1 + \delta}}\\
       \quad \lesssim \lambda \| Z_T \|^{1 / 2}_{L^4} \sup_{0 \leqslant t
       \leqslant T} \| K_t \|^{3 / 2}_{H^{1 - \delta}} \int_0^T \|
       \mathbb{W}^2_t \|_{B^{- 1 + \delta}_{6, \infty}} \frac{\mathd
       t}{\langle t \rangle^{1 + \delta}}\\
       \quad + \lambda^{3 / 2} \| Z_T \|^{1 / 2}_{L^4} \sup_{0 \leqslant t
       \leqslant T} \| K_t \|_{H^{1 - \delta}} \|\mathbb{W}_T^{[3]} \|^{1 /
       2}_{H^{4 \delta}} \int_0^T \| \mathbb{W}^2_t \|_{B^{- 1 + \delta}_{6,
       \infty}} \frac{\mathd t}{\langle t \rangle^{1 + \delta}}
     \end{array} \]
  
  and again
  \[ \begin{array}{l}
       \lambda \| Z_T \|^{1 / 2}_{L^4} \sup_{0 \leqslant t \leqslant T} \| K_t
       \|^{3 / 2}_{H^{1 - \delta}} \int_0^T \| \mathbb{W}^2_t \|_{B^{- 1 +
       \delta}_{7, \infty}} \frac{\mathd t}{\langle t \rangle^{1 + \delta}}\\
       \quad \leqslant C \lambda^7 \int_0^T \| \mathbb{W}^2_t \|^8_{B^{- 1 +
       \delta}_{7, \infty}} \frac{\mathd t}{\langle t \rangle^{1 + \delta}} +
       \varepsilon \sup_{0 \leqslant t \leqslant T} \| K_t \|^2_{H^{1 -
       \delta}} + \varepsilon \lambda \| Z_T \|^4_{L^4} .
     \end{array} \]
  While
  \begin{eqnarray*}
    &  & \lambda^{3 / 2} \| Z_T \|^{1 / 2}_{L^4} \sup_{0 \leqslant t
    \leqslant T} \| K_t \|_{H^{1 - \delta}} \|\mathbb{W}_T^{[3]} \|^{1 /
    2}_{H^{4 \delta}} \int_0^T \| \mathbb{W}^2_t \|_{B^{- 1 + \delta}_{6,
    \infty}} \frac{\mathd t}{\langle t \rangle^{1 + \delta}}\\
    & \leqslant & C \lambda^{11 / 3} \int_0^T \| \mathbb{W}^2_t \|^{8 /
    3}_{B^{- 1 + \delta}_{7, \infty}} \frac{\mathd t}{\langle t \rangle^{1 +
    \delta}} \|\mathbb{W}_T^{[3]} \|^{8 / 6}_{H^{4 \delta}} + \sup_{0
    \leqslant t \leqslant T} \| K_t \|^2_{H^{1 - \delta}} + \lambda \| Z_T
    \|_{L^4} .
  \end{eqnarray*}
\end{proof}

\begin{lemma}
  For any small $\varepsilon > 0$ there exists $\delta > 0$ such that
  \[ | \Upsilon^{(4)}_T | \leqslant C (\varepsilon, \delta) E (\lambda) Q_T +
     \varepsilon \| K_T \|^2_{H^{1 - \delta}} + \varepsilon \lambda \| Z_T
     \|^4_{L^4} . \]
\end{lemma}

\begin{proof}
  Using Lemma~\ref{lemma:bound-cubic} we establish that
  \[ \left| \lambda \bint W_T K^3_T \right| \leqslant E (\lambda) \| W_T
     \|^K_{W^{- 1 / 2 - \varepsilon, p}} + \delta (\| K_T \|^2_{H^{1 -
     \varepsilon}} + \lambda \| K_T \|^4_{L^4}) . \]
  Next, we can write,
  \[ \lambda^3 \left| \bint W_T (\mathbb{W}_T^{[3]})^2 K_T \right| \lesssim
     \lambda^3 \left| \bint W_T (\mathbb{W}_T^{[3]} \succ \mathbb{W}_T^{[3]})
     K_T \right| + \lambda^3 \| W_T \|_{B_{6, \infty}^{- 1 / 2 - \delta}} \|
     \mathbb{W}_T^{[3]} \|^2_{B_{6, 4}^{- 1 / 2 - \delta}} \| K_T \|_{H^{1 -
     \varepsilon}} . \]
  which can be easily estimated by Young's inequality. Decomposing
  \[ W_T (\mathbb{W}_T^{[3]} \succ \mathbb{W}_T^{[3]}) = W_T \succ
     (\mathbb{W}_T^{[3]} \succ \mathbb{W}_T^{[3]}) + W_T \prec
     (\mathbb{W}_T^{[3]} \succ \mathbb{W}_T^{[3]}) + W_T \circ
     (\mathbb{W}_T^{[3]} \succ \mathbb{W}_T^{[3]}) . \]
  We can estimate the first two terms by
  \[ \lambda^3 \left| \bint W_T \succ (\mathbb{W}_T^{[3]} \succ
     \mathbb{W}_T^{[3]}) K_T \right| \lesssim \lambda^3 \| W_T \|_{B_{6,
     \infty}^{- 1 / 2 - \delta}} \| \mathbb{W}_T^{[3]} \|^2_{B_{6, 2}^0} \|
     K_T \|_{H^{1 - \varepsilon}}, \]
  and
  \[ \lambda^3 \left| \bint W_T \prec (\mathbb{W}_T^{[3]} \succ
     \mathbb{W}_T^{[3]}) K_T \right| \lesssim \lambda^3 \| W_T \|_{B_{6, 2}^{-
     1 / 2 - \delta}} \| \mathbb{W}_T^{[3]} \|^2_{B_{6, \infty}^0} \| K_T
     \|_{H^{1 - \varepsilon}} . \]
  Young's inequality gives then the appropriate result. For the final term we
  use Proposition~\ref{commutatorestimate} to get
  \[ \begin{array}{l}
       \lambda^3 \left| \bint W_T \circ (\mathbb{W}_T^{[3]} \succ
       \mathbb{W}_T^{[3]}) K_T \right|\\
       \quad \lesssim \lambda^3 \left| \bint \mathbb{W}_T^{[3]}
       \mathbb{W}_T^{1 \circ [3]} K_T \right| + \lambda^3 \| W_T \|_{B_{4,
       \infty}^{- 1 / 2 - \delta}} \| \mathbb{W}_T^{[3]} \|^2_{B_{4, 2}^{- 1 /
       2 - \delta}} \| K_T \|_{H^{1 - \delta}}\\
       \quad \lesssim \lambda^3 \| \mathbb{W}_T^{[3]} \|_{B_{4, \infty}^{1 / 2
       - \delta}} \| \mathbb{W}_T^{1 \circ [3]} \|_{B_{4, 2}^{- \delta}} \|
       K_T \|_{H^{1 - \delta}} + \lambda^3 \| W_T \|_{B_{4, \infty}^{- 1 / 2 -
       \delta}} \| \mathbb{W}_T^{[3]} \|^2_{B_{4, 2}^{- 1 / 2 - \delta}} \|
       K_T \|_{H^{1 - \delta}}\\
       \quad \lesssim \lambda^6 C (\delta, \varepsilon) [\| \mathbb{W}_T^{[3]}
       \|_{B_{4, \infty}^{1 / 2 - \delta}} \| \mathbb{W}_T^{1 \circ [3]}
       \|_{B_{4, 2}^{- \delta}} + \| W_T \|_{B_{4, \infty}^{- 1 / 2 - \delta}}
       \| \mathbb{W}_T^{[3]} \|^2_{B_{4, 2}^{- 1 / 2 - \delta}}]^2 +
       \varepsilon \| K_T \|^2_{H^{1 - \delta}} .
     \end{array} \]
  For the last term we estimate
  \begin{eqnarray*}
    \left| \lambda^2 \bint (W_T \mathbb{W}_T^{[3]}) K^2_T \right| & \lesssim &
    \lambda^2 \| W_T \mathbb{W}_T^{[3]} \|_{B^{- 1 / 2 - \delta}_{7,
    \infty}}  \| K_T \|^2_{B_{7 / 3, 2}^{1 / 2 + \delta}},
  \end{eqnarray*}
  which can be estimated like in Lemma \ref{firstterm} after we observe that
  \[ \| W_T \mathbb{W}_T^{[3]} \|_{B^{- 1 / 2 - \delta}_{7, \infty}}
     \leqslant \| W_T \succ \mathbb{W}_T^{[3]} \|_{B^{- 1 / 2 - \delta}_{7,
     \infty}} + \| W_T \circ \mathbb{W}_T^{[3]} \|_{B^{- 1 / 2 -
     \delta}_{7, \infty}} + \| W_T \prec \mathbb{W}_T^{[3]} \|_{B^{- 1 / 2
     - \delta}_{7, \infty}} \]
  \[ \lesssim \| W_T \|_{B^{- 1 / 2 - \delta}_{14, \infty}} \|
     \mathbb{W}_T^{[3]} \|_{B^0_{14, \infty}} + \| \mathbb{W}_T^{1 \circ
     [3]} \|_{B^{- \delta}_{7, \infty}} \]
  and use Lemma~\ref{lemma:renormalized-terms} to bound $\mathbb{W}_T^{1 \circ
  [3]}$.
\end{proof}

\begin{lemma}
  Assume that
  \[ \sup_T \frac{| \gamma_T |}{\langle T \rangle^{1 / 4}} + \int^T_0 \frac{|
     \gamma_t | \mathd t}{\langle t \rangle^{5 / 4}} < \infty . \]
  Then for any small $\varepsilon > 0$ there exists $\delta > 0$ such
  that\label{gamma5}
  \[ | \Upsilon^{(5)}_T | \leqslant C_{\varepsilon} E (\lambda) \left[ \frac{|
     \gamma_T |}{\langle T \rangle^{1 / 4}} + \int^T_0 \frac{| \gamma_t |
     \mathd t}{\langle t \rangle^{5 / 4}} \right]^2 + \varepsilon \| Z_T
     \|^2_{H^{1 / 2 - \delta}} + \varepsilon \lambda \| Z_T \|^4_{L^4} . \]
\end{lemma}

\begin{proof}
  We can estimate
  \[ \left| \lambda^2 \gamma_T \bint Z^{\flat}_T (Z_T - Z^{\flat}_T) \right|
     \leqslant \lambda^2 | \gamma_T | \| Z^{\flat}_T \|_{L^2} \| Z_T -
     Z^{\flat}_T \|_{L^2} \lesssim \lambda^2 \frac{| \gamma_T |}{\langle T
     \rangle^{1 / 4}} \| Z^{\flat}_T \|_{L^2} \| Z_T - Z^{\flat}_T \|_{H^{1 /
     4}}, \]
  and
  \[ \left| \lambda^2 \gamma_T \bint (Z_T - Z^{\flat}_T)^2 \right| \leqslant
     \lambda^2 | \gamma_T | \| Z_T - Z^{\flat}_T \|^2_{L^2} \lesssim \lambda^2
     \frac{| \gamma_T |}{\langle T \rangle^{1 / 4}} \| Z^{\flat}_T - Z_T
     \|_{L^2} \| Z_T - Z^{\flat}_T \|_{H^{1 / 4}} . \]
  These bounds imply that both of them remain bounded provided $\gamma_T$ does
  not grow too fast in $T$ which is indeed insured by
  Lemma~\ref{lemma:renormalized-terms}. For the last term we can apply the
  estimate
  \[ \left| \lambda^2 \int^T_0 \bint \gamma_t Z^{\flat}_t \dot{Z}_t^{\flat}
     \mathd t \right| \leqslant \lambda^2 \int^T_{0} | \gamma_t | \|
     Z^{\flat}_t \|_{L^2}  \| \dot{Z}_t^{\flat} \|_{L^2} \mathd t \lesssim
     \lambda^2 \| Z_T \|_{L^2}  \| Z_T \|_{H^{1 / 4}} \int^T_0 \frac{|
     \gamma_t | \mathd t}{\langle t \rangle^{5 / 4}} . \]
  Again, after we have fixed $\gamma_t$ below, we will see this to be bounded.
  Collecting these bounds we get
  \[ | \Upsilon^{(5)}_T | \lesssim C_{\varepsilon} \lambda^7 \left[ \frac{|
     \gamma_T |}{\langle T \rangle^{1 / 4}} + \int^T_0 \frac{| \gamma_t |
     \mathd t}{\langle t \rangle^{5 / 4}} \right]^2 + \lambda \varepsilon \|
     Z_T \|^4_{L^4} + \varepsilon \| Z_T \|^2_{H^{1 / 2 - \delta}} . \]
\end{proof}

\begin{lemma}
  For any small $\varepsilon > 0$ there exists a $\delta > 0$ such that
  \begin{eqnarray*}
    | \Upsilon^{(6)}_T | & \leqslant & C (\varepsilon, \delta) E (\lambda) Q_T
    + \varepsilon \| K_T \|^2_{H^{1 - \delta}} + \varepsilon \lambda \| Z_T
    \|^4_{L^4} .
  \end{eqnarray*}
\end{lemma}

\begin{proof}
  We start by observing that
  \[ \lambda^2 \left| \bint (\mathbb{W}^2_T \circ \mathbb{W}_T^{[3]} + 2
     \gamma_T W_T) K_T \right| \lesssim \lambda^2 \| \mathbb{W}^{2 \diamond
     [3]}_T \|_{W^{- 1 / 2 - \varepsilon, 2}} \| K_T \|_{W^{1 / 2 +
     \varepsilon, 2}} . \]
  and using Lemma~\ref{lemma:renormalized-terms} and
  eq.~{\eqref{eq:control-K}} we have this term under control. Next split
  \[ \left| \frac{\lambda^2}{2} \mathbb{E} \int_0^T \bint [(J_t
     (\mathbb{W}^2_t \succ Z^{\flat}_t))^2 + 2 \dot{\gamma}_t (Z^{\flat}_t)^2]
     \mathd t \right| \]
  \[ \lesssim \frac{\lambda^2}{2} \left| \int^T_0 \bint ((J_t \mathbb{W}^2_t
     \succ Z^{\flat}_t))^2 - (J_t \mathbb{W}^2_t \circ J_t \mathbb{W}^2_t)
     (Z^{\flat}_t)^2 \mathd t \right| + \lambda^2 \left| \int_0^T \bint
     \mathbb{W}^{\langle 2 \rangle \diamond \langle 2 \rangle}_t
     (Z^{\flat}_t)^2 \mathd t \right| \]
  and note that $\langle t \rangle^{1 / 2} J_t$ satisfies the assumptions of
  Proposition~\ref{paraproductleibniz} with $m = - 1$. Therefore
  \[ \| J_t (\mathbb{W}^2_t \succ Z^{\flat}_t) - (J_t \mathbb{W}^2_t) \succ
     Z^{\flat}_t \|_{H^{1 / 4 - 2 \delta}} \lesssim \langle t \rangle^{- 1 /
     2} \| \mathbb{W}^2_t \|_{B_{6, \infty}^{- 1 - \delta}} \| Z^{\flat}_t 
     \|_{B_{3, 3}^{- 1 / 4 - \delta}}, \]
  and by Proposition~\ref{multiplierestimate},
  \[ \| J_t (\mathbb{W}^2_t \succ Z^{\flat}_t) \|_{H^{- 2 \delta}} + \| J_t
     (\mathbb{W}^2_t \succ Z^{\flat}_t) \|_{H^{- 2 \delta}} \lesssim \langle t
     \rangle^{- 1 / 2 - \delta} \| \mathbb{W}^2_t \|_{B_{6, \infty}^{- 1 -
     \delta}} \| Z^{\flat}_t  \|_{B_{3, 3}^0} . \]
  and by binomial formula
  \[ \left| \frac{\lambda^2}{2} \int_0^T \bint (J_t (\mathbb{W}^2_t \succ
     Z^{\flat}_t))^2 \mathd t - \frac{\lambda^2}{2} \int_0^T \bint ((J_t
     \mathbb{W}^2_t) \succ Z^{\flat}_t)^2 \mathd t \right| \]
  \[ \lesssim \lambda^2 \sup_{t \leqslant T}  [\| Z^{\flat}_t  \|_{B_{3, 3}^0}
     \| Z^{\flat}_t  \|_{B_{3, 3}^{1 / 4 - \delta}}] \int_0^T \|
     \mathbb{W}^2_t \|^2_{B_{6, \infty}^{- 1 - \delta}} \frac{\mathd
     t}{\langle t \rangle^{1 + \delta}} . \]
  \[ \lesssim \lambda^2 \sup_{t \leqslant T}  [\| Z^{\flat}_t  \|_{L^4} \|
     Z^{\flat}_t  \|_{H^{1 / 2 - \delta}}] \int_0^T \| \mathbb{W}^2_t
     \|^2_{B_{6, \infty}^{- 1 - \delta}} \frac{\mathd t}{\langle t \rangle^{1
     + \delta}} \]
  Which can be easily estimated by Young's inequality. From
  Proposition~\ref{adjointparaproduct} and Proposition~\ref{besovembedding}
  \[ \left| \frac{\lambda^2}{2} \bint ((J_t \mathbb{W}^2_t \succ
     Z^{\flat}_t))^2 - \frac{\lambda^2}{2} \bint (J_t \mathbb{W}^2_t \succ
     Z^{\flat}_t) \circ J_t \mathbb{W}^2_t Z^{\flat}_t \right| \lesssim
     \lambda^2 \| J_t \mathbb{W}^2_t \|_{B_{6, \infty}^{- 1 - \delta}}^2  \|
     Z^{\flat}_t \|_{B_{3, \infty}^{- 1 / 4 - \delta}}  \| Z^{\flat}_t
     \|_{B_{3, 3}^0} . \]
  and by interpolation
  \[ \lesssim \lambda^2 \| J_t \mathbb{W}^2_t \|_{B_{6, \infty}^{- 1 -
     \delta}}^2  \| Z^{\flat}_t  \|_{L^4} \| Z^{\flat}_t  \|_{H^{1 / 2 -
     \delta}} \]
  The integrability of this term in time follows from the inequality
  \[ \| J_t \mathbb{W}^2_t \|_{B_{6, \infty}^{- 1 - \delta}}^2 \lesssim
     \langle t \rangle^{- 1 - 2 \delta} \| \mathbb{W}^2_t \|_{B_{6, \infty}^{-
     1 - \delta}}^2 . \]
  To prove this last bound, recall that $t^{1 / 2} J_t$ is a Fourier
  multiplier with symbol
  \[ \langle k \rangle^{- 1} (- \hat{\rho}' (\langle k \rangle / t) \langle k
     \rangle / t)^{1 / 2} = \langle k \rangle^{- 1} \eta (\langle k \rangle /
     t), \]
  where $\eta$ is a smooth function supported in an annulus of radius $1$.
  Using this observation and applying Proposition~\ref{multiplierestimate}
  gives the estimate. Applying Proposition~\ref{commutatorestimate} and
  Proposition~\ref{besovembedding} we get
  \[ \lambda^2 \| (J_t \mathbb{W}^2_t \succ Z^{\flat}_t) \circ J_t
     \mathbb{W}^2_t - (J_t \mathbb{W}^2_t \circ J_t \mathbb{W}^2_t)
     (Z^{\flat}_t) \|_{B_{3 / 2, \infty}^0} \lesssim \lambda^2 \| J_t
     \mathbb{W}^2_t \|_{B_{6, \infty}^{- 1 - \delta}}^2  \| Z^{\flat}_t
     \|_{B_{3, \infty}^{3 \delta}} . \]
  and after using duality and interpolation we obtain
  \[ \frac{\lambda^2}{2} \left| \int^T_0 \bint ((J_t \mathbb{W}^2_t \succ
     Z^{\flat}_t))^2 - (J_t \mathbb{W}^2_t \circ J_t \mathbb{W}^2_t)
     (Z^{\flat}_t)^2 \mathd t \right| \]
  \[ \lesssim \lambda^2 \sup_{t \leqslant T}  [\| Z^{\flat}_t  \|_{L^4} \|
     Z^{\flat}_t  \|_{H^{1 / 2 - \delta}}] \int_0^T \| \mathbb{W}^2_t
     \|^2_{B_{6, \infty}^{- 1 - \delta}} \frac{\mathd t}{\langle t \rangle^{1
     + \delta}} \]
  \[ \lesssim \varepsilon \left( \frac{1}{2} \sup_{t \leqslant T} \|
     Z^{\flat}_t  \|_{H^{1 / 2 - \delta}}^2 + \lambda \| Z_T \|_{L^4}^4
     \right) + C (\varepsilon, \delta) \lambda^7 \left( \int_0^T \|
     \mathbb{W}^2_t \|^2_{B_{6, \infty}^{- 1 - \delta}} \frac{\mathd
     t}{\langle t \rangle^{1 + \delta}} \right)^4 \]
  \[ \lesssim \varepsilon \left( \frac{1}{2} \| Z_T \|_{H^{1 / 2 - \delta}}^2
     + \lambda \| Z_T \|_{L^4}^4 \right) + C (\varepsilon, \delta) \lambda^7
     \int_0^T \| \mathbb{W}^2_t \|^8_{B_{6, \infty}^{- 1 - \delta}}
     \frac{\mathd t}{\langle t \rangle^{1 + \delta}} . \]
  Finally we have
  \[ \lambda^2 \left| \int_0^T \bint \mathbb{W}^{\langle 2 \rangle \diamond
     \langle 2 \rangle}_t (Z^{\flat}_t)^2 \mathd t \right| \lesssim \lambda^2
     \left[ \int_0^T \| \mathbb{W}^{\langle 2 \rangle \diamond \langle 2
     \rangle}_t \|_{L^4} \mathd t \right] \| Z_T \|_{H^{\varepsilon}} \| Z_T
     \|_{L^4} \]
  \[ \leqslant C (\varepsilon) \lambda^7 \left[ \int_0^T \|\mathbb{W}^{\langle
     2 \rangle \diamond \langle 2 \rangle}_t \|_{L^4} \mathd t \right]^4 +
     \lambda \varepsilon \| Z_T \|_{L^4}^4 + \varepsilon \| Z_T \|_{H^{1 / 2 -
     \delta}}^2 . \]
  Using eq.~{\eqref{eq:bound-Z-K}} to control $\| Z_T \|_{H^{1 / 2 - \delta}}$
  in terms of $K_T$ we obtain the claim.
\end{proof}

\section{Stochastic estimates\label{sec:stochastic}}

In this section we close our argument proving the following lemmas which give
uniform estimates as $T \rightarrow \infty$ of the stochastic terms appearing
in our analytic estimates.

\begin{lemma}
  \label{lemma:renormalized-terms}For any $\varepsilon > 0$ and any $p > 1, r
  < \infty, q \in [1, \infty]$, there exists a constant $C (\varepsilon, p,
  q)$ which does not depend on $\Lambda$ such that
  \begin{equation}
    \sup_T \mathbb{E} [\| W_T \circ W_T^{[3]} \|_{B^{- \varepsilon}_{r, q}}^p]
    \leqslant C (\varepsilon, p, q) . \label{eq:stoch-est-w-w3}
  \end{equation}
  Moreover there exists a function $\gamma_t \in C^1 (\mathbb{R}_+,
  \mathbb{R})$ such that for any $\varepsilon > 0$ and any $p > 1$,
  \begin{equation}
    \sup_T \mathbb{E} [\| (\mathbb{W}^2_T \circ \mathbb{W}_T^{[3]} - 2
    \gamma_T W_T) \|_{B^{- 1 / 2 - \varepsilon}_{r, q}}^p] \leqslant C
    (\varepsilon, p, q), \label{eq:renorm-1}
  \end{equation}
  \begin{equation}
    \mathbb{E} \left[ \left( \int^{\infty}_0 \| J_t \mathbb{W}^2_t \circ J_t
    \mathbb{W}^2_t - 2 \dot{\gamma}_t \|_{B^{- \varepsilon}_{r, q}} \mathd t
    \right)^p \right] \leqslant C (\varepsilon, p, q) . \label{eq:renorm-2}
  \end{equation}
  \[ \sup_t \mathbb{E} [\| J_t \mathbb{W}^2_t \circ J_t \mathbb{W}^2_t - 2
     \dot{\gamma}_t \|_{B^{- \varepsilon}_{r, q}}] \leqslant C (\varepsilon,
     p, q) \]
  and
  \begin{equation}
    | \gamma_t | + \langle t \rangle | \dot{\gamma}_t | \lesssim 1 + \log
    \langle t \rangle, \qquad t \geqslant 0. \label{eq:bounds-gamma}
  \end{equation}
  Furthermore $\gamma$ is independent of $\Lambda$. By Besov embedding
  H{\"o}lder norms of these objects are also uniformly bounded in $T$ (but not
  uniformly in $\Lambda$).
\end{lemma}

\begin{proof}
  We will concentrate in proving the bounds on the renormalized terms in
  eq.~{\eqref{eq:renorm-1}} and {\eqref{eq:renorm-2}} and leave to the reader
  to fill the details for the easier term in eq.~{\eqref{eq:stoch-est-w-w3}}.
  Recall the representation of $W_t = Y_t$ in terms of the family of Brownian
  motions $(B_t^n)_{t, n}$ in eq.~{\eqref{eq:rep-Y}}. Wick's products of the
  Gaussian field $W_T$ can be represented as iterated stochastic integrals
  wrt. $(B_t^n)_{t, n}$. In particular, if we let $\mathd w_s (k) = \langle k
  \rangle^{- 1} \sigma_s (k) \mathd B_s^k$, we have
  \[ \mathbb{W}^2_T (x) = 12 \llbracket W_T^2 \rrbracket (x) = 24 \sum_{k_1,
     k_2} e^{i (k_1 + k_2) \cdot x} \int_0^T \int_0^{s_2} \mathd w_{s_1} (k_1)
     \mathd w_{s_2} (k_2), \]
  \[ \mathbb{W}^{[3]}_T (x) = 24 \sum_{k_1, k_2, k_3} e^{i k_{(123)} \cdot x}
     \int_0^T \int_0^{s_3} \int_0^{s_2} \left( \int_{s_3}^T \frac{\sigma_u^2
     (k_{(123)})}{\langle k_{(123)} \rangle^2} \mathd u \right) \mathd w_{s_1}
     (k_1) \mathd w_{s_2} (k_2) \mathd w_{s_3} (k_3), \]
  where $k_{(123)} \assign k_1 + k_2 + k_3$. Now products of iterated
  integrals can be decomposed in sums of iterated integrals and we get
  \begin{equation}
    \begin{array}{lll}
      \Delta_q (\mathbb{W}^{2 \diamond [3]}_T) (x) & = & \Delta_q
      (\mathbb{W}^2_T \circ \mathbb{W}^{[3]}_T - 2 \gamma_T W_T) (x)\\
      & = & \sum_{k_1, \ldots, k_5} \int_{A_T^5} G^{2 \diamond [3]}_{0, q}
      ((s, k)_{1 \cdots 5}) \mathd w_{s_1} (k_1) \cdots \mathd w_{s_5} (k_5)\\
      &  & + \sum_{k_1, \ldots, k_3} \int_{A_T^3} G^{2 \diamond [3]}_{1, q}
      ((s, k)_{1 \cdots 3}) \mathd w_{s_1} (k_1) \cdots \mathd w_{s_3} (k_3)\\
      &  & + \sum_{k_1} \int_{A_T^1} G^{2 \diamond [3]}_{2, q} ((s, k)_1)
      \mathd w_{s_1} (k_1),
    \end{array} \label{eq:W2o3}
  \end{equation}
  where $A^n_T \assign \{ 0 \leqslant s_1 < \cdots < s_n \leqslant T \}
  \subseteq [0, T]^n$ and where the deterministic kernels are given by
  \begin{eqnarray*}
    G^{2 \diamond [3]}_{0, q} ((s, k)_{1 \cdots 5}) & \assign & (24^2) K_q
    (k_{(1 \cdots 5)}) e^{i (k_{(1 \cdots 5)}) \cdot x} \sum_{\sigma \in
    \tmop{Sh} (2, 3)} \sum_{i \sim j} \times\\
    &  & \times K_i (k_{(\sigma_1 \sigma_2)}) K_j (k_{(\sigma_3 \sigma_4
    \sigma_5)}) \left( \int_{s_{\sigma_5}}^T \frac{\sigma_u (k_{(\sigma_3
    \sigma_4 \sigma_5)})^2}{\langle k_{(\sigma_3 \sigma_4 \sigma_5)}
    \rangle^2} \mathd u \right),\\
    G^{2 \diamond [3]}_{1, q} ((s, k)_{1 \cdots 3}) & \assign & (24^2) K_q
    (k_{(1 \cdots 3)}) e^{i (k_{(1 \cdots 3)}) \cdot x} \sum_{\sigma \in
    \tmop{Sh} (1, 2)} \sum_{i \sim j} \sum_p \int_0^T \mathd r \frac{\sigma_r
    (p)^2}{\langle p \rangle^2} \times\\
    &  & \times K_i (k_{\sigma_1} + q) K_j (k_{(\sigma_2 \sigma_3)} - q)
    \left( \int_{s_{\sigma_3} \vee r}^T \frac{\sigma_u (k_{(\sigma_2
    \sigma_3)} - p)^2}{\langle k_{(\sigma_2 \sigma_3)} - p \rangle^2} \mathd u
    \right),\\
    G^{2 \circ [3]}_{2, q} ((s, k)_1) & \assign & (24^2) K_q (k_1) e^{i k_1
    \cdot x} \sum_{i \sim j} \sum_{p_1, p_2} \int_0^T \mathd r_1 \int_0^T
    \mathd r_2 \frac{\sigma_{r_1} (p_1)^2}{\langle p_1 \rangle^2}
    \frac{\sigma_{r_2} (p_2)^2}{\langle p_2 \rangle^2} \times\\
    &  & \times K_i (p_1 + p_2) K_j (k_1 - p_1 - p_2) \left( \int_{r_1 \vee
    r_2 \vee s_1}^T \frac{\sigma_u (k_1 - p_1 - p_2)^2}{\langle k_1 - p_1 -
    p_2 \rangle^2} \mathd u \right),\\
    G^{2 \diamond [3]}_{2, q} ((s, k)_1) & \assign & G^{2 \circ [3]}_{2, q}
    ((s, k)_1) - 2 \gamma_T K_q (k_1) e^{i k_1 \cdot x},
  \end{eqnarray*}
  where $\tmop{Sh} (k, l)$ is the set of permutations of $\{ 1, \ldots, k + l
  \}$ keeping the orders $\sigma (1) < \cdots < \sigma (k)$ and $\sigma (k +
  1) < \cdots < \sigma (k + l)$ and where, for any symbol $z$, we denote with
  expression of the form $z_{1 \cdots n}$ the vector $(z_1, \ldots, z_n)$.
  Estimation of $\Delta_q (\mathbb{W}^2_T \circ \mathbb{W}^{[3]}_T) (x)$
  reduces then to estimate each of the three iterated integrals using BDG
  inequalities to get, for any $p \geqslant 2$,
  \[ I_{0, q} = \left\{ \mathbb{E} \left[ \left| \sum_{k_1, \ldots, k_5}
     \int_{A_T^5} G^{2 \diamond [3]}_{0, q} ((s, k)_{1 \cdots 5}) \mathd
     w_{s_1} (k_1) \cdots \mathd w_{s_5} (k_5) \right|^{2 p} \right]
     \right\}^{1 / p} \]
  \[ \lesssim \mathbb{E} \left[ \left| \sum_{k_1, \ldots, k_5} \int_{A_T^5}
     G^{2 \diamond [3]}_{0, q} ((s, k)_{1 \cdots 5}) \mathd w_{s_1} (k_1)
     \cdots \mathd w_{s_5} (k_5) \right|^2 \right] \]
  \[ \lesssim \sum_{k_1, \ldots, k_5} \int_{A_T^5} | G^{2 \diamond [3]}_{0, q}
     ((s, k)_{1 \cdots 5}) |^2 \frac{\sigma_{s_1} (k_1)^2}{\langle k_1
     \rangle^2} \cdots \frac{\sigma_{s_5} (k_5)^2}{\langle k_5 \rangle^2}
     \mathd s_1 \cdots \mathd s_5 . \]
  The kernel $G^{2 \diamond [3]}_{0, q} ((s, k)_{1 \cdots 5})$ being a
  symmetric function of its argument, we can simplify this expression into an
  integral over $[0, T]^5$:
  \[ \  \]
  \[ I_{0, q} \lesssim \sum_{k_1, \ldots, k_5} \int_{[0, T]^5} | G^{2 \diamond
     [3]}_{0, q} ((s, k)_{1 \cdots 5}) |^2 \frac{\sigma_{s_1} (k_1)^2}{\langle
     k_1 \rangle^2} \cdots \frac{\sigma_{s_5} (k_5)^2}{\langle k_5 \rangle^2}
     \mathd s_1 \cdots \mathd s_5 . \]
  Under the measure $\frac{\sigma_{s_5} (k_5)^2}{\langle k_5 \rangle^2} \mathd
  s_5$, we have
  \[ \left| \int_{s_{\sigma_5}}^T \frac{\sigma_u (k_{(\sigma_3 \sigma_4
     \sigma_5)})^2}{\langle k_{(\sigma_3 \sigma_4 \sigma_5)} \rangle^2} \mathd
     u \right| \lesssim \frac{1}{\langle k_{\sigma_5} \rangle^2} . \]
  Therefore with some standard estimates we can reduce us to consider
  \[ I_{0, q} \lesssim \sum_{k_1, \ldots, k_5} \int_{[0, T]^5} \frac{K_q
     (k_{(1 \cdots 5)})^2}{\langle k_5 \rangle^4} \mathbbm{1}_{k_{(12)} \sim
     k_{(345)}} \frac{\sigma_{s_1} (k_1)^2}{\langle k_1 \rangle^2} \cdots
     \frac{\sigma_{s_5} (k_5)^2}{\langle k_5 \rangle^2} \mathd s_1 \cdots
     \mathd s_5 \]
  \[ \lesssim \sum_{k_1, \ldots, k_5} \int_{[0, T]^5} \frac{K_q (k_{(1 \cdots
     5)})^2}{\langle k_5 \rangle^4} \mathbbm{1}_{k_{(12)} \sim k_{(345)}}
     \frac{\sigma_{s_1} (k_1)^2}{\langle k_1 \rangle^2} \cdots
     \frac{\sigma_{s_5} (k_5)^2}{\langle k_5 \rangle^2} \mathd s_1 \cdots
     \mathd s_5 \]
  \[ \lesssim \sum_{k_1, \ldots, k_5} \frac{K_q (k_{(1 \cdots 5)})^2}{\langle
     k_5 \rangle^4} \mathbbm{1}_{k_{(12)} \sim k_{(345)}} \frac{1}{\langle k_1
     \rangle^2} \cdots \frac{1}{\langle k_5 \rangle^2} \]
  \[ \lesssim \sum_{p_1, p_2} \mathbbm{1}_{p_1 \sim p_2} K_q (p_1 + p_2)^2
     \sum_{k_1, \ldots, k_5} \frac{1}{\langle k_5 \rangle^4}
     \mathbbm{1}_{k_{(12)} = p_1, k_{(345)} = p_2} \frac{1}{\langle k_1
     \rangle^2} \cdots \frac{1}{\langle k_5 \rangle^2} \]
  \[ \lesssim \sum_{p_1, p_2} \mathbbm{1}_{p_1 \sim p_2} K_q (p_1 + p_2)^2
     \frac{1}{\langle p_1 \rangle} \frac{1}{\langle p_2 \rangle^4} \lesssim
     \sum_{p_1, r} K_q (r)^2 \frac{1}{\langle p_1 \rangle} \frac{1}{\langle
     p_1 + r \rangle^4} \lesssim \sum_r K_q (r)^2 \frac{1}{\langle r
     \rangle^2} \lesssim 2^q . \]
  Now by similar reasoning we also have
  \[ | G^{2 \diamond [3]}_{1, q} ((s, k)_{1 \cdots 3}) | \lesssim \sum_{\sigma
     \in \tmop{Sh} (1, 2)} | K_q (k_{(1 \cdots 3)}) | \sum_{i \sim j} \sum_p
     \int_0^T \mathd r \frac{\sigma_r (p)^2 | K_i (k_{\sigma_1} + p) K_j
     (k_{(\sigma_2 \sigma_3)} - p) |}{\langle p \rangle^2 \langle k_{\sigma_1}
     + p \rangle^2} \]
  \[ \lesssim \sum_{\sigma \in \tmop{Sh} (1, 2)} \frac{| K_q (k_{(1 \cdots
     3)}) |}{\langle k_{\sigma_1} \rangle} \]
  so
  \[ I_{1, q} = \left\{ \mathbb{E} \left[ \left| \sum_{k_1, \ldots, k_3}
     \int_{A_T^3} G^{2 \diamond [3]}_{1, q} ((s, k)_{1 \cdots 5}) \mathd
     y_{s_1} (k_1) \cdots \mathd y_{s_3} (k_3) \right|^{2 p} \right]
     \right\}^{1 / p} \]
  \[ \lesssim \sum_{k_1, \ldots, k_3} \int_{[0, T]^5} \left| \sum_{\sigma \in
     \tmop{Sh} (1, 2)} \frac{| K_q (k_{(1 \cdots 3)}) |}{\langle k_{\sigma_1}
     \rangle} \right|^2 \frac{\sigma_{s_1} (k_1)^2}{\langle k_1 \rangle^2}
     \cdots \frac{\sigma_{s_3} (k_3)^2}{\langle k_3 \rangle^2} \mathd s_1
     \cdots \mathd s_3 \]
  \[ \lesssim \sum_{k_1, \ldots, k_3} \frac{| K_q (k_{(1 \cdots 3)})
     |^2}{\langle k_1 \rangle^4 \langle k_2 \rangle^2 \langle k_3 \rangle^2}
     \lesssim \sum_r \frac{K_q (r)^2}{\langle r \rangle^2} \lesssim 2^q . \]
  Finally, we note that the same strategy cannot be applied to the first
  chaos, since the kernel $G^{2 \diamond [3]}_{2, q}$ cannot be uniformly
  bounded. We let
  \[ \begin{array}{lll}
       A_T (s_1, k_1) & \assign & (24^2) \sum_{i \sim j} \sum_{q_1, q_2}
       \int_0^T \mathd r_1 \int_0^T \mathd r_2 \frac{\sigma_{r_1}
       (q_1)^2}{\langle q_1 \rangle^2} \frac{\sigma_{r_2} (q_2)^2}{\langle q_2
       \rangle^2} \times\\
       &  & \times K_i (q_1 + q_2) K_j (k_1 - q_1 - q_2) \left( \int_{r_1
       \vee r_2 \vee s_1}^T \frac{\sigma_u^2 (k_1 - q_1 - q_2)}{\langle k_1 -
       q_1 - q_2 \rangle^2} \mathd u \right),
     \end{array} \]
  so
  \[ G^{2 \diamond [3]}_{2, q} ((s, k)_1) = K_q (k_1) e^{i k_1 \cdot x} [A_T
     (s_1, k_1) - 2 \gamma_T] . \]
  Observe that
  \[ \begin{array}{lll}
       A_T (0, 0) & = & (12^2 \cdot 2) \sum_{q_1, q_2} \int_0^T \mathd r_1
       \int_0^T \mathd r_2 \frac{\sigma_{r_1} (q_1)^2}{\langle q_1 \rangle^2}
       \frac{\sigma_{r_2} (q_2)^2}{\langle q_2 \rangle^2} \times\\
       &  & \times \int_{r_1 \vee r_2}^T \frac{\sigma_u^2 (q_1 +
       q_2)}{\langle q_1 + q_2 \rangle^2} \mathd u \sum_{i \sim j} K_i (q_1 +
       q_2) K_j (- q_1 - q_2) .
     \end{array} \]
  We choose $\gamma_T$ as
  \begin{equation}
    \gamma_T = A_T (0, 0) = (12^2 \cdot 2) \sum_{q_1, q_2} \int_0^T \mathd u
    \int_0^u \mathd r_1 \int_0^u \mathd r_2 \frac{\sigma_{r_1}
    (q_1)^2}{\langle q_1 \rangle^2} \frac{\sigma_{r_2} (q_2)^2}{\langle q_2
    \rangle^2} \frac{\sigma_u^2 (q_1 + q_2)}{\langle q_1 + q_2 \rangle^2}
    \label{eq:choice-gamma}
  \end{equation}
  where we used the fact that for all $q \in \mathbb{R}^d$ we have $\sum_{i
  \sim j} K_i (q) K_j (q) = 1$, since $\int f \circ g = \int f g$. Note that,
  as claimed,
  \[ | \gamma_T | \lesssim \sum_{q_1, q_2} \frac{\mathbbm{1}_{| q_1 |, | q_2
     |, | q_1 + q_2 | \lesssim T}}{\langle q_1 \rangle^2 \langle q_2 \rangle^2
     \langle q_1 + q_2 \rangle^2} \lesssim 1 + \log \langle T \rangle . \]
  Now
  \[ A_T (s_1, k_1) - 2 \gamma_T = (24^2 \cdot 6) \sum_{q_1, q_2} \int_0^T
     \mathd r_1 \mathd r_2 \frac{\sigma_{r_1} (q_1)^2}{\langle q_1 \rangle^2}
     \frac{\sigma_{r_2} (q_2)^2}{\langle q_2 \rangle^2} \sum_{i \sim j} K_i
     (q_1 + q_2) \times \]
  \[ \times \left( K_j (k_1 - q_1 - q_2) \int_{s_1 \vee r_1 \vee r_2}^T
     \frac{\sigma_u^2 (k_1 - q_1 - q_2)}{\langle k_1 - q_1 - q_2 \rangle^2}
     \mathd u - K_j (q_1 + q_2) \int_{r_1 \vee r_2}^T \frac{\sigma_u^2 (q_1 +
     q_2)}{\langle q_1 + q_2 \rangle^2} \mathd u \right) \]
  so when $| q_1 + q_2 | \gg | k_1 |$ the quantity in round brackets can be
  estimated by $| k_1 | \langle q_1 + q_2 \rangle^{- 4}$ while when $| q_1 +
  q_2 | \lesssim | k_1 |$ it is estimated by $\langle q_1 + q_2 \rangle^{- 2}$
  so we have
  \[ | A_T (s_1, k_1) - \gamma_T | \lesssim \sum_{q_1, q_2} \frac{1}{\langle
     q_1 \rangle^2} \frac{1}{\langle q_2 \rangle^2} \frac{1}{\langle q_1 + q_2
     \rangle^2} \left( \mathbbm{1}_{| q_1 + q_2 | \lesssim | k_1 |} +
     \mathbbm{1}_{| q_1 + q_2 | \gtrsim | k_1 |} \frac{| k_1 |}{\langle q_1 +
     q_2 \rangle^2} \right) \]
  \[ \lesssim 1 + \log \langle k_1 \rangle . \]
  And then with this choice of $\gamma_T$ the kernel $\tilde{G}^{2 \circ
  [3]}_{2, q}$ stays uniformly bounded as $T \rightarrow \infty$ and satisfies
  \[ | G^{2 \diamond [3]}_{2, q} ((s, k)_1) | \lesssim K_q (k_1) \log \langle
     k_1 \rangle . \]
  From this we easily deduce that
  \[ I_{2, q} = \left\{ \mathbb{E} \left[ \left| \sum_{k_1} \int_{A_T} G^{2
     \diamond [3]}_{2, q} ((s, k)_{1 \cdots 5}) \mathd y_{s_1} (k_1)
     \right|^{2 p} \right] \right\}^{1 / p} \lesssim q 2^q, \qquad q \geqslant
     - 1. \]
  All together these estimates imply that
  \[ \mathbb{E} \| \Delta_q \mathbb{W}^{2 \diamond [3]}_T \|_{L^{2 p}}^{2 p}
     \lesssim (q 2^{q / 2})^{2 p}, \qquad q \geqslant - 1. \]
  Standard argument allows to deduce eq.~{\eqref{eq:renorm-1}}. The analysis
  of the other renormalized product proceeds similarly. Let
  \[ V (t) \assign \mathbb{W}^{\langle 2 \rangle \diamond \langle 2 \rangle}_t
     = J_t \mathbb{W}^2_t \circ J_t \mathbb{W}^2_t - 2 \dot{\gamma}_t, \qquad
     t \geqslant 0. \]
  First note that by definition of Besov spaces we have
  \[ \mathbb{E} \left[ \left( \int^{\infty}_0 \| V (t) \|_{B^{- \varepsilon -
     d / r}_{r, r}} \mathd t \right)^p \right] \]
  \[ \lesssim \mathbb{E} \left[ \left( \int^{\infty}_0 \left( \sum_q 2^{- q r
     (\varepsilon + d / r)} \| \Delta_q V (t) \|_{L^r} \right)^{1 / r} \mathd
     t \right)^p \right] . \]
  By Minkowski's integral inequality this is bounded by
  \[ \lesssim \left( \int^{\infty}_0 \mathd t \left\{ \mathbb{E} \left[ \left(
     \sum_q 2^{- q r (\varepsilon + d / r)} \| \Delta_q V (t) \|_{L^r}^r
     \right)^{p / r} \right] \right\}^{1 / p} \right)^p . \]
  When $r \geqslant p$ Jensen's inequality and Fubini's theorem give
  \[ \lesssim \left( \int^{\infty}_0 \mathd t \left\{ \sum_q 2^{- q r
     (\varepsilon + d / r)} \int_{\Lambda} \frac{\mathd x}{| \Lambda |}
     \mathbb{E} [| \Delta_q V (t) (x) |^r] \right\}^{1 / r} \right)^p . \]
  Finally hypercontractivity and stationarity allow to reduce this to bound
  \[ \lesssim \left( \int^{\infty}_0 \mathd t \left\{ \sum_q 2^{- q r
     (\varepsilon + d / r)}  (\mathbb{E} [| \Delta_q V (t) (0) |^2])^{r / 2}
     \right\}^{1 / r} \right)^p . \]
  Letting $I_q (t) =\mathbb{E} [| \Delta_q V (t) (0) |_{}^2]$ we have
  \[ \mathbb{E} \left[ \left( \int^{\infty}_0 \|\mathbb{W}^{\langle 2 \rangle
     \diamond \langle 2 \rangle}_t \|_{B^{- \varepsilon - d / r}_{r, r}}
     \mathd t \right)^p \right] \lesssim \left( \int^{\infty}_0 \mathd t
     \left\{ \sum_q 2^{- q r (\varepsilon + d / r)}  (I_q (t))^{r / 2}
     \right\}^{1 / r} \right)^p . \]
  Now we decompose the random field $\Delta_q (\mathbb{W}^{\langle 2 \rangle
  \diamond \langle 2 \rangle}_t) (x)$ into homogeneous stochastic integral as
  above and obtain
  \begin{equation}
    \begin{array}{lll}
      \Delta_q (\mathbb{W}^{\langle 2 \rangle \diamond \langle 2 \rangle}_t)
      (x) & = & \sum_{k_1, \ldots, k_4} \int_{A_t^4} G^{\langle 2 \rangle
      \circ \langle 2 \rangle}_{0, q} ((s, k)_{1 \cdots 4}) \mathd w_{s_1}
      (k_1) \cdots \mathd w_{s_4} (k_4)\\
      &  & + \sum_{k_1, k_2} \int_{A_t^2} G^{\langle 2 \rangle \circ \langle
      2 \rangle}_{1, q} ((s, k)_{1 2}) \mathd w_{s_1} (k_1) \mathd w_{s_2}
      (k_2)\\
      &  & + G^{J 2 \circ J 2}_{2, q}
    \end{array} \label{eq:W2o2}
  \end{equation}
  with
  \begin{eqnarray}
    G^{\langle 2 \rangle \diamond \langle 2 \rangle}_{0, q} ((s, k)_{1 \cdots
    4}) & = & (24^2) K_q (k_{(1 \cdots 4)}) e^{i (k_{(1 \cdots 4)}) \cdot x}
    \times \nonumber\\
    &  & \times \sum_{\sigma \in \tmop{Sh} (2, 2)} \sum_{i \sim j} K_i
    (k_{(\sigma_1 \sigma_2)}) K_j (k_{(\sigma_3 \sigma_4)}) \frac{\sigma_t
    (k_{(\sigma_1 \sigma_2)})}{\langle k_{(\sigma_1 \sigma_2)} \rangle}
    \frac{\sigma_t (k_{(\sigma_3 \sigma_4)})}{\langle k_{(\sigma_3 \sigma_4)}
    \rangle} \nonumber\\
    G^{\langle 2 \rangle \diamond \langle 2 \rangle}_{1, q} ((s, k)_{12}) & =
    & (24^2) K_q (k_{(12)}) e^{i (k_{(12)}) \cdot x} \sum_{\sigma \in
    \tmop{Sh} (1, 1)} \sum_{i \sim j} \sum_q \times \nonumber\\
    &  & \times \int_0^t \mathd r \frac{\sigma_r^2 (q)}{\langle q \rangle^2}
    K_i (k_{\sigma_1} + q) K_j (k_{\sigma_2} - q) \left( \frac{\sigma_t
    (k_{\sigma_1} + q)}{\langle k_{\sigma_1} + q \rangle} \frac{\sigma_t
    (k_{\sigma_2} - q)}{\langle k_{\sigma_2} - q \rangle} \right) \nonumber\\
    G^{\langle 2 \rangle \diamond \langle 2 \rangle}_{2, q} & = & (24^2)
    \mathbbm{1}_{q = - 1}  \sum_{i \sim j} \sum_{q_1, q_2} \int_0^t \mathd r_1
    \int_0^t \mathd r_2 \times \nonumber\\
    &  & \times \frac{\sigma_{r_1} (q_1)^2}{\langle q_1 \rangle^2}
    \frac{\sigma_{r_2} (q_2)^2}{\langle q_2 \rangle^2} K_i (q_1 + q_2) K_j (-
    q_1 - q_2) \frac{\sigma_t (q_1 + q_2)^2}{\langle q_1 + q_2 \rangle^2}
    \nonumber\\
    &  & - 2 \dot{\gamma}_t \mathbbm{1}_{q = - 1} . \nonumber
  \end{eqnarray}
  Using our choice of $\gamma_T$ in eq.~{\eqref{eq:choice-gamma}} we have that
  \[ \dot{\gamma}_t = (12^2 \cdot 2) \sum_{q_1, q_2} \int_0^t \mathd r_1
     \int_0^t \mathd r_2 \frac{\sigma_{r_1} (q_1)^2}{\langle q_1 \rangle^2}
     \frac{\sigma_{r_2} (q_2)^2}{\langle q_2 \rangle^2} \frac{\sigma_t^2 (q_1
     + q_2)}{\langle q_1 + q_2 \rangle^2} \mathd u, \]
  which implies also that
  \[ G^{\langle 2 \rangle \diamond \langle 2 \rangle}_{2, q} = 0, \quad
     \text{and} \qquad | \dot{\gamma}_t | \lesssim \frac{1 + \log \langle t
     \rangle}{\langle t \rangle} . \]
  as claimed. We pass now to estimate the other two chaoses. The technique is
  the same we used above. Consider first
  \[ I_{0, q} (t) \assign \mathbb{E} \left[ \left| \sum_{k_1, \ldots, k_4}
     \int_{A_t^4} G^{\langle 2 \rangle \diamond \langle 2 \rangle}_{0, q} ((s,
     k)_{1 \cdots 4}) \mathd w_{s_1} (k_1) \cdots \mathd w_{s_4} (k_4)
     \right|^2 \right] \]
  \[ \lesssim \sum_{k_1, \ldots, k_4} \int_{A_t^4} | G^{\langle 2 \rangle
     \diamond \langle 2 \rangle}_{0, q} ((s, k)_{1 \cdots 4}) |^2
     \frac{\sigma_{s_1} (k_1)^2}{\langle k_1 \rangle^2} \cdots
     \frac{\sigma_{s_4} (k_4)^2}{\langle k_4 \rangle^2} \mathd s_1 \cdots
     \mathd s_4 \]
  \[ \lesssim \sum_{k_1, \ldots, k_4} \int_{[0, t]^4} | G^{\langle 2 \rangle
     \diamond \langle 2 \rangle}_{0, q} ((s, k)_{1 \cdots 4}) |^2
     \frac{\sigma_{s_1} (k_1)^2}{\langle k_1 \rangle^2} \cdots
     \frac{\sigma_{s_4} (k_4)^2}{\langle k_4 \rangle^2} \mathd s_1 \cdots
     \mathd s_4 \]
  \[ \lesssim \sum_{k_1, \ldots, k_4} K_q (k_{(1 \cdots 4)})^2 \int_{[0, t]^4}
     \frac{\sigma_t^2 (k_{(12)})}{\langle k_{(12)} \rangle^2} \frac{\sigma_t^2
     (k_{(34)})}{\langle k_{(34)} \rangle^2} \frac{\sigma_{s_1}
     (k_1)^2}{\langle k_1 \rangle^2} \cdots \frac{\sigma_{s_4}
     (k_4)^2}{\langle k_4 \rangle^2} \mathd s_1 \cdots \mathd s_4 \]
  \[ \lesssim \sum_{k_1, \ldots, k_4} K_q (k_{(1 \cdots 4)})^2
     \frac{\sigma_t^2 (k_{(12)})}{\langle k_{(12)} \rangle^2} \frac{\sigma_t^2
     (k_{(34)})}{\langle k_{(34)} \rangle^2} \frac{1}{\langle k_1 \rangle^2}
     \cdots \frac{1}{\langle k_4 \rangle^2} \]
  \[ \lesssim \frac{\mathbbm{1}_{2^q \lesssim t}}{\langle t \rangle^6}
     \sum_{k_1, \ldots, k_4} \frac{K_q (k_{(1 \cdots 4)})^2}{\langle k_1
     \rangle^2 \langle k_2 \rangle^2 \langle k_3 \rangle^2 \langle k_4
     \rangle^2} \lesssim \frac{\mathbbm{1}_{2^q \lesssim t}}{\langle t
     \rangle^6} 2^{4 q} \]
  where we used that $| \sigma_t (x) | \lesssim t^{- 1 / 2} \mathbbm{1}_{x
  \sim t}$. Now taking $\varepsilon + d / r > 0$ we have
  \[ \int^{\infty}_0 \mathd t \left\{ \sum_q 2^{- q r (\varepsilon + d / r)} 
     (I_{0, q} (t))^{r / 2} \right\}^{1 / r} \lesssim \int^{\infty}_0 \mathd t
     \left\{ \sum_{q : 2^q \lesssim t}  \frac{2^{q r (2 - \varepsilon - d /
     r)}}{\langle t \rangle^{3 r}} \right\}^{1 / r} \]
  \[ \lesssim \int^{\infty}_0 \frac{\mathd t}{\langle t \rangle^{1 +
     \varepsilon + d / r}} \lesssim 1. \]
  Taking into account that $| k_1 |, | k_2 | \lesssim t$ we can estimate

  \[ | G^{\langle 2 \rangle \diamond \langle 2 \rangle}_{1, q} ((s, k)_{12}) |
     \lesssim | K_q (k_{(12)}) | \sum_p \frac{\mathbbm{1}_{| p | \lesssim
     t}}{\langle p \rangle^2} \left( \frac{\sigma_t (k_1 + p)}{\langle k_1 + p
     \rangle} \frac{\sigma_t (k_2 - p)}{\langle k_2 - p \rangle} \right)
     \lesssim | K_q (k_{(12)}) | \langle t \rangle^{- 2}, \]
  from which we deduce that
  \[ I_{1, q} (t) \assign \mathbb{E} \left[ \left| \sum_{k_1, k_2}
     \int_{A_t^2} G^{\langle 2 \rangle \diamond \langle 2 \rangle}_{1, q} ((s,
     k)_{12}) \mathd w_{s_1} (k_1) \mathd w_{s_2} (k_2) \right|^2 \right] \]
  \[ \lesssim \sum_{k_1, k_2} \int_{A_t^2} | G^{\langle 2 \rangle \diamond
     \langle 2 \rangle}_{1, q} ((s, k)_{12}) |^2 \frac{\sigma_{s_1}
     (k_1)^2}{\langle k_1 \rangle^2} \frac{\sigma_{s_2} (k_2)^2}{\langle k_2
     \rangle^2} \mathd s_1 \mathd s_2 \]
  \[ \lesssim \langle t \rangle^{- 4} \sum_{k_1, k_2} | K_q (k_{(12)}) |^2
     \int_{[0, t]^2} \frac{\sigma_{s_1} (k_1)^2}{\langle k_1 \rangle^2}
     \frac{\sigma_{s_2} (k_2)^2}{\langle k_2 \rangle^2} \mathd s_1 \mathd s_2
  \]
  \[ \lesssim \langle t \rangle^{- 4} \sum_{k_1, k_2} | K_q (k_{(12)}) |^2
     \frac{\mathbbm{1}_{k_1 \lesssim t}}{\langle k_1 \rangle^2}
     \frac{\mathbbm{1}_{k_2 \lesssim t}}{\langle k_2 \rangle^2} \lesssim
     \langle t \rangle^{- 4} 2^{2 q} \mathbbm{1}_{2^q \lesssim t}, \]
  and then, as for $I_{0, q}$, we have

  \[ \int^{\infty}_0 \mathd t \left( \sum_q 2^{- q r (\varepsilon + d / r)} 
     (I_{1, q} (t))^{r / 2} \right)^{1 / r} \lesssim \int^{\infty}_0
     \frac{\mathd t}{\langle t \rangle^2} \left( \sum_q 2^{q r (1 -
     \varepsilon - d / r)} \mathbbm{1}_{2^q \lesssim t} \right)^{1 / r}
     \lesssim 1, \]
  as claimed. From these estimates standard arguments give
  eq.~{\eqref{eq:renorm-2}}.
\end{proof}

\begin{lemma}
  \[ \mathbb{E} [\| \mathbb{W}^3_T \|^p_{L^p}]^{1 / p} \lesssim T^{3 / 2} \]
  This implies that $\mathbb{W}^{\langle 3 \rangle} \in C ([0, \infty], B_{p,
  p}^{- 1 / 2 - \kappa}) \cap L^2 (\mathbb{R}_+, B_{p, p}^{- 1 / 2 - \kappa})$
  for any $p < \infty$ uniformly in the volume and $\mathbb{W}^{\langle 3
  \rangle} \in C \left( [0, \infty], \VV^{- 1 / 2 - \kappa} \right) \cap L^2
  \left( \mathbb{R}_+, \VV^{- 1 / 2 - \kappa} \right)$
\end{lemma}

Estimating the cube
\[ \mathbb{W}^3_T (x) = 12 \llbracket W_T^3 \rrbracket (x) = 24 \sum_{k_1,
   k_2, k_3} e^{i (k_1 + k_2 + k_3) \cdot x} \int_0^T \int_0^{s_2} \mathd
   w_{s_1} (k_1) \mathd w_{s_2} (k_2) \mathd w_{s_3} (k_3) \]
We get for any $p$, by space homogeneity,
\[ \mathbb{E} [\| \mathbb{W}^3_T (x) \|^{2 p}_{L^p}]^{1 / p} =\mathbb{E}
   \left[ \left| \sum_{k_1, k_2, k_3} \int_0^T \int_0^{s_2} \int^{s_1}_0
   \mathd w_{s_1} (k_1) \mathd w_{s_2} (k_2) \mathd w_{s_3} (k_3) \right|^{2
   p} \right]^{1 / p} \]
\begin{eqnarray*}
  \mathbb{E} [\| \mathbb{W}^3_T (x) \|^p_{L^p}] & = & \mathbb{E} \left[ \left|
  \sum_{k_1, k_2, k_3} \int_0^T \int_0^{s_2} \int^{s_1}_0 \mathd w_{s_1} (k_1)
  \mathd w_{s_2} (k_2) \mathd w_{s_3} (k_3) \right|^p \right]^{1 / p}\\
  & \lesssim & \mathbb{E} \left[ \left| \sum_{k_1, k_2, k_3} \int_0^T
  \int_0^{s_2} \int^{s_1}_0 \mathd w_{s_1} (k_1) \mathd w_{s_2} (k_2) \mathd
  w_{s_3} (k_3) \right|^2 \right]\\
  & = & \sum_{k_1, k_2, k_3} \int_0^T \int_0^{s_2} \int^{s_1}_0
  \frac{\sigma_{s_1} (k_1)^2}{\langle k_1 \rangle^2} \cdots \frac{\sigma_{s_3}
  (k_3)^2}{\langle k_3 \rangle^2} \mathd s_1 \cdots \mathd s_3
\end{eqnarray*}
Now
\begin{eqnarray*}
  &  & \sum_{k_1, k_2, k_3} \int_0^T \int_0^{s_2} \int^{s_1}_0
  \frac{\sigma_{s_1} (k_1)^2}{\langle k_1 \rangle^2} \cdots \frac{\sigma_{s_3}
  (k_3)^2}{\langle k_3 \rangle^2} \mathd s_1 \cdots \mathd s_3\\
  & \leqslant & \sum_{k_1, k_2, k_3} \int_0^T \int_0^T \int^T_0
  \frac{\sigma_{s_1} (k_1)^2}{\langle k_1 \rangle^2} \cdots \frac{\sigma_{s_3}
  (k_3)^2}{\langle k_3 \rangle^2} \mathd s_1 \cdots \mathd s_3\\
  & = & \left( \sum_k \int^T_0 \frac{\sigma_s (k)^2}{\langle k \rangle^2}
  \mathd s \right)^3\\
  & \lesssim & T^3
\end{eqnarray*}
Now the second properties follow by the fact that $\sigma_t$ is supported in
an annulus of radius $t$, so
\[ \| \mathbb{W}_t^{\langle 3 \rangle} \|_{B_{p, p}^{- 1 / 2 - \kappa}} =
   \left\| \frac{\sigma_t (D)}{\langle D \rangle} \mathbb{W}_t^3
   \right\|_{B_{p, p}^{- 1 / 2 - \kappa}} \lesssim \| \sigma_t (D)
   \mathbb{W}_t^3 \|_{B_{p, p}^{- 3 / 2 - \kappa}} \lesssim \langle t
   \rangle^{- 1 / 2 - \kappa} \langle t \rangle^{- 3 / 2} \| \mathbb{W}_t^3
   \|_{L^p} \]
and Hoelder estimates follow by Besov-embedding.

\

\

\

\appendix\section{Besov spaces and paraproducts}\label{sec:appendix-para}

In this section we will recall some well known results about Besov spaces,
embeddings, Fourier multipliers and paraproducts. The reader can find full
details and proofs
in~{\cite{bahouri_fourier_2011,gubinelli_paracontrolled_2015}} and for
weighted spaces in ~{\cite{gubinelli_global_2018,mourrat_plane_2015}}. First
recall the definition of Littlewood--Paley blocks. Let $\chi, \varphi$ be
smooth radial functions $\mathbb{R}^d \rightarrow \mathbb{R}$ such that
\begin{itemize}
  \item $\tmop{supp} \chi \subseteq B (0, R)$, $\tmop{supp} \varphi \subseteq
  B (0, 2 R) \setminus B (0, R)$;
  
  \item $0 \leqslant \chi, \varphi \leqslant 1$, $\chi (\xi) + \sum_{j \geq 0}
  \varphi (2^{- j} \xi) = 1$ for any $\xi \in \mathbb{R}^d$;
  
  \item $\tmop{supp} \varphi (2^{- j} \cdot) \cap \tmop{supp} \varphi (2^{- i}
  \cdot) = \varnothing$ if $| i - j | > 1$.
\end{itemize}
Introduce the notations $\varphi_{- 1} = \chi$, $\varphi_j = \varphi (2^{- j}
\cdot)$ for $j \geqslant 0$. For any $f \in \CS' (\Lambda)$ we define the
operators $\Delta_j f \assign \mathcal{F}^{- 1} \varphi_j (\xi) \hat{f}
(\xi)$, $j \geqslant - 1$.

\begin{definition}
  We say a function $\rho : \mathbb{R}^d \rightarrow \mathbb{R}$ of the form
  $\rho (x) = \langle x \rangle^{- \sigma}$ for $\sigma \geqslant 0$ is a
  weight.
\end{definition}

\begin{definition}
  For a weight $\rho$ let
  \[ \| f \|_{L^p (\rho)} = \left( \int_{\mathbb{R}^d} | f (x) |^p \rho (x)
     \mathd x \right)^{1 / p} \]
  and by $L^p (\rho)$ the space of functions for which this norm is finite.
  For function defined on a torus in $\mathbb{R}^d$ we consider their periodic
  extensions on $\mathbb{R}^d$.
\end{definition}

\begin{definition}
  Let $s \in \mathbb{R}, p, q \in [1, \infty]$, and $\rho$ be a weight. For a
  Schwarz distribution $f \in \CS' (\Lambda)$ define the norm
  \[ \| f \|_{B_{p, q}^s (\rho)} = \| (2^{j s} \| \Delta_j f \|_{L^p
     (\rho)})_{j \geqslant - 1} \|_{\ell^q} . \]
  Then the space $B^s_{p, q}$(r is the set of functions in $\CS' (\Lambda)$
  such that this norm is finite. We denote $B_{p, q}^s = B_{p, q}^s (1 / |
  \Lambda |)$ the normalized Besov space and $H^s = B^s_{2, 2}$ the Sobolev
  spaces, and by $\VV^s = B_{\infty, \infty}^s$ the (unweighted) Hoelder
  spaces.
\end{definition}

\begin{definition}
  Let $s \in \mathbb{R}$ and $\rho$ be a weight. Then we denote by
  \[ \| f \|_{\VV^s (\rho)} = \| (2^{j s} \| \rho \Delta_j f
     \|_{L^{\infty}})_{j \geqslant - 1} \|_{\ell^{\infty}} \]
  and by $\VV^s (\rho)$ the space of Schwarz distributions such that this norm
  is finite. 
\end{definition}

\begin{proposition}
  \label{besovembedding}Let $\delta > 0$.We have for any $q_1, q_2 \in [1,
  \infty], q_1 < q_2$
  \[ \| f \|_{B_{p, q_2}^s} \leq \| f \|_{B_{p, q_1}^s} \leq \| f \|_{B_{p,
     \infty}^{s + \delta}} \]
  Furthermore, if we denote by $W^{s, p}$ the normalized fractional Sobolev
  spaces then for any $q \in [1, \infty]$
  \[ \| f \|_{B_{p, q}^s} \leq \| f \|_{W^{s + \delta, p}} \leq \| f \|_{B_{p,
     \infty}^{s + 2 \delta}} \]
\end{proposition}

\begin{proposition}
  \label{compactembedding}For any $s_1, s_2 \in \mathbb{R}$ such that $s_1 <
  s_2$, any $p, q \in [1, \infty]$ the Besov space $B^{s_1}_{p, q_{}}$ is
  compactly embedded into $B^{s_2}_{p, q}$.
\end{proposition}

\begin{definition}
  Let $f, g \in \mathcal{\CS} (\Lambda)$. We define the paraproducts and
  resonant product
  \[ f \succ g = g \prec f \assign \sum_{j < i - 1} \Delta_i f \Delta_j g,
     \qquad \text{and} \qquad f \circ g \assign \sum_{| i - j | \leqslant 1}
     \Delta_i f \Delta_j g. \]
  Then
  \[ fg = f \prec g + f \circ g + f \succ g. \]
\end{definition}

\begin{proposition}
  \label{paraproductestimate}Let $f, g \in \mathcal{\CS} (\Lambda)$. We define
  the paraproducts and resonant product by
  \[ f \succ g = g \prec f \assign \sum_{j < i - 1} \Delta_i f \Delta_j g,
     \qquad \text{and} \qquad f \circ g \assign \sum_{| i - j | \leqslant 1}
     \Delta_i f \Delta_j g. \]
  Then
  \[ fg = f \prec g + f \circ g + f \succ g. \]
  Moreover for any weight $\rho$, $\beta \leqslant 0, \alpha \in \mathbb{R}$
  and $p_1, p_2 \in [1, \infty]$, $\frac{1}{p_1} + \frac{1}{p_2} =
  \frac{1}{p}$ we have the estimates
  \[ \begin{array}{lll}
       \| f \succ g \|_{B^{\alpha + \beta}_{p, q} (\rho)} & \lesssim & \| f
       \|_{B^{\alpha}_{p_1, \infty} (\rho)} \| g \|_{B^{\beta}_{p_2, q}
       (\rho)},
     \end{array} \]
  and for any $\alpha, \beta \in \mathbb{R}$ such that $\alpha + \beta > 0$
  the estimates
  \[ \begin{array}{lll}
       \| f \circ g \|_{B^{\alpha + \beta}_{p, q} (\rho)} & \lesssim & \| f
       \|_{B^{\alpha}_{p_1, \infty} (\rho)} \| g \|_{B^{\beta}_{p_2, q}
       (\rho)},
     \end{array} \]
\end{proposition}

For a proof see Theorem~3.17 and Remark~3.18 in~{\cite{mourrat_plane_2015}}.

\begin{proposition}
  For any weight $\rho, \beta \leqslant 0, \alpha \in \mathbb{R}$ we have
  \[ \begin{array}{lll}
       \| f \succ g \|_{B^{\alpha + \beta}_{p, q} (\rho^2)} & \lesssim & \| f
       \|_{\VV^{\alpha} (\rho)} \| g \|_{B^{\beta}_{p, q} (\rho)},
     \end{array} \]
\end{proposition}

The proof is an easy modification of the proof of Theorem 3.17
in~{\cite{mourrat_plane_2015}}.

\begin{proposition}
  \label{commutatorestimate}Let $\alpha \in (0, 1)$ $\beta, \gamma \in
  \mathbb{R}$ such that $\beta + \gamma < 0$, $\alpha + \beta + \gamma > 0$
  and $p_1, p_2, p_3, p \in [1, \infty]$ such that $\frac{1}{p_1} +
  \frac{1}{p_2} + \frac{1}{p_3} = \frac{1}{p}$. Then there exists a bounded
  trilinear form $\mathfrak{K}_1 (f, g, h)$ such that for any $\delta > 0$,
  \[ \| \mathfrak{K}_1 (f, g, h) \|_{B_{p_{}, \infty}^{\alpha + \beta +
     \gamma}} \lesssim \| g \|_{B_{p_1, \infty}^{\alpha}} \| f \|_{B_{p_2,
     \infty}^{\beta}} \| h \|_{B_{p_3, \infty}^{\gamma}}, \]
  and when $f, g, h \in \CS$ we have
  \[ \mathfrak{K}_1 (f, g, h) = (f \succ g) \circ h - g (f \circ h) . \]
\end{proposition}

\begin{proof}
  The proof is a slight modification of the one given in
  {\cite{gubinelli_paracontrolled_2015}} Lemma 2.97 from
  {\cite{bahouri_fourier_2011}} and interpolation implies that $\| \Delta_j f
  g - \Delta_j (f g) \|_{L^p} \leq 2^{- j \alpha} \| f \|_{W^{\alpha, p_1}} \|
  g \|_{L^{p_2}}$. This in turn gives after some algebraic computations(see
  {\cite{gubinelli_paracontrolled_2015}})that
  \[ \Delta_j (f \succ g) = (\Delta_j f) \succ g + R_j (f, g) \]
  with $\| R_j (f, g) \|_{L^p} \lesssim 2^{- j (\alpha + \beta)} \| f
  \|_{B_{p_1, \infty}^{\alpha}} \| g \|_{B_{p_2, \infty}^{\beta}}$. Now to
  prove the statement of the proposition observe that for smooth $f, g, h$ we
  have
  \[ \mathfrak{K}_1 (f, g, h) = \sum_{j, k \geqslant - 1} \sum_{| i - j |
     \leqslant 1} \Delta_j (f \succ \Delta_k g) \Delta_i h - \Delta_k g
     \Delta_j f_i h \]
  Now observe that the term $f \succ \Delta_k g$ has Fourier transform outside
  of $2^k B$ for some Ball $B$ independent of $k$, so choosing $N$ large
  enough we can rewrite the sum as
  \[ \mathfrak{K}_1 (f, g, h) = \sum_{j, k \geqslant - 1} \sum_{| i - j |
     \leqslant 1} \mathbbm{1}_{k \leqslant i + N} (\Delta_j f \Delta_k g
     \Delta_i h + R_j (f, \Delta_k g)) - \Delta_k g \Delta_j f_i h \]
  \[ \sum_{j, k \geqslant - 1} \sum_{| i - j | \leqslant 1} \mathbbm{1}_{k
     \leqslant i + N} R_j (f, \Delta_k g) _i h - \mathbbm{1}_{k \geqslant i +
     N} \Delta_k g \Delta_j f_i h \]
  Now we estimate the norm of the two terms separately. First note that for
  fixed $j$
  \[ \sum_{k \geqslant - 1} \sum_{| i - j | \leqslant 1} \mathbbm{1}_{k
     \leqslant i + N} R_j (f, \Delta_k g) \]
  has a Fourier transform supported in $2^j B$. By Lemma 2.69 from
  {\cite{bahouri_fourier_2011}} it is enough to get an estimate on  $\sup_k
  \left\| 2^{(\alpha + \beta + \gamma) j} \sum_{j \geqslant - 1} \sum_{| i - j
  | \leqslant 1} \mathbbm{1}_{k \leqslant i + N} R_j (f, \Delta_k g) _i h
  \right\|_{L^p}$ to estimate it in $B_{p_{}, \infty}^{\alpha + \beta +
  \gamma}$, so by H{\"o}lder inequality
  \begin{eqnarray*}
    \left\| \sum_{| i - j | \leqslant 1} R_j \left( f, \sum^{i + N}_{k
    \geqslant - 1} \Delta_k g \right) _i h \right\|_{L^p} & \lesssim & \sum_{|
    i - j | \leqslant 1} 2^{- j (\alpha + \beta)} 2^{- i \gamma} \| g
    \|_{B_{p_1, \infty}^{\alpha}} \| f \|_{B_{p_2, q_1}^{\beta}} \| h
    \|_{B_{p_3, q_2}^{\gamma}}\\
    & \lesssim & 2^{- j (\alpha + \beta + \gamma)} \| g \|_{B_{p_1,
    \infty}^{\alpha}} \| f \|_{B_{p_2, q_1}^{\beta}} \| h \|_{B_{p_3,
    q_2}^{\gamma}}
  \end{eqnarray*}
  For the second term observe that for fixed $k$ the Fourier transform of
  \[ \sum_{j \geqslant - 1} \sum_{| i - j | \leqslant 1} \mathbbm{1}_{k
     \geqslant i + N} \Delta_k g \Delta_j f_i h \]
  is supported in $2^k B$. Now we can estimate again by H{\"o}lder inequality
  \begin{eqnarray*}
    & \lesssim & \left\| \sum_{j \geqslant - 1} \sum_{| i - j | \leqslant 1}
    \mathbbm{1}_{k \geqslant i + N} \Delta_k g \Delta_j f_i h \right\|_{L^p}\\
    & \lesssim & 2^{- \alpha k} \sum^{k + N}_{j \geqslant - 1} 2^{- (\beta +
    \gamma) k} \mathbbm{1}_{k \geqslant i + N} \| g \|_{B_{p_1,
    \infty}^{\alpha}} \| f \|_{B_{p_2, \infty}^{\beta}} \| h \|_{B_{p_1,
    \infty}^{\gamma}}\\
    & \lesssim & 2^{- j (\alpha + \beta + \gamma)} \| g \|_{B_{p_1,
    \infty}^{\alpha}} \| f \|_{B_{p_2, q_1}^{\beta}} \| h \|_{B_{p_3,
    q_2}^{\gamma}}
  \end{eqnarray*}
  \[ \  \]
\end{proof}

\begin{proposition}
  \label{adjointparaproduct}Assume $\alpha \in (0, 1)$, $\beta, \gamma \in
  \mathbb{R}$ such that $\beta + \gamma < 0$, and $\alpha + \beta + \gamma =
  0$ , $\frac{1}{p_1} + \frac{1}{p_2} + \frac{1}{p_3} = 1$ and $\frac{1}{q_1}
  + \frac{1}{q_2} = 1$. Then there exists a bounded trilinear form
  $\mathfrak{K}_2 (f, g, h)$ for which
  \[ | \mathfrak{K}_2 (f, g, h) | \lesssim \| f \|_{B_{p_1, \infty}^{\alpha}}
     \| g \|_{B_{p_2, q_1}^{\beta}} \| h \|_{B_{p_2, q_2}^{\gamma}} \]
  and
  \[ \mathfrak{K}_2 (f, g, h) = \frac{1}{| \Lambda |} \int_{\Lambda} [(f \succ
     g) h - (f \circ h) g] \]
  for smooth functions.
\end{proposition}

\begin{proof}
  This is modification of the proof of Lemma A.6 in
  {\cite{gubinelli_hamiltonian_2018}}. Repeating an algebraic computation
  given in {\cite{gubinelli_hamiltonian_2018}} in the proof of Lemma A.6, we
  get that for smooth $f, g, h$ we have
  \[ \mathfrak{K}_2 (f, g, h) = \left( \sum_{i \geqslant k - 1, | j - k |
     \leqslant L} - \sum_{i \sim k, 1 < | j - k | \leqslant L} \right) \langle
     \Delta_i g, \Delta_j h \Delta_k f \rangle \]
  Then we estimate
  \begin{eqnarray*}
    | \mathfrak{K}_2 (f, g, h) | & \lesssim & \left( \sum_{i \geqslant k - 1,
    | j - k | \leqslant L} - \sum_{i \sim k, 1 < | j - k | \leqslant L}
    \right) \langle \Delta_i g, \Delta_j h \Delta_k f \rangle\\
    & \lesssim & \sum_{i \geqslant k, j \sim k} | \langle \Delta_i g,
    \Delta_j h \Delta_k f \rangle |\\
    & \lesssim & \sum_{i \geqslant k, j \sim k} \| \Delta_k f \|_{L^{p_1}} \|
    \Delta_i g \|_{L^{p_2}} \| \Delta_j h \|_{L^{p_3}}\\
    & \lesssim & \sup_k  \left( 2^{\alpha k} \| \Delta_k f \|_{L^{p_1}}
    \right) \sum_k \sum_{i \geqslant k, j \sim k} 2^{(\beta + \gamma) k} \|
    \Delta_i g \|_{L^{p_2}} \| \Delta_j h \|_{L^{p_3}}\\
    & \lesssim & \| f \|_{B_{p_1, \infty}^{\alpha}} \| g \|_{B_{p_2,
    q_1}^{\beta}} \| h \|_{B_{p_2, q_2}^{\gamma}}
  \end{eqnarray*}
\end{proof}

\begin{proposition}
  \label{prop:squarecomm}There exists a family $(\mathfrak{K}_{3, t})_{t
  \geqslant 0}$ of bounded multilinear forms on $\VV^{- 1 - \kappa} \times
  \VV^{- 1 - \kappa} \times H^{1 / 2 - \delta} \times H^{1 / 2 - \delta}$ such
  that for smooth $\varphi, \psi, g^{(1)}, g^{(2)}$ it holds
  \[ \mathfrak{K}_{3, t} (\varphi, \psi \comma g^{(1)}, g^{(2)}) = \bint [J_t
     (\varphi \succ g^{(1)}) J_t (\psi \succ g^{(2)}) - (J_t \varphi \circ J_t
     \psi) g^{(1)} g^{(2)}], \]
  and
  \[ | \mathfrak{K}_{3, t} (\varphi, \psi, g^{(1)}, g^{(2)}) | \lesssim
     \frac{1}{\langle t \rangle^{1 + \delta}} \| \varphi \|_{\VV^{- 1 -
     \kappa}} \| \psi \|_{\VV^{- 1 - \kappa}} \| g^{(1)} \|_{H^{1 / 2 -
     \delta}} \| g^{(2)} \|_{H^{1 / 2 - \delta}}, \]
  for some $\delta > 0$.
\end{proposition}

\tmcolor{black}{\begin{proof}
  Note that $\langle t \rangle^{1 / 2} J_t$ satisfies the assumptions of
  Proposition~\ref{paraproductleibniz} and with $m = - 1$, therefore using
  also Proposition \ref{besovembedding}
  \[ \| J_t (\varphi \succ g^{(1)}) - J_t \varphi \succ g^{(1)} \|_{H^{1 / 2 -
     2 \delta - \kappa}} \lesssim \langle t \rangle^{- 1 / 2} \| \varphi
     \|_{\VV^{- 1 - \kappa}} \| g^{(1)} \|_{H^{1 / 2 - \delta}} \]
  and therefore
  \begin{eqnarray*}
    &  & \left| \bint [J_t (\varphi \succ g^{(1)}) - (J_t \varphi \succ
    g^{(1)})] J_t (\psi \succ g^{(2)}) \right|\\
    & \lesssim & \| J_t (\varphi \succ g^{(1)}) - J_t \varphi \succ g^{(1)}
    \|_{H^{1 / 2 - 2 \delta - \kappa}} \| J_t (\psi \succ g^{(2)}) \|_{H^{- 1
    / 2 + 2 \delta + \kappa}}\\
    & \lesssim & \langle t \rangle^{- 1 / 2} \| \varphi \|_{\VV^{- 1 -
    \kappa}} \| g^{(1)} \|_{H^{1 / 2 - \delta}} \langle t \rangle^{- 1 / 2 -
    \delta} \| \psi \|_{\VV^{- 1 - \kappa}} \| g^{(2)} \|_{H^{1 / 2 - \delta}}
  \end{eqnarray*}
  and by symmetry also
  \begin{eqnarray*}
    &  & \left| \bint [J_t (\varphi \succ g^{(1)}) J_t (\psi \succ g^{(2)}) -
    (J_t \varphi \succ g^{(1)}) (J_t \psi \succ g^{(2)})] \right|\\
    & \lesssim & \langle t \rangle^{- 1 - \delta} \| \varphi \|_{\VV^{- 1 -
    \kappa}} \| g^{(1)} \|_{H^{1 / 2 - \delta}} \| \psi \|_{\VV^{- 1 -
    \kappa}} \| g^{(2)} \|_{H^{1 / 2 - \delta}}
  \end{eqnarray*}
  Furthermore from Proposition~\ref{adjointparaproduct} and for sufficiently
  small $\kappa, \delta$
  \begin{eqnarray*}
    &  & \left| \bint (J_t \varphi \succ g^{(1)}) (J_t \psi \succ g^{(2)}) -
    \bint ((J_t \varphi \succ g^{(1)}) \circ J_t \psi) g^{(1)}_t \right|\\
    & \lesssim & \| J_t \varphi \|_{\VV^{- \kappa - \delta}}  \| g^{(1)}
    \|_{H^{1 / 2 - \delta}}  \| J_t \psi \|_{\VV^{- \kappa}} \| g^{(2)}
    \|_{H^{1 / 2 - \delta}}\\
    & \lesssim & \langle t \rangle^{- 1 - \delta} \| \varphi \|_{\VV^{- 1 -
    \kappa}}  \| g^{(1)} \|_{H^{1 / 2 - \delta}}  \| \psi \|_{\VV^{- 1 -
    \kappa}} \| g^{(2)} \|_{H^{1 / 2 - \delta}}
  \end{eqnarray*}
  and applying Proposition-\ref{commutatorestimate}
  \begin{eqnarray*}
    &  & \| (J_t \varphi^{(1)} \succ g^{(1)}) \circ J_t \psi_t - (J_t
    \varphi_t \circ J_t \psi_t) (g^{(1)}) \|_{H^{- 1 / 2 + \delta}}\\
    & \lesssim & \| J_t \varphi_t \|_{\VV^{- \kappa - \delta}}  \| g^{(1)}
    \|_{H^{1 / 2 - \delta}}  \| J_t \psi_t \|_{\VV^{- \kappa}}\\
    & \lesssim & \langle t \rangle^{- 1 - \delta} \| \varphi \|_{\VV^{- 1 -
    \kappa}}  \| g^{(1)} \|_{H^{1 / 2 - \delta}}  \| \psi \|_{\VV^{- 1 -
    \kappa}}
  \end{eqnarray*}
  and putting things together gives the estimate.
  
  \ 
\end{proof}}

\begin{definition}
  A smooth function $\eta$ is said to be an $S^m$ multiplier if for every
  multiindex $\alpha$ there exists a constant $C_{\alpha}$ such that
  \begin{equation}
    \label{symbolinq} \quad \left| \frac{\partial^{\alpha}}{\partial
    \xi^{\alpha}} f (\xi) \right| \lesssim_{\alpha} (1 + | \xi |)^{m - |
    \alpha |}, \qquad \xi \in \mathbb{R}^d .
  \end{equation}
  We say that a family $\eta_t$ is a uniformly $S^m$ multiplier
  if~{\eqref{symbolinq}} is satisfied for every $t$ with $C_{\alpha}$
  independent of $t$. 
\end{definition}

\begin{proposition}
  \label{multiplierestimate}Let $\eta$ be an $S^m$ multiplier, $s \in
  \mathbb{R}$, $p, q \in [1, \infty]$, and $f \in B_{p, q}^s (\mathbb{T}^d)$,
  then
  \[ \| \eta (\mathD) f \|_{B_{p, q}^{s - m}} \lesssim \| f \|_{B_{p, q}^s} .
  \]
  Furthermore the constant depends only on $s, p, q, d$ and the constants
  $C_{\alpha}$ in {\eqref{symbolinq}}.
\end{proposition}

For a proof see {\cite{bahouri_fourier_2011}} Lemma 2.78.

\begin{proposition}
  \label{paraproductleibniz}Assume $m \leqslant 0$, $\alpha \in (0, 1), \beta
  \in \mathbb{R}$. Let $\eta$ be an $S^m$ multiplier and $q, p_1, p_2 \in [1,
  \infty]$, $\frac{1}{p_1} + \frac{1}{p_2} = \frac{1}{p}$, $f \in B_{p_1,
  \infty}^{\beta}$, $g \in B_{p_1, \infty}^{\alpha}$. Then for any $\delta >
  0$.
  \[ \| \eta (\mathD) (f \succ g) - (\eta (\mathD) f \succ g) \|_{B_{p,
     q}^{\alpha + \beta - m - \delta}} \lesssim \| f \|_{B_{p_1,
     \infty}^{\beta}} \| g \|_{B_{p_1, \infty}^{\alpha}} . \]
  The constant depends only on $\alpha, \beta, \delta$ and the constants in
  {\eqref{symbolinq}}.
\end{proposition}

For a proof see {\cite{bahouri_fourier_2011}} Lemma 2.99.

\begin{proposition}
  \label{gagliardo-nirenberg}Let $\theta$ $p, p_1, p_2$ and $s, s_1, s_2$ be
  such that $\frac{1}{p} = \frac{\theta}{p_1} + \frac{1 - \theta}{p_2}$ and $s
  = \theta s_1 + (1 - \theta) s_2$ and assume that $f \in W^{s_1, p_1} \cap
  W^{s_2, p_2}$. Then
  \[ \| f \|_{W^{s, p}} \leqslant \| f \|^{\theta}_{W^{s_1, p_1}} \| f \|^{1
     - \theta}_{W^{s_2, p_2}} . \]
\end{proposition}

For a proof see {\cite{Brezis_2017}}.

\

\end{document}